\documentclass[11pt]{article}
\usepackage[english]{babel}
\usepackage{xcolor}
\usepackage{amsmath}
\usepackage{amsthm}
\usepackage{amsfonts}
\usepackage{amssymb}
\usepackage{euscript}
\usepackage[cp1251]{inputenc}
\usepackage[pdftex]{graphicx}
\usepackage [autostyle, english = american]{csquotes}
\usepackage{float}
\usepackage[pdfpagelabels,bookmarksdepth=3]{hyperref}
\hypersetup{colorlinks=true,citecolor=black,urlcolor=black,linkcolor=black}

\newtheorem{theorem}{Theorem}[section]
\newtheorem{lemma}[theorem]{Lemma}

\newtheorem{corollary}[theorem]{Corollary}

\theoremstyle{definition}
\newtheorem{definition}[theorem]{Definition}

\newtheorem{conjecture}[theorem]{Conjecture}

\theoremstyle{remark}

\numberwithin{equation}{section}

\begin{document}

\begin{flushleft}
\textbf{\large{Zoo of monotone Lagrangians in $\mathbb{C}P^n$}}
\end{flushleft}
\begin{flushleft}
\textbf{Vardan Oganesyan}\\
\end{flushleft}

\textbf{Abstract.} Let $P \subset \mathbb{R}^m$ be a polytope of dimension $m$ with $n$ facets and $a_1,\ldots, a_{n}$ be the normal vectors to the facets of $P$. Assume that $P$ is Delzant, Fano, and $a_1 + \ldots + a_{n} = 0$. We associate a monotone embedded Lagrangian $L \subset \mathbb{C}P^{n-1}$ to $P$.  As an abstract manifold, the Lagrangian $L$ fibers over $(S^1)^{n-m-1}$ with fiber $\mathcal{R}_P$, where $\mathcal{R}_P$ is defined by a system of quadrics in $\mathbb{R}P^{n-1}$. The manifold $\mathcal{R}_P$ is called the real toric space.

We find an effective method for computing the Lagrangian quantum cohomology  groups of the mentioned Lagrangians. Then we construct explicitly some set of wide and narrow Lagrangians.

Our method yields many different monotone Lagrangians with rich topological properties, including non-trivial Massey products, complicated fundamental group and complicated singular cohomology ring.

Interestingly, not only the methods of toric topology can be used to construct monotone Lagrangians, but the converse is also true: the symplectic topology of Lagrangians can be used to study the topology of $\mathcal{R}_P$. General formulas for the rings $H^{*}(\mathcal{R}_P, \mathbb{Z})$, $H^{*}(\mathcal{R}_P, \mathbb{Z}_2)$ are not known.  Since we have a method for constructing narrow Lagrangians, the spectral sequence of Oh can be used to study the singular cohomology ring of $\mathcal{R}_P$. This idea will be developed in a further paper.

\tableofcontents
\newpage

\section{Introduction}

Let $L$ be a Lagrangian submanifold in a symplectic manifold $M$. Lagrangians have two classical homomorphisms defined on $\pi_2(M,L)$:
\begin{equation*}
\omega: \pi_2(M, L) \rightarrow \mathbb{R}, \quad \mu: \pi_2(M, L) \rightarrow \mathbb{Z},
\end{equation*}
where $\omega$ computes the symplectic area of a disc and $\mu$ is the Maslov index. The Lagrangian submanifold $L$ is called \emph{monotone} if there is a constant $\lambda > 0$ such that $\omega = \lambda \mu$. Monotone Lagrangian submanifolds are of special interest due to their prominent role in Floer theory.

Let us assume that $M$  is a complex projective space $\mathbb{C}P^n$ endowed with its standard Kahler symplectic structure, and $L \subset \mathbb{C}P^{n}$ is a monotone embedded Lagrangian submanifold. Denote by $N_L$ its minimal Maslov number. Not much is known about how restrictive is the monotonicity condition. The most simple example of a monotone Lagrangian in $\mathbb{C}P^n$ is $\mathbb{R}P^n$. Its minimal Maslov number is equal to $n+1$. It is proved that $n+1$ is the maximal value for $N_L$, i.e the minimal Maslov number of any embedded monotone Lagrangian $L \subset \mathbb{C}P^n$ satisfies $1 \leqslant N_L \leqslant n+1$~\cite{Seidel,bircan1}. Moreover, if $N_L = n+1$, then  $L$ is homotopy equivalent to $\mathbb{R}P^n$~\cite{jacksmith1}.

Let us assume that $L \subset \mathbb{C}P^n$ is embedded, monotone and satisfies $2H_1(L, \mathbb{Z})=0$, i.e. $\forall \alpha \in H_1(L, \mathbb{Z}),$ $2\alpha = 0$. We see that $\mathbb{R}P^n$ is an example of such $L$. It turns out that the cohomology ring $H^{*}(L, \mathbb{Z}_2)$  must be isomorphic to $H^{*}(\mathbb{R}P^n, \mathbb{Z}_2)$ and $N_L = n+1$ \cite{bircan1,birciel,Seidel}. Combining all these facts we get

\begin{center}
$2H_1(L, \mathbb{Z}) = 0 \quad \iff \quad L$ is homotopy equivalent to $\mathbb{R}P^n$.
\end{center}

\noindent
It is not known if $L$ with the property above is homeomorphic to $\mathbb{R}P^n$.

If $3H_1(L, \mathbb{Z}) = 0$ and  $n=8$, then $H^{*}(L, \mathbb{Z}_2) = H^{*}(SU(3)/\mathbb{Z}_3, \mathbb{Z}_2)$ (see Iriyeh ~\cite{Hiroshi}). There are similar formulas for $n=5$ and $n=26$  (see~\cite{Hiroshi}). Moreover, Iriyeh constructed examples with property $3H_1(L, \mathbb{Z}) = 0$.

The Lagrangian quantum cohomology $QH^{*}(L, \mathbb{Z}_2[T, T^{-1}])$ is extremely difficult to compute. There is no general procedure to even decide whether or not it vanishes. If $QH^{*}(L, \mathbb{Z}_2[T, T^{-1}]) = 0$, then $L$ is called \emph{narrow}. If there exists an isomorphism $QH^{*}(L, \mathbb{Z}_2[T, T^{-1}]) \cong H^{*}(L, \mathbb{Z}_2[T, T^{-1}])$, then we call $L$ \emph{wide}.

In ~\cite{Usher}, Oakley and Usher constructed monotone Lagrangians
\begin{equation*}
\begin{gathered}
\frac{S^1 \times S^k \times S^m}{(e^{i\varphi},x,y) \sim (-e^{i\varphi},-x,-y) \sim (e^{i\varphi},-x,-y)} = \\
=\mathbb{R}P^1 \times \frac{(S^k \times S^m)}{(x,y) \sim (-x,-y)} \subset \mathbb{C}P^{k+m+1}
\end{gathered}
\end{equation*}
Moreover, they proved that their examples are displaceable if $k+m \geqslant 3$, therefore quantum cohomology vanishes. These are the only known displaceable monotone Lagrangians in $\mathbb{C}P^n$. There are also many examples of monotone tori in $\mathbb{C}P^n$ distinct up to Hamiltonian isotopy~\cite{Vianna1, Vianna2,Chek2,Cho}.
\\

Biran and Cornea conjectured in~\cite{bircan2} that monotone Lagrangians $L \subset \mathbb{C}P^n$ are either wide or narrow.
\\
- Assume that $L$ is wide. It is known that the Lagrangian quantum cohomology of $L \subset \mathbb{C}P^n$ is $2-$periodic (see \cite{Seidel, bircan1}), i.e. $QH^{*}(L, \mathbb{Z}_2[T, T^{-1}]) = QH^{*+2}(L, \mathbb{Z}_2[T, T^{-1}])$. So, we get strong restrictions on the singular cohomology groups of $L$.
\\
- Assume that $L$ is narrow. Then the spectral sequence of Oh (see \cite{Oh, buhovsky}) converges to zero. This puts restrictions on the singular cohomology ring and the minimal Maslov number.

Hamiltonian-minimal Lagrangian submanifolds in toric manifolds are studied in~\cite{MironovCn, Mirpan, kotel}. In particular, a Hamiltonian-minimal Lagrangian submanifold $\widetilde{L} \subset \mathbb{C}^n$ is associated to each Delzant polytope $P$.

In \cite{Og, Og2}, we show that the Lagrangian $\widetilde{L} \subset \mathbb{C}^n$ associated to a Delzant polytope $P$ is monotone if and only if the polytope is Fano. Moreover, the Maslov class of $\widetilde{L}$ can be computed explicitly. Using that we constructed new examples of monotone Lagrangians in $\mathbb{C}^n$ with interesting properties. In this paper we expand on these ideas. We propose a very effective method for constructing monotone Lagrangians of $\mathbb{C}P^n$ and computing their Floer homology.

\begin{theorem}\label{pollagrther}(see Section $\ref{mainconstructionmain}$ for the proof)
Let $P \subset \mathbb{R}^m$ be a polytope of dimension $m$ with $n$ facets and $a_1,\ldots, a_{n}$ be the normal vectors to the facets of $P$. Assume that $P$ is Delzant, Fano, and $a_1 + \ldots + a_{n} = 0$.

We associate a monotone embedded Lagrangian $L \subset \mathbb{C}P^{n-1}$ to the polytope $P$.  As an abstract manifold, the Lagrangian $L$ fibers over $(n-m-1)-$torus $T^{n-m-1}$ with fiber $\widetilde{\mathcal{R}}_P/\mathbb{Z}_2$, where $\widetilde{\mathcal{R}}_P$ is the so-called real moment-angle manifold. Moreover, $\widetilde{\mathcal{R}}_P/\mathbb{Z}_2$ can be defined by a system of quadrics in $\mathbb{R}P^{n-1}$.
\end{theorem}

\begin{lemma}\label{helplemma}(see section $\ref{somelemmas}$ for the proof)
Assume that a real moment-angle manifold $\widetilde{\mathcal{R}}_P$ is a complex moment-angle manifold $\widetilde{\mathcal{Z}}_Q$ of some polytope $Q$ (see  Section $\ref{torsiondefinitions}$). In other words, $P = Q^{(2,\ldots,2)}$ and $\widetilde{\mathcal{R}}_P = \widetilde{\mathcal{Z}}_Q$ ( see Section $\ref{wedge}$). Then
\item[-] The polytope $Q$ is $\frac{m}{2}-$simensional and has $\frac{n}{2}$ facets;
\item[-] The Lagrangian  $L$ is diffeomorphic to $\widetilde{\mathcal{R}}_P/\mathbb{Z}_2 \times T^{n-m-1}$.  Moreover, $L$ is orientable and the action of $\mathbb{Z}_2$ on $\widetilde{\mathcal{R}}_P$ is isotopic to identity;
\item[-] $L$ is spin if and only if $\frac{n}{2}$ is even.
\end{lemma}

For example, if $P$ is the standard $n-$simplex, then $L = \mathbb{R}P^{n-1}$. In general, real moment-angle manifolds define a rich family of smooth manifolds and their topology is far from being completely understood (see \cite{Lopez,  real2}). The manifold $\widetilde{\mathcal{R}}_P/\mathbb{Z}_2$ is called the real toric space. In fact, $\widetilde{\mathcal{R}}_P/\mathbb{Z}_2$ can be defined by a system of quadrics in $\mathbb{R}P^{n-1}$. Unfortunately, general formulas for cohomology rings $H^{*}(\widetilde{\mathcal{R}}_P/\mathbb{Z}_2, \mathbb{Z})$, $H^{*}(\widetilde{\mathcal{R}}_P/\mathbb{Z}_2, \mathbb{Z}_2)$ are not known. There are formulas only for $H^{*}(\widetilde{\mathcal{R}}_P/\mathbb{Z}_2, G)$, where $G$ is a ring and $2$ is invertible in $G$ (see \cite{real2}).

Since the action of $\mathbb{Z}_2$ is free on $\widetilde{\mathcal{R}}_P$ and the Eilenberg-MacLane space $K(\mathbb{Z}_2, 1) = \mathbb{R}P^{\infty}$, we have the Cartan-Leray spectral sequence (see \cite[p. 206]{Weibel})
\begin{equation*}
E_2^{p,q} = H^{p}(\mathbb{R}P^{\infty}, H^{q}(\widetilde{\mathcal{R}}_P, \mathbb{Z}_2)) \Rightarrow H^{*}(\widetilde{\mathcal{R}}_P/\mathbb{Z}_2, \mathbb{Z}_2),
\end{equation*}
Denote the differential of $E_k(\widetilde{\mathcal{R}}_P/\mathbb{Z}_2, \mathbb{Z}_2)$ by $\partial_k$.

Let $Q$ be a polytope and $v_1, \ldots, v_n$ be its facets. Define
\begin{equation*}
m(Q) = min\{k \in \mathbb{N} \; | \; v_{i_1}\cap \ldots \cap v_{i_k} = \emptyset  \},
\end{equation*}
where $v_{i_1},\ldots,v_{i_r}$ are arbitrary facets.

The following theorem says that the quantum cohomology groups of some of the mentioned Lagrangians vanish.

\begin{theorem}\label{vanishquantum}
Assume that a polytope $P$ satisfies the conditions of Theorem $\ref{pollagrther}$. Let  $L \subset \mathbb{C}P^{n-1}$ be the embedded monotone Lagrangian associated to $P$. Assume that $\widetilde{\mathcal{R}}_P = \widetilde{\mathcal{Z}}_Q$ for some polytope $Q$, i.e. $\widetilde{\mathcal{R}}_P$ is the complex moment-angle manifold associated to $Q$.  Suppose that $Q$ is irredundant and $N_L > 2$. We have:
\item[-] If $\partial_{2m(Q) }(E_{2m(Q) }^{0, 2m(Q) - 1}) = 0 $, then $QH^{*}(L, \mathbb{Z}_2[T, T^{-1}]) = 0$;
\item[-] If $L$ is spin, then $QH(L, \mathbb{Z}[T, T^{-1}])$ is not isomorphic to $H^{*}(L, \mathbb{Z}[T, T^{-1}])$;
\item[-] Let $G$ be a ring such that $2$ is invertible in $G$. If $L$ is spin, then $QH(L, G[T, T^{-1}]) = 0$;
\item[-] $H^{q}(\widetilde{\mathcal{R}}_P, \mathbb{Z}) = 0 \;\;$  for  $\;0< q <2m(Q)-1$.
\end{theorem}

\textbf{Applications to algebraic topology.}(see Section $\ref{sympalg}$ for more details). Let  $\widetilde{\mathcal{R}}_P/\mathbb{Z}_2$ be a real toric space. For example, $\widetilde{\mathcal{R}}_P/\mathbb{Z}_2$ can be defined by a system of real quadrics in $\mathbb{R}P^{n-1}$.  Sometimes, we can find the additive structure of $H^{*}(\widetilde{\mathcal{R}}_P/\mathbb{Z}_2, \mathbb{Z}_2)$ using the Cartan-Leray spectral sequnce, but it is not clear how to find its cohomology ring. By Theorem $\ref{pollagrther}$, we have embedded monotone Lagrangian  $L \subset \mathbb{C}P^n$ associated to $\widetilde{\mathcal{R}}_P/\mathbb{Z}_2$. This means that $QH^{*}(L, \mathbb{Z}_2[T, T^{-1}])$ can be considered as an extra structure on $\widetilde{\mathcal{R}}_P/\mathbb{Z}_2$. It follows from Theorem $\ref{vanishquantum}$ that $QH(L, \mathbb{Z}_2[T, T^{-1}]) = 0$ for some real toric spaces and we get that the spectral sequence of Oh converges to zero. This puts some restrictions on the singular cohomology ring. So, information obtained from the spectral sequence of Oh can be used to find the ring structure of $H^{*}(\widetilde{\mathcal{R}}_P/\mathbb{Z}_2, \mathbb{Z}_2)$. In Section $\ref{sympalg}$ we consider nontrivial example, where we show how this new method works. This idea will be developed in a further paper.
\\
\\
\textbf{Remark.} The goal of all the theorems below is to show how Theorems $\ref{pollagrther}$ and $\ref{vanishquantum}$ work. We show that the theorems above yields a rich set of Lagrangians.
\\
\\
First, let us start with simple examples. The theorem below shows that our method provides some wide Lagrangians

\begin{theorem}\label{twospheretheorem}
Let $p, n$ be integers greater than $1$. Assume that $L$ is a total space of a fibration (the fibration is constructed in the proof)
\begin{equation*}
L \xrightarrow{(S^{p-1} \times S^{n-p-1})/\mathbb{Z}_2} S^1, \quad (x,y) \sim (-x-y), \;\; \; (x,y) \in S^{p-1} \times S^{n-p-1},
\end{equation*}
i.e. $L$ fibers over $S^1$ with fiber $(S^{p-1} \times S^{n-p-1})/\mathbb{Z}_2$.  We have:
\item[-] There exists a monotone Lagrangian embedding of $L$ into $\mathbb{C}P^{n-1}$ with minimal Maslov number $N_L = gcd(p,n)$, where $gcd$ stands for the greatest common divisor;
\item[-] If $p$ is odd and $n =2p$, then $QH^{*}(L, \mathbb{Z}_2[T, T^{-1}]) \cong H^{*}(L, \mathbb{Z}_2[T, T^{-1}])$. Hence, $L$ is not displaceable;
\item[-] The fibration over $S^1$ is trivial if and only if $p,n$ are even numbers, i.e. $L = S^1 \times (S^{p-1} \times S^{n-p-1})/\mathbb{Z}_2$ if and only if $p,n$ are even;
\item[-] If $n = 0$ mod $4$, then $L$ is spin;
\item[-] If $p,n$ are even, $L$ is spin, and $2$ is invertible in a ring $G$, then $QH^{*}(L, G[T, T^{-1}]) = 0$;
\end{theorem}

\vspace{.14in}

Unfortunately, we do not know the quantum cohomology groups of all Lagrangians $L$ (unless $n=2p$, $p$ is odd) from the theorem above. There is a rich group $\text{Symp}(\mathbb{C}P^{n-1}, L)$ and methods of Smith developed in~\cite{jacksmith1} may help to find $QH(L, \mathbb{Z}[T, T^{-1}])$.

Theorem $\ref{twospheretheorem}$ says that $QH^{*}(L, \mathbb{Z}_2[T, T^{-1}]) \cong H^{*}(L, \mathbb{Z}_2[T, T^{-1}])$ whenever $n=2p$ and $p$ is odd. We show in the proof  that $H^{*}(L, \mathbb{Z}_2)$ contains some extra relations (there are less relations for $n\neq 2p$) and this implies that the spectral sequence of Oh collapses at $E_1$. This shows how some relations in the singular cohomology ring may not allow a Lagrangian to be narrow.

Let us jump to the next level and consider more complicated Lagrangians.
\begin{theorem}\label{wide-narrow}
Let $k$ be an arbitrary positive integer. Assume that
\begin{equation*}
L = \widetilde{M}/\mathbb{Z}_2 \times T^3,
\end{equation*}
where $\widetilde{M}$ is a manifold (constructed in the proof) of dimension $12k-4 $  such that $H^{*}(\widetilde{M}, \mathbb{Z})$ and $H^{*}(\#9(S^{4k-1} \times S^{8k-3}) \#8(S^{6k-2} \times S^{6k-2}), \; \mathbb{Z})$ are isomorphic.  We have:
\item[-] There exists a monotone Lagrangian embedding of $L$ into $\mathbb{C}P^{12k-1}$. The minimal Maslov number $N_{L} = 2k$;
\item[-] If $k\geqslant 2$, then $QH^{*}(L, \mathbb{Z}_2[T, T^{-1}]) =  0$;
\item[-] The Lagrangian $L$ is spin. If $2$ is invertible in a ring $G$ and  $k \geqslant 2$, then $QH^{*}(L, G[T, T^{-1}]) =  0$.
\end{theorem}

\vspace{.07in}

Let us construct an example with more complicated topology.  Assume that $X$ is a manifold of dimension $n$ without boundary and write $X_{-}$ for $X$ minus a small $n-$disk. Then, we define
\begin{equation*}
\mathcal{G}(X) = \partial(X_{-} \times D^2),
\end{equation*}
where $D^2$ is 2-disk. The operation $\mathcal{G}X$ was defined by Gonzalez Acuna in ~\cite{Gonzales}. Let us mention the following interesting property:
\begin{equation*}
\begin{gathered}
\mathcal{G}(X \# Y) = \mathcal{G}(X)\#\mathcal{G}(Y),
\end{gathered}
\end{equation*}
where $X, Y$ are connected manifolds.

\vspace{.09in}

\begin{theorem}\label{wide-narrow2}
Let $k$ be an arbitrary positive integer and
\begin{equation*}
K = \mathcal{G}(\mathcal{G}(S^3 \times S^3 \times S^3)) \#7(S^3 \times S^8) \#12(S^4 \times S^7) \#10(S^5 \times S^6).
\end{equation*}
Assume that
\begin{equation*}
L = \widetilde{M}/\mathbb{Z}_2 \times T^4,
\end{equation*}
where $\widetilde{M}$ is a manifold (constructed in the proof) of dimension $20k - 5$  such that $H^{*}(K, \mathbb{Z})$ and $H^{*}(\widetilde{M}, \mathbb{Z})$ are isomorphic as ungraded rings. We have:
\item[-] There exists a monotone Lagrangian embedding of $L$ into $\mathbb{C}P^{20k - 1}$. The minimal Maslov number $N_{L} = 2k$;
\item[-] If $k\geqslant 2$, then $QH(L, \mathbb{Z}_2[T, T^{-1}]) = 0$;
\item[-] The Lagrangian $L$ is spin. If $2$ is invertible in a ring $G$ and  $k \geqslant 2$, then $QH^{*}(L, G[T, T^{-1}]) =  0$.
\end{theorem}

\vspace{.05in}

Below we construct Lagrangians with nontrivial Massey product. Let us note that existence of nontrivial Massey products reflects a complexity of the topology of a Lagrangian.  This shows that our method provides Lagrangians with rich topology.

\begin{theorem}\label{masseylagr}
Let $k$ be an arbitrary positive integer. There exists a monotone embedded Lagrangian $L \subset \mathbb{C}P^{42k+1}$ with nontrivial triple Massey prosuct. The Lagrangian $L$ is orientable.
\end{theorem}

Massey products of monotone Lagrangians and their $A_{\infty}$ algebras will be studied systematically  in a further paper. Moreover, the example above can be the first step to study the difference between the quantum Massey product and the classical one.

Let $X$ be a manifold. By Sullivan's well-known result \cite{sullivan2}, the existence of a non-trivial Massey product in the singular cohomology ring $H^{*}(X, \mathbb{Q})$ is an obstruction to the rational formality of $X$. Many examples of formal manifolds are known: spheres, projective spaces, compact Lie groups, homogeneous spaces, and all compact Kahler manifolds (see \cite{sullivan}). In particular, if $X$ is a symplectic manifold and there is non-trivial Massey product in $H^{*}(X, \mathbb{Q})$, then $X$ is not Kahler. This provides an effective method for constructing simply connected symplectic non-Kahler manifolds (see \cite{Taimanov1, Taimanov2}).

It is not clear if a monotone Lagrangian of a formal symplectic manifold can be non-formal. The theorem above says that there exists non-formal monotone Lagrangians of $\mathbb{C}P^n$.

\vspace{.06in}

Another question to study is the fundamental group of Lagrangians. In this paper we show that fundamental groups of monotone Lagrangians in $\mathbb{C}P^n$ can be complicated.

\begin{theorem}\label{fundgroup}
Let $S_{17}$ be an orientable surface of genus $17$. Suppose that $L$ is a total space of a nonorientable fibration (the action of $\mathbb{Z}_2$ is defined in the proof))
\begin{equation*}
\begin{gathered}
L \xrightarrow{S_{17}/\mathbb{Z}_2} T^3,
\end{gathered}
\end{equation*}
The fundamental group of $L$ satisfies the following exact sequence:
\begin{equation*}
1 \rightarrow \pi_1(S_{17}/\mathbb{Z}_2) \rightarrow \pi_1(L) \rightarrow \pi_1(T^3) \rightarrow 1.
\end{equation*}
There exists a monotone Lagrangian embedding of $L$ into $\mathbb{C}P^5$.
\end{theorem}

In fact, the fundamental group of the constructed Lagrangians are related to the right-angled Artin groups (see \cite{verevkin}).

\vspace{0.04in}

The theorem below shows that monotone Lagrangians of $\mathbb{C}P^n$ can have an arbitrary torsion. Since the proof of the theorem below is technically complicated and long, we omit the proof.

\begin{theorem}\label{torsion}
For each odd $q \in \mathbb N$, and arbitrary  $N \in \mathbb N$ there are numbers $k = k(q,N)$, $n=n(q,N)$ such that the following is satisfied:
\item[-]  There exist embedded monotone Lagrangians  $L \subset \mathbb{C}P^n$ with the property $H^k(L, \mathbb{Z}) = \mathbb{Z}_q \oplus \ldots$;
\item[-] The minimal Maslov number $N_L$ is even and $N_L \geqslant 2N$;
\item[-] The Lagrangians $L$ are spin;
\item[-] The Lagrangian quantum cohomology of some (not all)  Lagrangians vanishes, i.e. $QH^{*}(L, \mathbb{Z}_2[T, T^{-1}]) = 0$.

\end{theorem}

The proof of Theorem $\ref{torsion}$ is based on the paper \cite[Section ~11]{Bossio}. The main steps of the proof are the following:
\\
- Consider some triangulation of a Lens space;
\\
- Construct a Delzant polytope $P$ from the triangulation and show that the corresponding moment angle manifold has torsion;
\\
- Apply an appropriate multiwedge operation (see Section $\ref{wedge})$ to make $P$ Fano.

\vspace{.06in}

In the end of this section, let us show the following example:

\begin{theorem}\label{threespheretheo}
Let $k$ be an arbitrary non-negative integer and let $p_1, p_2, n$ be integers greater than $2$. Assume that
\begin{equation*}
k < p_1, \quad \; p_2 - p_1+k > p_1, \quad \; n-p_2-k > p_1.
\end{equation*}
Suppose that $L$ is a total space of a fibration
\begin{equation*}
\begin{gathered}
L \xrightarrow{(S^{p_1-1} \times S^{p_2 - p_1 - 1} \times S^{n-p_2-1})/\mathbb{Z}_2} T^2, \\
(x,y,z) \sim (-x-y,-z), \;\;\; (x,y,z) \in S^{p_1-1} \times S^{p_2 - p_1 - 1} \times S^{n-p_2-1}.
\end{gathered}
\end{equation*}
Then
\item[-] There exists monotone Lagrangian embedding of $L$ into $\mathbb{C}P^{n-1}$ with minimal Maslov number $gcd(p_1, p_2-p_1+k, n - p_2 - k)$, where $gcd$ stands for the greatest common divisor;
\item[-] If $k,p_1,p_2,n$ are even numbers, then the fibration is trivial;
\item[-] Note that the topology of $L$ is independent of $k$, but the minimal Maslov number depends on $k$;
\end{theorem}

Let us propose the following conjecture:

\begin{conjecture}
Let $L$ be the Lagrangian defined in Theorem $\ref{threespheretheo}$. Assume that $k=1, p_1 =4, p_2=11, n=20$. Then,  $QH^{*}(L, \mathbb{Z}_2[T, T^{-1}])$ is not equal to zero and is not isomorphic to $H^{*}(L, \mathbb{Z}_2[T, T^{-1}])$. In other words, $L$ is not wide and is not narrow.
\end{conjecture}

\vspace{.15in}

\noindent \textbf{Acknowledgments.}
The author thanks Octav Cornea, Egor Shelukhin, Semyon Abramyan, Artem Kotelskiy, Ivan Limonchenko for many helpful discussions.

This work was supported by the Russian Science Foundation (grant 21-41-00018).

\section{Real toric topology}

In this section we give only basic definitions. For more details we refer the reader to the book \cite[Section 1, Section 6]{torictop}.

\subsection{Polytopes and intersection of real quadrics}\label{mainconstructionpol}

Let $P$ be a convex and bounded polytope. Consider a system of $n$ linear inequalities defining $P$ in $\mathbb{R}^m$
\begin{equation}\label{polytope}
P_{A,b}=\{x\in \mathbb{R}^m: <a_i,x>+b_i \geqslant  0 \quad \;\;\; \quad i=1,...,n  \},
\end{equation}
where $<\cdot,\cdot>$ is the standard Euclidian product on $\mathbb{R}^m$, $a_i \in \mathbb{R}^m$, and $b_i \in \mathbb{R}$. Denote $b=(b_1,...,b_n)^T$, $x=(x_1,...,x_m)^T$, and by $A$ the $m\times n$ matrix whose columns are the vectors $a_i$, where $T$ stands for transposition. Then, our polytope can be written in the following form:
\begin{equation*}
P_{A,b}=\{x\in \mathbb{R}^m: (A^Tx + b)_i \geqslant 0 \}.
\end{equation*}

Faces of the maximal dimension are called \emph{facets} and faces of dimension zero are called \emph{vertices}. If the polytope has dimension $m$, then facets have dimension $m-1$.

\begin{definition}
The polytope $P$ is called \emph{generic} if for any vertex  $x\in P$ the normal vectors $a_i$ of the hyperplanes containing $x$ are linearly independent. The polytope is called \emph{simplicial} if all faces of the polytope are simplices.
\end{definition}

In this paper we assume that all our polytopes are generic, bounded, and contain at least one vertex.
\\

We define
\begin{equation}\label{polytope2}
\begin{gathered}
i_{A,b} : \mathbb{R}^m \rightarrow \mathbb{R}^n, \\
i_{A,b}(x)=A^Tx + b = (<a_1, x> + b_1,...,<a_n, x> + b_n)^T.
\end{gathered}
\end{equation}

Let $\Gamma $ be $(n-m)\times n$-matrix whose rows form a basis of linear relations between the vectors $a_i$. The set of columns $\gamma_1,...,\gamma_n$ of $\Gamma$ is called a \emph{Gale dual} configuration of $a_1,...,a_n$. Each of the matrices $A$ and $\Gamma$ determines the other uniquely up to multiplication by an invertible matrix from the left. We have

\begin{equation}\label{polsys}
\begin{gathered}
i_{A,b}(\mathbb{R}^m)= \{u \in \mathbb{R}^n : \Gamma u = \Gamma b \},\\
\Gamma A^T = 0, \quad u=(u_1,...,u_n)^T, \\
i_{A,b}(P) = i_{A,b}(\mathbb{R}^m)\cap \mathbb{R}^n_{+}
\end{gathered}
\end{equation}

Assume that we have intersection of $(n-m)$ quadrics
\begin{equation}\label{quadrics}
\mathcal{R} = \{u\in \mathbb{R}^n : \gamma_{1,i}u_1^2 + ... + \gamma_{n,i}u_n^2 = \delta_i, \quad i=1,...,n-m, \}.
\end{equation}
The coefficients of the quadrics define $(n-m)\times n$ matrix $\Gamma=(\gamma_{j,i})$. The group $\mathbb{Z}_{2}^n$ acts on $\mathcal{R}$ by
\begin{equation}\label{actionofgr}
\varepsilon\cdot (u_1,...,u_n) = ((-1)^{\varepsilon_1} u_1,...,(-1)^{\varepsilon_n} u_n), \;\; \varepsilon=(\varepsilon_1,...,\varepsilon_n) \in \mathbb{Z}_2^n.
\end{equation}
The quotient $\mathcal{R}/\mathbb{Z}_{2}^n$ can be identified with the set of nonnegative solutions of the system
\begin{equation}\label{nonnegsol}
\left\{
 \begin{array}{l}
\gamma_{1,1}u_1 + ... + \gamma_{n,1}u_n = \delta_1\\
\ldots \quad \quad \\
\gamma_{1,(n-m)}u_1 + ... + \gamma_{n,(n-m)}u_n = \delta_{n-m}
 \end{array}
\right.
\end{equation}
We obtain the same system as in (\ref{polsys}). Solving the homogeneous version of $(\ref{nonnegsol})$ we get the matrix $A$. So, rows of the system above form a basis of linear relations between the vectors $a_i$. Then, we construct a polytope (\ref{polytope}), where $b=(b_1,...,b_n)$ is an arbitrary solution of $(\ref{nonnegsol})$. As a result, we get a polytope $P$ constructed by $A$ and $b$.

Assume we have a polytope $P$. Let us replace $u_i$ by $u_i^2$ in (\ref{polsys}). We get $(n-m)$ quadrics as in $(\ref{quadrics})$. We see that
\begin{equation}\label{rightsidepol}
\Gamma b = \delta.
\end{equation}

So, a polytope defines a system of quadrics and a system of quadrics defines a polytope. We denote by $\widetilde{\mathcal{R}}_P$ the subset of $\mathbb{R}^n$ defined by system $(\ref{quadrics})$. In other words, $\widetilde{\mathcal{R}}_P$ is defined by the system of quadrics associated to the polytope $P$.

\begin{definition}
The manifold $\widetilde{\mathcal{R}}_{P}$ is called the \emph{real moment-angle manifold} associated to $P$.
\end{definition}

\textbf{Example.} Let $P$ be a $5-$gon defined by inequalities
\begin{equation*}
\left\{
 \begin{array}{l}
x_1 + 1 \geqslant 0, \;\; x_2 + 1 \geqslant 0 \\
-x_1 + 1\geqslant 0 \;\; -x_2 + 1 \geqslant 0 \\
-x_1-x_2 + 1 \geqslant 0
 \end{array}
\right.
\end{equation*}
Then $a_1 = (1,0)$, $a_2 = (0,1)$, $a_3 = (-1,0)$, $a_4 = (0,-1)$, $a_5 = (-1,-1)$, $b=(1,1,1,1)$. There are three linear relations $a_1+a_3 = 0$, $a_2 + a_4 = 0$, $a_1+a_2+a_5 = 0$, and $\Gamma$ is the matrix whose rows are $(1,0,1,0,0)$, $(0,1,0,1,0)$, $(1,1,0,0,1)$. We have $\Gamma b = (2,2,3)$. Then, the real moment-angle manifold  $\widetilde{\mathcal{R}}_P$ is defined by
\begin{equation*}
\left\{
 \begin{array}{l}
u_1^2 + u_3^2 = 2\\
u_2^2 + u_4^2 = 2 \\
u_1^2 + u_2^2 + u_5^2 = 3
 \end{array}
\right.
\end{equation*}
It turns out that $\widetilde{\mathcal{R}}_P$ is diffeomorphic to orientable surface of genus $5$ (see \cite{Lopez}, \cite[Proposition 4.1.8 and Theorem 6.2.4]{torictop}). In fact, the real moment-angle manifold associated to $m-$gon is diffeomorphic to an orientable surface of genus $1+(m-4)2^{m-3}$ (see \cite[Proposition 4.1.8 and Theorem 6.2.4]{torictop}).

\begin{definition}
It may happen that some of the inequalities (\ref{polytope}) can be removed from the presentation without changing $P_{A,b}$. Such inequalities are called \emph{redundant}. A presentation without redundant inequalities is called \emph{irredundant}.
\end{definition}

\begin{theorem}(\cite[Proposition 2.1, Theorem 4.3]{Mirpan}, \cite[Proposition 6.2.3]{torictop}, \cite[Lemma 0.12]{Bossio}). Let $\gamma_j= (\gamma_{j,1},...,\gamma_{j,n-m})$ be the $j-$th column of system $(\ref{quadrics})$ and $\delta = (\delta_1,...,\delta_{n-m})$.  Assume that the vectors $\gamma_1,...,\gamma_n$ span  $\mathbb{R}^{n-m}$. We have

\vspace{.04in}
\textbf{1.} The real moment-angle manifold $\widetilde{\mathcal{R}}_{P}$ is nonempty and smooth if and only if the presentation $P_{A,b}$ is generic. Equivalently, $\widetilde{\mathcal{R}}_{P}$ is nonempty and smooth if and only if the following conditions are satisfied: \\
$\bullet$ $\delta \in \mathbb{R}_{\geqslant 0}\langle \gamma_1,...,\gamma_n \rangle$ ($\delta$ is in the cone generated by $\gamma_1,...,\gamma_n)$;\\
$\bullet$ If $\delta \in \mathbb{R}_{\geqslant 0}\langle \gamma_{i_1},...,\gamma_{i_r} \rangle$, then $r \geqslant n-m$.

\vspace{.04in}

\textbf{2.} The manifold $\widetilde{\mathcal{R}}_{P}$ is connected if and only if the presentation $P_{A,b}$ is irredundant. If $P_{A,b}$ is obtained from $P_{A',b}$ by adding $r$ irredundant inequalities, then $\widetilde{\mathcal{R}}_{P} = \widetilde{\mathcal{R}}_{P'} \times \mathbb{Z}_2^{r}$.
\end{theorem}

\textbf{Remark.} In the given references the second part of the theorem above is proved for complex moment-angle manifolds, where $\mathbb{Z}_2$ is replaced by $S^1$. As it is discussed in Section $\ref{torsiondefinitions}$ the real moment-angle manifold is the real part of the complex one. So, the theorem above easily follows from theorems proved in the mentioned references.

\begin{definition}
An $m-$dimensional polytope  $P \subset \mathbb{R}^m$ is called \emph{Delzant} if it is generic and for any vertex $x \in P$ the vectors $a_i$ normal to the facets meeting at $x$ form a basis for the lattice $\mathbb{Z}^m$.
\end{definition}

\begin{definition}\label{Fanopolytope}
A Delzant polytope $P$ is called \emph{Fano} if it is defined by
\begin{equation*}
P_{A,b}=\{x\in \mathbb{R}^m: \langle a_i,x \rangle + c \geqslant  0 \quad \;\;\; \quad i=1,...,n  \},
\end{equation*}
where  each vector $a_i \in \mathbb{Z}^m$ is the primitive integral interior normal to the corresponding facet. In other words, $c=b_1=...=b_n$.
\end{definition}

The faces of a given polytope $P$ form a partially ordered set with respect to inclusion. It is called the \emph{face poset} of P. Two polytopes are called \emph{combinatorially equivalent} if their face posets are isomorphic. The class of combinatorially equivalent polytopes is called the \emph{combinatorial polytope}.

The combinatorial polytope can have different geometric representations given by formulas similar to $(\ref{polytope})$. In this paper we associate a Lagrangian submanifold to each Delzant polytope $P$ and the Maslov classes of the constructed Lagrangians depend on the representation. On the other hand, the homeomorphism type of $\widetilde{\mathcal{R}}_{P}$ depends only on the combinatorial type of $P$ (see \cite[Proposition 6.2.3]{torictop}). So, when we  work on the singular cohomology ring of $\widetilde{\mathcal{R}}_{P}$ we can change the geometric representations of $P$. Any combinatorial polytope $P$ has  representation $(\ref{polytope})$ with $b_i = 1$ for all $i$. Indeed, we need to move the origin to the interior of $P$ by a parallel transform and divide each inequality by $b_i$. Note that \textbf{0}$\in int(P)$ if and only if $b_i > 0$ for all $i$, where $\textbf{0}$ and $int(P)$ stand for the origin of $\mathbb{R}^m$ and the interior of $P$, respectively.

\begin{definition}
The dual polytope $P^{*}$ of $P$ is defined as
\begin{equation*}
P^{*} = \{y \in \mathbb{R}^m: \; <y,x> + 1 \geqslant 0 \; for \; all \; x \in P \}.
\end{equation*}
\end{definition}

\begin{theorem}(\cite[Theorem 2.11]{Gunter}).
Assume that $P$ is given by $(\ref{polytope})$ and $b_i>0$ for any $i$. We have
\\
$\bullet \;\; P^{*}$ is bounded; \\
$\bullet \;\; (P^{*})^{*} = P$;\\
$\bullet \;$ If the polytope $P$ is given by inequalities $(\ref{polytope})$ with $b_i=1$, then $P^{*} = conv(a_1, . . . , a_m)$, where $conv(a_1,...,a_m)$ stands for the convex hull of the vectors $a_1,..,a_m$.
\end{theorem}

\begin{definition}
For a polytope $P$, we denote by $K_P$ the boundary $\partial P^{*}$ of the dual polytope.
\end{definition}

\begin{lemma}\label{simplicialpolytope} (\cite[Example 2.2.4]{torictop}).
If $P$ is generic, then $P^{*}$ is simplicial. Moreover, $K_P$ coincides with the nerve of the covering of $\partial P$ by the facets. That is, the vertices of $K_P$ are identified with the facets of $P$, and some set of vertices of $K_P$ spans a simplex whenever the intersection of the corresponding facets of $P$ is nonempty.
\end{lemma}

In this paper all polytopes are bounded. Let us give a lemma which will be used later. Taking linear combinations of equations we can always write the system  $(\ref{quadrics})$ in the form

\begin{equation*}
\begin{gathered}
\widetilde{\mathcal{R}}_P = \left\{
 \begin{array}{l}
h_1u_1^2+...+h_n u_n^2 = a\\
\gamma_{1,j}u_1^2 + ... + \gamma_{n,j}u_n^2 = 0
 \end{array}
 \right.
 \\
j=1,...,n-m-1
\end{gathered}
\end{equation*}

\begin{lemma}\label{boundedpolytope} (\cite{torictop} Proposiion 1.2.7).
The polytope $P$ is bounded if and only if $a > 0$ and $h_1,...,h_n > 0$.
\end{lemma}

\begin{definition}
Let $P$ be a generic polytope and $H$ be a hyperplane that does not contain any vertex of $P$. If $H$ separates a face $F$ from the other vertices of $P$ and $F \subset \{ H < 0 \}$, then we say that the polytope $P \cap \{H \geqslant 0 \}$ is obtained from $P$ by cutting off the face $F$.
\end{definition}

\subsection{Intersection of complex quadrics}\label{torsiondefinitions}

Let $P_{A,b}$ be a generic polytope defined by $(\ref{polytope})$. Then, replacing $u_j$ by $|z_j|$ in $(\ref{quadrics})$ we get intersection of complex quadrics
\begin{equation*}
\widetilde{\mathcal{Z}}_{P} = \{z \in \mathbb{C}^n : \gamma_{1,j}|z_1|^2 + ... + \gamma_{n,j}|z_n|^2 = \delta_j, \quad j=1,...,n-m, \}.
\end{equation*}
We see that $\widetilde{\mathcal{R}}_{P} = Re(\widetilde{\mathcal{Z}}_{P})$.

\begin{definition}
The manifold $\widetilde{\mathcal{Z}}_{P}$ is called the \emph{complex moment-angle} manifold associated to $P$.
\end{definition}

The difference is the existence of an action of $n-$torus $T^{n} = (e^{i\pi\varphi_1},\ldots,e^{i\pi\varphi_n})$. The torus acts on $z = (z_1, \ldots,z_n)$ by coordinate-wise multiplication. Then, the set $\widetilde{\mathcal{Z}}_{P}/T^n$ can be identified with non-negative solutions of system $(\ref{nonnegsol})$. So, in the same way we get that a polytope defines a system of complex quadrics and a system of complex quadrics defines a polytope. Note that any intersection of complex quadrics in $\mathbb{C}^n$ can be interpreted as an intersection of real quadrics in $\mathbb{R}^{2n}$. Indeed, $|z_j|^2 = |x_j^2 + iy_j^2|^2 = x_j^2 + y_j^2$.

Let us study the topology of $\widetilde{\mathcal{Z}}_{P}$.  First, let us mention the following lemma:

\begin{lemma}\label{complexirr}(\cite[Proposition 6.2.3]{torictop}).
Assume that $P_{A,b}$ is obtained from $P_{A',b}$ by adding $r$ irredundant inequalities, then $\widetilde{\mathcal{Z}}_{P} = \widetilde{\mathcal{Z}}_{P'} \times (S^1)^{r}$.
\end{lemma}

Let $P$ be a generic polytope with $n$ facets and $K_P = \partial P^{*}$. From Lemma $\ref{simplicialpolytope}$ we know that $P^{*}$ is simplicial and $K_P$ is simplicial complex. We denote the facets of $P$ by $v_1, \ldots, v_n$. Let $\mathbf{k}$ be either a field of zero characteristic, or $\mathbb{Z}$. Let $\mathbf{k}[v_1 \ldots v_n]$ be a graded polynomial algebra on $n$ variables, $deg(v_i) = 2$. Recall that the vertices of $K_P$ are identified with facets of $P$ and  a set of vertices of $K_P$ spans a simplex whenever the intersection of the corresponding facets of $P$ is nonempty.
\vspace{.12in}
\begin{definition}(see  \cite[Section ~3]{torictop}).
Let $\mathcal{I}_P$ be the ideal generated by those monomials $v_{i_1}\ldots v_{i_r}$ for which $v_{i_1}\cap...\cap v_{i_r} = \emptyset$. The face ring (or the Stanley-Reisner ring) of $P$ is the quotient graded ring $\mathbf{k}[P] = \mathbf{k}[v_1, \ldots, v_n]/\mathcal{I}_P$. Then $\mathbf{k}[P]$ has a natural structure of algebra and is a module over the polynomial algebra $\mathbf{k}[v_1,...,v_n]$ via the quotient projection $\mathbf{k}[v_1,...,v_n] \rightarrow \mathbf{k}[P]$.
\end{definition}

\vspace{.12in}

\textbf{Remark}. In references the face ring is defined in a different way. They consider an arbitrary simplicial complex $K$ and denote the vertices of $K$ by numbers $\{1,...,n\}$. The ideal $\mathcal{I}_K$ is generated by those monomials $v_{i_1}\ldots v_{i_r}$ for which vertices $i_1,...,i_r$ does not form a simplex in $K$. The face ring is defined by $\mathbf{k}[K] = \mathbf{k}[v_1, \ldots, v_n]/\mathcal{I}_K$. If $K=K_P$ for some generic polytope $P$, then the two given definitions coincide. By a result of Bruns and Gubeladze $\cite{bruns}$ two simplicial complexes $K_1$ and $K_2$ are combinatorially equivalent if and only if their Stanley-Reisner algebras are isomorphic. Thus, $\mathbf{k}[P]$ is a complete combinatorial invariant of $P$.
\\

Let us introduce a formula for the cohomology ring of $\widetilde{\mathcal{Z}}_P$. The face ring $\mathbf{k}[P]$ acquires a module structure over the polynomial algebra. Therefore, we can consider the corresponding $Tor$-modules (see \cite[Appendix Section A.2]{torictop}):

\begin{equation*}
{Tor}_{\mathbf{k}[v_1,...,v_n]}(\mathbf{k}[P], \mathbf{k}) = \bigoplus\limits_{i, j \geqslant 0} {Tor}^{-i, 2j}_{\mathbf{k}[v_1,...,v_n]}(\mathbf{k}[P], \mathbf{k}).
\end{equation*}
Also, we consider the differential bigraded algebra
\begin{equation*}
\begin{gathered}
R(P) = \Lambda[y_1,...,y_n] \otimes \mathbf{k}[P]/(v_i^2 = y_iv_i, \quad 1 \leqslant i \leqslant n), \\
deg(y_i) = 1, \;\;\; deg(v_i) = 2; \;\; \quad dv_i = 0, \;\;\; dy_i = v_i.
\end{gathered}
\end{equation*}

\begin{theorem}\label{torcohom}(\cite[Theorem 4.5.4]{torictop})
Assume that $P$ is irredundant. The cohomology ring of $\widetilde{\mathcal{Z}}_P$ is given by
\begin{equation*}
H^{k}(\widetilde{\mathcal{Z}}_P, \mathbf{k})= \bigoplus\limits_{2j-i=k}{Tor}^{-i, 2j}_{\mathbf{k}[v_1,...,v_n]}(\mathbf{k}[P], \mathbf{k}) = H^{k}[R(P), d], \\
\end{equation*}
where $\mathbf{k}$ is either $\mathbb{Z}$, or a field of zero characteristic. Moreover, $H^1(\widetilde{\mathcal{Z}}_P, \mathbb{Z}) = H^2(\widetilde{\mathcal{Z}}_P, \mathbb{Z}) = 0$.
\end{theorem}

\textbf{Example.} Let us consider a $5-$gon (see Figure $\ref{fig:5gonfig}$)
\begin{equation*}
\left\{
 \begin{array}{l}
x_1 + 1 \geqslant 0, \;\; x_2 + 1 \geqslant 0 \\
-x_1 + 1\geqslant 0 \;\; -x_2 + 1 \geqslant 0 \\
-x_1-x_2 + 1 \geqslant 0
 \end{array}
\right.
\end{equation*}
The complex moment-angle manifold  $\widetilde{\mathcal{Z}}_P$ is defined by
\begin{equation*}
\left\{
 \begin{array}{l}
|z_1|^2 + |z_3|^2 = 2\\
|z_2|^2 + |z_4|^2 = 2 \\
|z_1|^2 + |z_2|^2 + |z_5|^2 = 3
 \end{array}
\right.
\end{equation*}
The manifold $\widetilde{\mathcal{Z}}_P$ is smooth $7-$manifold. Since $v_1 \cap v_4 = \emptyset$, we have $d(y_1v_4) = v_1v_4 = 0$. Also, $d(y_1y_4) = y_1v_4 - y_4v_1$, i.e. elements $y_1v_4$ and $y_4v_1$ represent the same element in $H^{*}[R(P), d]$. We see that $deg(y_1v_4) = 3$ and $y_1v_4$ represents some element in $H^3(\widetilde{\mathcal{Z}}_P, \mathbb{Z})$. Arguing in the same way, we get that the group $H^3(\widetilde{\mathcal{Z}}_P, \mathbb{Z})$ is generated by $y_1v_3, \; y_1v_4, \; y_2v_4, \; y_2v_5, \; y_3v_5$.
\\
Since $v_1 \cap v_3 = v_1 \cap v_4 = \emptyset$, we have $d(y_3y_4v_1) = y_5v_3v_1 - y_3v_5v_1 = 0$. In the same way, we see that $H^4(\widetilde{\mathcal{Z}}_P, \mathbb{Z})$ is generated by elements $y_3y_4v_1, \; y_4y_5v_2, \; y_1y_5v_3, \; y_1y_2v_4, \; y_2y_3v_5$.
\\
Since $v_1 \cap v_3 = \emptyset$ we see that $y_1v_3 \smile  y_3y_4v_1 = 0$ and $y_1v_3 \smile  y_4y_5v_2 = y_1y_4y_5v_3v_2 \neq 0$. Arguing in the same way, we can show that the ring $H^{*}(\widetilde{\mathcal{Z}}_P, \mathbb{Z})$ is isomorphic to the ring $H^{*}(\#5(S^3 \times S^4), \mathbb{Z})$.
\\
In fact, $\widetilde{\mathcal{Z}}_P  = \#5(S^3 \times S^4)$ (see \cite[Exercise 4.2.10]{torictop}).

\begin{figure}[h]
\centering
\includegraphics[width=0.3\linewidth]{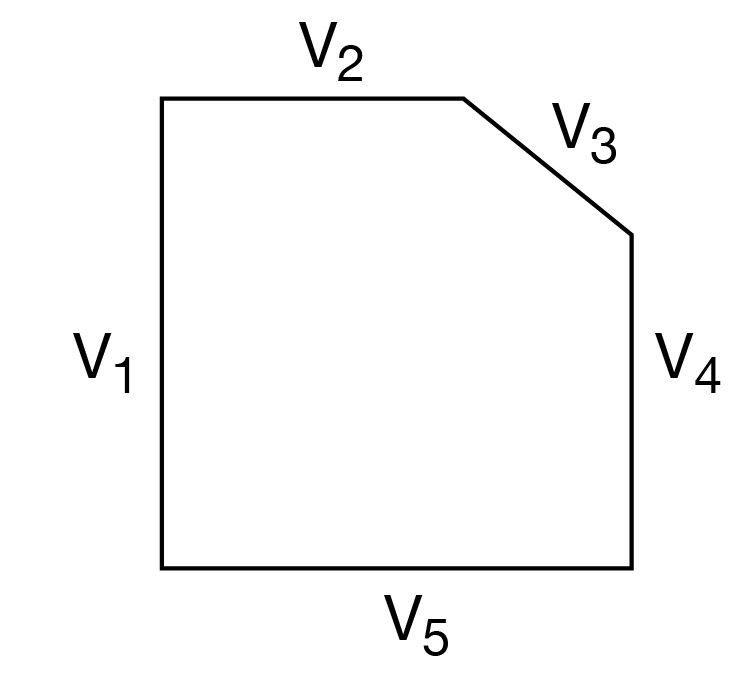}
\caption{5-gon}
  \label{fig:5gonfig}
\end{figure}

Assume that $X$ is a manifold of dimension $n$ without boundary and write $X_{-}$ for $X$ minus a small disk of dimension $n$. Then we define
\begin{equation*}
\mathcal{G}(X) = \partial(X_{-} \times D^2),
\end{equation*}
where $D^2$ is 2-disk. The operation $\mathcal{G}X$ was defined by Gonzalez Acuna in ~\cite{Gonzales}.

\begin{lemma}(\cite{Gonzales})
If $X$, $Y$ are connected manifolds, then
\begin{equation*}
\mathcal{G}(X \# Y) = \mathcal{G}(Y) \# \mathcal{G}(Y).
\end{equation*}
Moreover,
\begin{equation}\label{AcunaLeibniz}
\begin{gathered}
\mathcal{G}(S^p) = S^{p+1}, \\
\mathcal{G}(S^p \times S^q) = (S^{p+1} \times S^q) \times (S^p \times S^{q+1}). \\
\end{gathered}
\end{equation}
Let us note that the formula $(\ref{AcunaLeibniz})$ works only for $S^p \times S^q$. For arbitrary manifolds the formula $\ref{AcunaLeibniz}$ does not hold true.
\end{lemma}

\vspace{.1in}

Let $P$ be a generic polytope and $x \in P$ be a vertex. Denote by $P_x$ the polytope obtained by cutting off the vertex $x$. Let us mention the following two theorems:

\begin{theorem}\label{vertexcutoff}(\cite[Theorem 2.2]{Lopez3})
Let $P$ be a generic polytope of dimension $m$ with $n$ facets. If $n<3m$, then
\begin{equation*}
\widetilde{\mathcal{Z}}_{P_x} = \mathcal{G}(\mathcal{Z}_P) \mathop\sharp^{n-m}_{j=1} \binom{n-m}{j} (S^{j+2} \times S^{n+m-j-1}).
\end{equation*}
\end{theorem}

\begin{theorem}\label{simplexcutoff}(\cite{Gavran})
Let $P$ be a polytope obtained by cutting off $\ell$ vertices from the standard $m-$simplex. Then,
\begin{equation*}
\widetilde{\mathcal{Z}}_P= \mathop\sharp_{j=1}^{\ell}j\binom{\ell+1}{j+1}(S^{j+2} \times S^{2m + \ell - j - 1}).
\end{equation*}
\end{theorem}

\subsection{Simplicial multiwedge operation}\label{wedge}

Let $P$ be a generic polytope and  $\widetilde{\mathcal{R}}_P \subset \mathbb{R}^{n}$ be the corresponding real moment-angle manifold
\begin{equation}\label{equationwedge}
\left\{
 \begin{array}{l}
\gamma_{1,r}u_{1}^2  + ... + \gamma_{n,r}u_n^2 = \delta_r \quad r=1,...,n-m.
 \end{array}
\right.
\end{equation}
Assume that $P$ is given by inequalities $(\ref{polytope})$. Denote by  $\widetilde{\mathcal{R}}_{P, u_1 = 0} \subset \widetilde{\mathcal{R}}_P$ a subset  $\widetilde{\mathcal{R}}_P \cap \{u_1 = 0\}$. Note that $\widetilde{\mathcal{R}}_{P, u_1 = 0}$  is a real moment-angle manifold (it s defined by quadrics) and is associated to some polytope $P_{u_1 = 0}$. As we can see from $(\ref{polytope2})$ and $(\ref{polsys})$ the subset $\widetilde{\mathcal{R}}_{P, u_1 = 0}$ is associated to a facet $P_{u_1 = 0} \subset P$ defined by $<a_1,x>+b_1 =  0 \cap P$.

Let us duplicate the variable $u_1$ in $(\ref{equationwedge})$ with the same coefficients $\gamma_{1,j}$, i.e. we consider a system
\begin{equation*}
\left\{
 \begin{array}{l}
\gamma_{1,i}(u_{1,1}^2 + u_{1,2}^2) + ... + \gamma_{n,r}u_n^2 = \delta_r, \quad r=1,...,n-m
 \end{array}
\right.
\end{equation*}
where $u_{1,1}$ and $u_{1,2}$ are new variables repeating $u_1$. The system above defines a new real moment-angle manifold $\widetilde{\mathcal{R}}_{P}^{(2,1,\ldots, 1)} \subset \mathbb{R}^{n+1}$ and is associated to a new polytope $P^{(2,1,\ldots, 1)}$. Denote by $I$ the segment $[-1,1]$. It turns out that $P^{(2,1,\ldots, 1)}$ can be obtained from $P \times I$ by collapsing $P_{u_1=0} \times I$ into $P_{u_1=0}$, where $P_{u_1=0}$ is the facet of $P$ defined by $<a_1,x>+b_1 =  0 \cap P$.

\begin{definition}
The polytope $P^{(2,1,\ldots,1)}$ is called the book of $P$ and  $(P^{(2,1,\ldots,1)})^{*}$ is called the simplicial wedge of $P^{*}$.
\end{definition}

\textbf{Example.} Assume $P$ is $1-$simplex given by $x_1 +1 \geqslant 0$ and $-x_1 + 1 \geqslant 0$. Then, $\widetilde{\mathcal{R}}_P = S^1$ is defined by the quadric $u_1^2 + u_2^2 = 2$. Duplicating $u_1$ we get $\widetilde{\mathcal{R}}_{P}^{(2,1)}$ defined by $u_{1,1}^2 + u_{1,2}^2 + u_2^2 = 2$. The book of $P$ is obtained by collapsing the edge of the rectangle $P \times I$. We see that $P^{(2,1)}$ is $2-$simplex.
\\

The construction of $\mathcal{R}_{P}^{(2,1,\ldots, 1)}$ from $\mathcal{R}_P$ can be generalized. Let $J=(j_1,...,j_n)$ be a vector of positive integers. Then we construct
a new system of quadrics
\begin{equation*}
\left\{
 \begin{array}{l}
\gamma_{1,r}\sum\limits_{s=1}^{j_1}u_{1,s}^2  + ... + \gamma_{n,r}\sum\limits_{s=1}^{j_n}u_{n,s}^2 = \delta_r, \quad r=1,...,n-m
 \end{array}
\right.
\end{equation*}
obtained from $\widetilde{\mathcal{R}}_P$ by repeating $j_i$ times the variable $u_i$. Denote the resulting real moment-angle manifold by $\widetilde{\mathcal{R}}^{J}_P$ or $\widetilde{\mathcal{R}}^{(j_1,...,j_n)}_P$.

Note that $\widetilde{\mathcal{R}}_P^{(2,...,2)}$ is equivalent to intersection of complex quadrics $\widetilde{\mathcal{Z}}_P$. Indeed,
\begin{equation*}
\left\{
 \begin{array}{l}
\gamma_{1,r}(u_{1,1}^2 + u_{1,2}^2)  + ... + \gamma_{n,r}(u_{n,1}^2 + u_{n,2}^2) = \delta_r, \quad r=1,...,n-m
 \end{array}
\right.
\end{equation*}
is equivalent to
\begin{equation*}
\left\{
 \begin{array}{l}
\gamma_{1,r}|z_1|^2  + ... + \gamma_{n,r}|z_n|^2 = \delta_r, \quad r=1,...,n-m,
 \end{array}
\right.
\end{equation*}

If $J = (2j_1,...,2j_n)$, then $\widetilde{\mathcal{R}}^J_P$ can be interpreted as an intersection of complex quadrics too. In other words,

\begin{equation*}
\begin{gathered}
\widetilde{\mathcal{R}}^{(2j_1,...,2j_n)}_P = \widetilde{\mathcal{Z}}_P^{(j_1,...,j_n)} = \left\{
 \begin{array}{l}
\gamma_{1,r}\sum\limits_{s=1}^{j_1}|z_{1,s}|^2  + ... + \gamma_{n,r}\sum\limits_{s=1}^{j_n}|z_{n,s}|^2 = \delta_r,
 \end{array}
\right.
\\
r=1,...,n-m
\end{gathered}
\end{equation*}
were $z_{i,s}$ are new variables repeating $z_i$.

\begin{theorem}\label{cohomisotopic}(\cite[Corollary 7.6]{Bahri}).
Let $P$ be a generic polytope. Then, $H^{*}(\widetilde{\mathcal{R}}_P^{(2,...,2)}, \mathbb{Z})$ and $H^{*}(\widetilde{\mathcal{R}}_P^{(2j_1,...,2j_n)}, \mathbb{Z})$ are isomorphic as ungraded rings.
\end{theorem}

Note that the theorem above doesn't hold true for $\widetilde{\mathcal{R}}_P$ and $\widetilde{\mathcal{R}}_P^{(2,...,2)}$ (see \cite[Section ~3]{Lopez3}).
\\

Let us study the face ring of the polytope $P^{(j_1,\ldots,j_n)}$ (the face ring is defined in the previous section). Let us denote facets of $P$ by $v_1,...,v_n$ and facets of $P^{(j_1,...,j_n)}$ by $v_{1,1},\ldots, v_{1, j_1},\ldots, v_{n,1},...,v_{n, j_n}$.

\begin{lemma}\label{interfacetswedge}(\cite[Section ~2]{Bahri})
The intersection $v_{i_1}\cap \ldots \cap v_{i_r} = \emptyset$ if and only if $v_{i_1,1}\cap \ldots v_{i_1, j_{i_1}} \ldots \cap v_{i_r, 1} \cap \ldots \cap v_{i_r, j_{i_r}} = \emptyset$.
\end{lemma}

\textbf{Example.} Let $P$ be a rectangle and the corresponding complex moment-angle manifold is defined by
\begin{equation*}
\mathcal{Z}_P = \left\{
 \begin{array}{l}
|z_1|^2 + |z_3|^2 = 2 \\
|z_2|^2 + |z_4|^2 = 2
 \end{array}
\right.
\end{equation*}
We see that $\mathcal{Z}_P = S^3 \times S^3$ and cohomology ring of $\mathcal{Z}_P$ is generated by  $y_1v_3, \; y_2v_4 \in H^3(\mathcal{Z}_P, \mathbb{Z})$. The complex moment-angle manifold $\mathcal{Z}_P^{(2,3,1,2)}$ associated to the polytope $P^{(2,3,1,2)}$ and is defined by
\begin{equation*}
\mathcal{Z}_P^{(2,3,1,2)} = \left\{
 \begin{array}{l}
|z_{1,1}|^2 + |z_{1,2}|^2  + |z_3|^2 = 2 \\
|z_{2,1}|^2 + |z_{2,2}|^2 + |z_{2,3}|^2 + |z_{4,1}|^2 + |z_{4,2}|^2 = 2
 \end{array}
\right.
\end{equation*}
Easy to see that $\mathcal{Z}_P^{(2,3,1,2)} = S^5 \times S^9$. The cohomology ring of $\mathcal{Z}_P^{(2,3,1,2)}$ is generated by $y_{1,1}v_{1,2}v_3 \in H^5(\mathcal{Z}_P^{(2,3,1,2)}, \mathbb{Z}), \; y_{2,1}v_{2,2}v_{2,3}v_{4,1}v_{4,2} \in H^9(\mathcal{Z}_P^{(2,3,1,2)}, \mathbb{Z})$.
\\

\begin{theorem}\label{limonch}(see \cite{Limonchenko})
If $H^{*}(\mathcal{Z}_P, \mathbb{Q})$ has nonzero Massey product, then $H^{*}(\mathcal{Z}_P^{(j_1, \ldots, j_n)}, \mathbb{Q})$  has  nonzero Massey product for any $j_1, \ldots, j_n$.
\end{theorem}

\textbf{Remark.} The reader can skip the rest of this section.  Theorem $\ref{cohomisotopic}$ and Lemma $\ref{interfacetswedge}$ are proved in paper \cite{Bahri}, but the authors use simplicial complexes instead of intersection of quadrics. Let us give more details.

Let $P$ be a generic polytope. From Theorem $\ref{simplicialpolytope}$ we know that $P^{*}$ is simplicial and $K_P = \partial P^{*}$ is a simplicial complex. Assume that $K_P$ has $n$ vertices and denote its vertices by $\{1,...,n\}$.  Let $\mathcal{K}$ be the set of simplices of maximal dimension and we denote by $I$ a simplex of maximal dimension. We say that $i \in I$ if the vertex denoted by number $i$ belongs to the simplex $I$. We define the following:
\begin{equation*}
\begin{gathered}
Z(K_P) = \bigcup_{I \in \mathcal{K}}\prod_{i=1}^m W_i,
\\
W_i = D^{2} \;\; if \;\; i \in I, \quad \; W_i = S^{1} \;\; if \;\; i \not\in I.
\end{gathered}
\end{equation*}
where the union is taken over all simplices of maximal dimension $I$. It is not obvious that $Z(K_P)$ has smooth structure.

\begin{theorem}(\cite{Bahri} Theorem $7.5$ and \cite{torictop} Theorem $2.6.4$).
If $P$ is generic, then the complex moment-angle manifold $\mathcal{Z}_P$ is homeomorphic to $Z(K_P)$. 
\end{theorem}

\textbf{Example.} Let $P$ be $2-$simplex. Then $K_P$ has three vertices and we denote them by $\{1,2,3\}$. There are $3$ simplices of maximal dimension $I_1$, $I_2$, $I_3$. We have
\begin{equation*}
\begin{gathered}
Z(K_P) = \bigcup(S^1 \times D^2 \times D^2) \bigcup (D^2 \times S^1 \times D^2) = \bigcup (D^2 \times D^2 \times S^1) =\\
 \partial (D^2 \times D^2 \times D^2) = \partial (D^6) = S^5.
\end{gathered}
\end{equation*}
The complex moment-angle manifold $\mathcal{Z}_P$ associated to the standard $2-$simplex is given by
\begin{equation*}
|z_1|^2 + |z_2|^2 + |z_3|^2 = 3.
\end{equation*}
We see that $\mathcal{Z}_P = S^5 = Z(K_P)$.
\\

\section{Lagrangian quantum cohomology}\label{quantumdefinitions}

Let us consider $\mathbb{C}P^{n-1}$ endowed with the standard Fubini-Study form $\omega$. Assume that $\omega$ is normalised so that $\omega(\mathbb{C}P^1) = \pi$.  Let $L \subset \mathbb{C}P^{n-1}$ be a Lagrangian submanifold. The \emph{Maslov class} is a homomorphism $\mu : \pi_2(\mathbb{C}P^{n-1}, L) \rightarrow \mathbb{Z}$ defined by the following formula (see \cite{Kail}):
\begin{equation}\label{generalmaslovformula}
\mu(\alpha) = \frac{2n \omega(\alpha)}{\pi} + \frac{\omega_H(\partial \alpha)}{\pi},
\end{equation}
where $\alpha \in \pi_2(\mathbb{C}P^{n-1}, L)$, $\partial \alpha \in \pi_1(L)$ is its boundary, $H$ is the mean curvature vector of $L$, $\omega_H = \omega(H, \cdot)$. Let us assume that $L$ is minimal, i.e. the mean curvature vector $H=0$. Then we have

\begin{equation}\label{maslovclass}
\mu(\alpha) = \frac{2n \omega(\alpha)}{\pi}.
\end{equation}
A Lagrangian $L \subset \mathbb{C}P^{n-1}$ is called \emph{monotone} if
\begin{equation*}
\mu(\alpha) = k\omega(\alpha)
\end{equation*}
for any $\alpha \in \pi_2(\mathbb{C}P^n,L)$ and for some $k>0$.
\\
\\
\textbf{Remark}. If $L\subset \mathbb{C}P^{n-1}$ is a minimal Lagrangian, then $L$ is monotone.
\\

Define the \emph{minimal Maslov number} of $L$ to be the integer
\begin{equation*}
N_L = \{\mu(\alpha) > 0 \; | \; \alpha \in \pi_2(\mathbb{C}P^{n-1}, L)   \}.
\end{equation*}
We can define the Maslov class on $H_2(\mathbb{C}P^{n-1}, L, \mathbb{Z})$ and get a homomorphism $\mu: H_2(\mathbb{C}P^n, L) \rightarrow \mathbb{Z}$. We define the following number:
\begin{equation*}
N^H_L = \{\mu(\alpha) > 0 \; | \; \alpha \in  H_2(\mathbb{C}P^{n-1}, L, \mathbb{Z})  \}.
\end{equation*}
It turns out that for monotone Lagrangians $L \subset \mathbb{C}P^{n-1}$ we have $N_L = N_L^H$.
\\
\\
\textbf{Remark.} We are not giving the general definition of the Lagrangian quantum cohomology of $L$ in this paper. Here we briefly recall basic facts and give all definitions in a suitable form.  The reader can find all details and proofs in papers  \cite{bircan1, bircan2, buhovsky}.
\\

Let $L \subset \mathbb{C}P^{n-1}$ be a monotone orientable  spin Lagrangian with $N_L \geqslant 2$. Let $R$ be a ring and $A = R[T, T^{-1}]$ be the algebra of Laurent polynomials over $R$. We grade $A$ so that deg$(T) = N_L$. Let $f$ be a generic Morse function $f: L \rightarrow \mathbb{R}$. Denote by $C^{*}_f = R\langle Crit(f)\rangle $ the Morse complex associated to $f$. Define a complex generated by the critical points of $f$:
\begin{equation*}
CF^{*}(L, f, A) = R\langle Crit(f)\rangle \otimes A = C_f^{*} \otimes R[T, T^{-1}].
\end{equation*}
We grade $CF^{*}(L, f,  A)$ using the Morse indices of $f$ and the grading of $T$ mentioned above. The complex $CF^{*}(L, f, A)$ is called the \emph{Floer complex} of $L$.  There exists a differential (see \cite[Section 2.3]{bircan1}) $d_F : CF^{*}(L, f, A) \rightarrow CF^{*+1}(L, A)$ of the following form:
\begin{equation*}
d^{f}_{F} = \partial_f^0 + \partial_f^1 \otimes T + ... + \partial_f^m \otimes T^m,
\end{equation*}
where $\partial_f^i : C_f^{*} \rightarrow C_f^{* - iN_L + 1}$, $m = [\frac{dimL+1}{N_L}]$. Note that $\partial_f^0$ is the usual Morse-cohomology boundary operator, i.e. the cohomology of the complex $(CF^{*}(L,f, A), \partial_f^0)$ is isomorphic to the singular cohomology $H^{*}(L, A)$. Similarly, the cohomology of $(C^{*}_f, \partial_f^0)$ is isomorphic to the singular cohomology $H^{*}(L, R)$. The cohomology of the complex $(CF^{*}(L, f, A), d_F)$ is independent of $f$ and is called the \emph{Lagrangian quantum cohomology} of $L$ (see \cite[Theorem 2.1]{bircan1}. We denote it by $QH^{*}(L, A)$.

Assume that $f_1, f_2, f_3$ are generic Morse functions. Then, we are able to define a quantum product
\begin{equation*}
\begin{gathered}
CF^{i}(L,f_1,A) \otimes CF^{j}(L,f_2,A) \rightarrow CF^{i+j}(L,f_3,A) \\
x \ast y = x \ast_0 y + x \ast_1 y T +... + x \ast_k y T^k, \\
x \in C^{i}_{f_1}, \;\; y \in C^{j}_{f_2}, \;\; x \ast_s y \in C^{i+j - sN_L}_{f_3}.
\end{gathered}
\end{equation*}
The map $\ast$ is a chain map and descends to an operation in cohomology (see \cite[Theorem 2.2]{bircan1})
\begin{equation*}
QH^i(L, A) \otimes QH^j(L, A)  \rightarrow QH^{i+j}(L, A).
\end{equation*}
Moreover, the product satisfies the Leibniz rule in the following sense (see \cite{buhovsky}):
\begin{equation*}
d_F^{f_3}(x \ast y) = d_F^{f_1}(x)\ast y \pm x \ast d_F^{f_2}(y).
\end{equation*}
Without loss of generality, assume that the Morse function $f_1$ has a single minimum $x^{min}$. It is proved that $x^{min}$ is the unit, i.e. even at the chain level $x^{min} \ast y = y$ for all critical points of $f_2$.
\\
\\
Let us note that $QH(L, A) = 0$ if and only if $d_F^f(x) = x^{min}$ for some $x \in CF(L, f, A)$.
\\
\\
There exists a $A-$bilinear map chain map $h$
\begin{equation}\label{hyperact}
\begin{gathered}
h : CF^{*}(L, f, A) \rightarrow CF^{*+2}(L, f, A). \\
\end{gathered}
\end{equation}
Moreover, $h$ defines an isomorphism (see \cite[Theorem 2.4]{bircan1})
\begin{equation*}
h : QH^{i}(L, A) \rightarrow QH^{i+2}(L,A).
\end{equation*}
Assume that $\mathbb{C}P^{n-2} \subset \mathbb{C}P^{n-1}$ intersects $L$ transversally. Then $L \cap \mathbb{C}P^{n-2}$ has dimension $n-2$. Denote by $h_{0} \in H_{n-2}(L, R)$ the homology class represented by $L \cap \mathbb{C}P^{n-2}$ and denote by $h^{0} \in H^2(L,R)$ the Poincare dual class to $h_0$. Suppose that the Morse function $f$ has a single minimum $x^{min}$. If the minimal Maslov number $N_L > 2$, then
\begin{equation*}
h(x^{min}) = h^{0}.
\end{equation*}
As a result, we get the following lemma:

\begin{lemma}\label{zeroquantum}
If $N_L > 2$ and $h^{0} = 0$ as a class in $H^2(L, R)$, then $QH(L, A) = 0$.
\end{lemma}

\begin{proof}
We have  $h(x^{min}) = h^{0} = 0$. Therefore, there exists $x$ such that $d_F^f(x) = x^{min}$ because $h$ is isomorphism. Since $x^{min}$ is the unit of the algebra, we get that  $QH(L, A) = 0$.
\end{proof}

Let us recall that $A = R[T, T^{-1}]$ is the algebra of Laurent polynomials over $R$. We have the decomposition $A = \oplus A^i$, where $A^i$ is the subspace of homogenous elements of degree $i$. There exists a spectral sequence $(E^{p,q}, d_r)$ with the following properties (see \cite{bircan1}, \cite{buhovsky}):
\\
\\
-$E_0^{p,q} = C_f^{p+q - pN_L} \otimes T^{p}$ and $d_0 = \partial_0 \otimes 1$.
\\
\\
- $E^{p,q}_1 = H^{p+q - pN_L}(L, R) \otimes A^{pN_L}$, $d_1 = \delta_1 \otimes T$, where $\delta_1 : H^{p+q-pN_L} \rightarrow H^{p+q-(p+1)N_L}$. The multiplication on $H^{*}(L,R)$, induced
from the multiplication on $E^{p,q}_1$, coincides with the standard cup product. In other words, if $a = x\otimes T^{p} \in E^{p,q}_1$, $b = y \otimes T^{p'} \in E^{p',q'}_1$, then $a \cdot b = x \smile y \otimes T^{p+p'}$, where $\cdot$ is the product on $E_1$.
\\
\\
- For every $r \geqslant 1$, $E_r^{p,q} = V_r^{p,q} \otimes A^{p}$ with $d_r = \delta_r \otimes T^r$, where $\delta_r$ are homomorphisms $\delta_r : V_r^{p,q} \rightarrow V^{p+r, q - r+1}$ for every $p, q$ and satisfy $\delta_r \circ \delta_r = 0$.
\\
\\
- The spectral sequence $(E_r^{p,q}, d_r)$ converges to $QH(L, A)$.

\section{Real Moment-angle manifolds and Lagrangian submanifolds}\label{mainconstructionmain}
\subsection{Intersection of quadrics and Lagrangian submanifolds of $\mathbb{C}^{n}$}\label{maincostruction0}
Mironov in \cite{MironovCn} found a very interesting method for constructing Hamiltonian-minimal Lagrangian submanifolds of $\mathbb{C}^n$. In this section we explain his method.

Let $\widetilde{\Gamma}$ be a matrix with columns $\widetilde{\gamma}_j \in \mathbb{Z}^{n-m}$, $j=1,...,n$. Assume that the integer vectors (in this section we omit the transposition sign $T$ and write columns as rows)
\begin{equation*}
\widetilde{\gamma}_j=(\widetilde{\gamma}_{j,1},...,\widetilde{\gamma}_{j, n-m}) \in \mathbb{Z}^{n-m}, \quad j=1,...,n
\end{equation*}
are linearly independent and form a lattice $\widetilde{\Lambda} \subset \mathbb{R}^{n-m}$ of maximum rank $n-m$. The dual lattice $\widetilde{\Lambda}^{*}$ is defined by
\begin{equation*}
\widetilde{\Lambda}^{*}=\{\lambda^{*} \in \mathbb{R}^{n-m}| <\widetilde{\lambda}^{*}, \widetilde{\lambda}> \in \mathbb{Z},  \;\; \forall \widetilde{\lambda} \in \widetilde{\Lambda}\},
\end{equation*}
where $<\cdot, \cdot>$ is the standard Euclidian product on $\mathbb{R}^{n-m}$. Define the following group:
\begin{equation*}
D_{\widetilde{\Gamma}} = \widetilde{\Lambda}^{*}/2\widetilde{\Lambda}^{*}\approx \mathbb{Z}_2^{n-m}.
\end{equation*}
Denote by $T_{\widetilde{\Gamma}}$ an $(n-m)$-dimensional torus
\begin{equation}\label{torus0}
\begin{gathered}
T_{\widetilde{\Gamma}} = (e^{i\pi<\widetilde{\gamma}_1, \widetilde{\varphi}>},...,e^{i\pi<\widetilde{\gamma}_{n}, \widetilde{\varphi}>}) \subset \mathbb{C}^n, \\
\widetilde{\varphi}=(\varphi_1,...,\varphi_{n-m})\in \mathbb{R}^{n-m}.
\end{gathered}
\end{equation}

Let $\widetilde{\mathcal{R}}$ be an $m$-dimensional submanifold of $\mathbb{R}^{n}$ defined by a system
\begin{equation}\label{eqmain0}
\begin{gathered}
\widetilde{\mathcal{R}} = \left\{
 \begin{array}{l}
 \widetilde{\gamma}_{1,r}u_1^2 + ... + \widetilde{\gamma}_{n,r}u_{n}^2 = \delta_r, \quad r=1,...,n-m
 \end{array}
\right.
\end{gathered}
\end{equation}
Consider a map
\begin{equation*}
\begin{gathered}
\widetilde{\psi}: \mathcal{\widetilde{R}} \times T_{\widetilde{\Gamma}} \rightarrow \mathbb{C}^{n},\\
\widetilde{\psi}(u_1,...,u_{n},\widetilde{\varphi}) = (u_1e^{i\pi<\widetilde{\gamma}_1,\tilde{\varphi}>}, ... , u_{n}e^{i\pi<\widetilde{\gamma}_{n},\tilde{\varphi}>}).
\end{gathered}
\end{equation*}
Let $\widetilde{\varepsilon} \in D_{\widetilde{\Gamma}}$ be a nontrivial element. By definition, we have $<\widetilde{\gamma}_j, \widetilde{\varepsilon}> \in \mathbb{Z}$ and $\cos(\pi<\widetilde{\gamma}_j, \varepsilon>) = \pm 1$ for all $j$. Therefore, we can define an action of $D_{\widetilde{\Gamma}}$ on $\mathcal{\widetilde{R}}$ and $T_{\widetilde{\Gamma}}$
\begin{equation*}
\begin{gathered}
\widetilde{\varepsilon} \cdot (u_1,...,u_{n}) =  (u_1\cos(\pi<\widetilde{\gamma}_1, \widetilde{\varepsilon}>),..,u_{n}\cos(\pi< \widetilde{\gamma}_{n}, \widetilde{\varepsilon}>), \\
\widetilde{\varepsilon} \cdot \widetilde{\varphi} = \widetilde{\varphi} + \widetilde{\varepsilon},
\end{gathered}
\end{equation*}
We define the action of $D_{\widetilde{\Gamma}}$ on $\mathcal{\widetilde{R}} \times T_{\widetilde{\Gamma}}$ diagonally, i.e.
\begin{equation*}
\varepsilon \cdot (u_1,...,u_n, \widetilde{\varphi}) = (\varepsilon \cdot (u_1,...,u_n), \widetilde{\varphi} + \varepsilon).
\end{equation*}
Then, we see that
\begin{equation*}
\widetilde{\psi}(u_1,...,u_{n},\widetilde{\varphi}) =  \widetilde{\psi}( \varepsilon \cdot (u_1...,u_n, \widetilde{\varphi})).
\end{equation*}
The formula above shows that $\widetilde{\psi}(\mathcal{R} \times T_{\widetilde{\Gamma}})$ is not embedded and all points of an orbit of $D_{\widetilde{\Gamma}}$ have the same image. So, it is natural to take quotient of $\mathcal{\widetilde{R}} \times {T_{\widetilde{\Gamma}}}$ by  $D_{\widetilde{\Gamma}}$ and consider
\begin{equation*}
\mathcal{\widetilde{N}} = (\mathcal{\widetilde{R}} \times T_{\widetilde{\Gamma}})/D_{\widetilde{\Gamma}}.
\end{equation*}
As a result we have a map
\begin{equation}\label{mainmapcn}
\begin{gathered}
\widetilde{\psi}: \mathcal{\widetilde{N}} \rightarrow \mathbb{C}^{n},\\
\widetilde{\psi}(u_1,...,u_{n},\widetilde{\varphi}) = (u_1e^{i\pi<\widetilde{\gamma}_1, \widetilde{\varphi}>},..., u_{n}e^{i\pi<\widetilde{\gamma}_{n}, \widetilde{\varphi}>}).
\end{gathered}
\end{equation}
Note that $D_{\widetilde{\Gamma}}$ acts freely on $T_{\widetilde{\Gamma}}$. Therefore, the action is free on $\mathcal{\widetilde{R}} \times T_{\widetilde{\Gamma}}$ and $\mathcal{\widetilde{N}}$ is smooth $n$-manifold. The projection
\begin{equation*}
\mathcal{\widetilde{N}} = (\mathcal{\widetilde{R}} \times T_{\widetilde{\Gamma}})/D_{\widetilde{\Gamma}} \rightarrow T_{\widetilde{\Gamma}}/D_{\widetilde{\Gamma}} = T^{n-m}
\end{equation*}
onto the second factor defines a fibration over $(n-m)$-dimensional torus $T^{n-m}$ with fiber $\mathcal{\widetilde{R}}$.
\\
\\
\textbf{Remark.} Let  $\varepsilon_1,...,\varepsilon_{n-m}$ be a basis for $\widetilde{\Lambda}^{*}$. Then these vectors form a basis for the tori, i.e.
\begin{equation*}
\begin{gathered}
T_{\widetilde{\Gamma}} = S_1^1 \times ... \times S_{n-m}^1, \quad S_k^{1} = \mathbb{R}\langle 2\widetilde{\varepsilon}_k\rangle/\mathbb{Z}\langle 2\widetilde{\varepsilon}_k\rangle, \\
T^{n-m} = S_1^1 \times ... \times S_{n-m}^1, \quad S_k^{1} = \mathbb{R}\langle \widetilde{\varepsilon}_k\rangle/\mathbb{Z}\langle \widetilde{\varepsilon}_k\rangle
\end{gathered}
\end{equation*}

\begin{theorem}\label{mironovtheorem}(\cite{MironovCn}).
The submanifold $\widetilde{L} = \widetilde{\psi}(\mathcal{\widetilde{N}}) \subset \mathbb{C}^{n}$ is immersed Hamiltonian-minimal Lagrangian.
\end{theorem}

Let us denote by $\mathbb{Z}\langle \widetilde{\gamma}_1,....,\widetilde{\gamma}_n\rangle$ the set of integer linear combinations of vectors $\widetilde{\gamma}_1, ..., \widetilde{\gamma}_n$. For any $u=(u_1,...,u_n) \in \mathcal{\widetilde{R}}$ we have a sublattice
\begin{equation}\label{embcondition}
\widetilde{\Lambda}_u = \mathbb{Z}\langle \widetilde{\gamma}_j : u_j \neq 0 \rangle \subset \widetilde{\Lambda} = \mathbb{Z}\langle \widetilde{\gamma}_1,...,\widetilde{\gamma}_n \rangle.
\end{equation}

\begin{theorem}\label{mirpanemb0}(\cite[Lemma 3.1]{Mirpan}).
The Lagrangian $\widetilde{L} = \widetilde{\psi}(\mathcal{\widetilde{N}})$ is embedded if and only if $\widetilde{\Lambda}_u = \widetilde{\Lambda}$ for any $u \in \mathcal{\widetilde{R}}$.
\end{theorem}

We know from Section $\ref{mainconstructionpol}$ that there exists a polytope $P$ associated to $\mathcal{\widetilde{R}}$.

\begin{theorem}\label{embdelzant0}(\cite[Lemma 4.5]{Mirpan}).
Let $P$ be the polytope associated to $\widetilde{\mathcal{R}}$. The Lagrangian $\widetilde{L} = \widetilde{\psi}(\mathcal{\widetilde{N}})$ is embedded if and only if the polytope $P$ is Delzant. So, the polytope $P$ is Delzant if and only if $\widetilde{\Lambda}_u = \widetilde{\Lambda}$ for any $u \in \mathcal{\widetilde{R}}$.
\end{theorem}

Let us prove the following lemma, which will be used later.

\begin{lemma}\label{equivmon}
The polytope $P$ is Fano if and only if
\begin{equation}\label{condforlemma}
\widetilde{\gamma}_1 + ... + \widetilde{\gamma}_n = C\delta,
\end{equation}
where $\delta = (\delta_1,...,\delta_{n-m})$ and $C$ is some number.
\end{lemma}

\begin{proof}
From formula $(\ref{rightsidepol})$ we know that $\widetilde{\Gamma} b = \delta$, where $\widetilde{\Gamma}$ is the matrix with columns $\widetilde{\gamma}_1,...,\widetilde{\gamma}_n$ and $b=(b_1,...,b_{n})^T$ defined in formula $(\ref{polytope})$.

If condition $\ref{condforlemma}$ is satisfied, then  $b = (\frac{1}{C},...,\frac{1}{C})$ solves   $\widetilde{\Gamma} b = \delta$. This yields that the polytope $P$ is written as in definition $\ref{Fanopolytope}$. Therefore, $P$ is Fano.

Assume that $P$ is Fano. This means that $b = (C,...,C)$. Then the formula  $\widetilde{\Gamma} b = \delta$ implies $\gamma_1+...+\gamma_n = C\delta$.
\end{proof}

\textbf{Remark.} In fact, the number $C$ is always positive (we assume that $P$ is bounded). Taking some linear combinations we can assume that $(\ref{eqmain0})$ is written in the form
\begin{equation*}
\begin{gathered}
\left\{
 \begin{array}{l}
\widetilde{\gamma}_{1,1}u_1^2 + ... + \widetilde{\gamma}_{1, n}u_n^2 =  a   \\
\widetilde{\gamma}_{1,j}u_1^2 + ... + \widetilde{\gamma}_{n,j}u_{n}^2 = 0 , \quad \; j=2,...,n-m
 \end{array}
\right.
\end{gathered}
\end{equation*}
where $a > 0$ and $\widetilde{\gamma}_{1,1},...,\widetilde{\gamma}_{1,n} > 0$ (see Lemma $\ref{boundedpolytope}$). Then the condition $\widetilde{\gamma}_1 + ... + \widetilde{\gamma}_n = aC$ implies $C>0$.
\\

Let us define an $(n-m)$-dimensional vector
\begin{equation*}
\widetilde{\gamma}_1+...+\widetilde{\gamma}_n = (t_1,...,t_{n-m}).
\end{equation*}

In fact, Mironov in his paper \cite{MironovCn} almost found the Maslov class of the constructed Lagrangians. It was noticed by the author of this paper that the monotone Lagrangians are related to Fano polytopes. Let $\omega_{st}$ be the standard symplectic form of $\mathbb{C}^n$.

\begin{theorem}\label{fanocondition}(\cite{Og} Theorem $1.1$)
Assume that $\widetilde{\mathcal{R}}$ is connected (equivalently P is irredundant). The Lagrangian $\widetilde{L} = \widetilde{\psi}(\mathcal{\widetilde{N}}) \subset \mathbb{C}^n$ is embedded and monotone if and only if $P$ is Delzant and Fano. Moreover, the Maslov class of $\widetilde{L}$ is given by
\begin{equation*}
\mu_{\widetilde{L}}(\beta) = \int\limits_{\partial\beta} t_1d\varphi_1 + ... + t_{n-m}d\varphi_{n-m}, \;\;\; \beta \in \pi_2(\mathbb{C}^n, \widetilde{L}),
\end{equation*}
where $\partial(\beta)$ is the boundary of $\beta$.

More explicitly, if $\widetilde{\gamma}_1+...+\widetilde{\gamma}_n = C\delta$ (equivalently $P$ is Fano, see Lemma $\ref{equivmon}$), then
\begin{equation*}
\mu_{\widetilde{L}} = \frac{2C}{\pi}\omega_{st}.
\end{equation*}
The Maslov class is given by
\begin{equation*}
\mu_{\widetilde{L}}(\beta) = \int\limits_{\partial\beta} C\delta_1d\varphi_1 + ... + C\delta_{n-m}d\varphi_{n-m}, \;\;\; \beta \in \pi_2(\mathbb{C}^n, \widetilde{L}).
\end{equation*}
\end{theorem}

\subsection{Intersection of quadrics and Lagrangian submanifolds of $\mathbb{C}P^{n-1}$}\label{maincostruction}

The ideas of this section belong to Mironov, Panov, and Kotelskiy \cite{MironovCn, Mirpan, kotel}. We give a different explanation of their results and prove some new lemmas.
\\

We keep notations from Section $\ref{maincostruction0}$. Let $\Gamma$ be a matrix with columns $\gamma_j \in \mathbb{Z}^{n-m-1}$, $j=1,...,n$, $m \geqslant 1$. Assume that the integer vectors
\begin{equation*}
\gamma_j=(\gamma_{j,1},...,\gamma_{j, n-m-1}) \in \mathbb{Z}^{n-m-1}, \quad j=1,...,n
\end{equation*}
are linearly independent. Suppose  $(\ref{eqmain0})$ is written in the form

\begin{equation}\label{eqmain}
\begin{gathered}
\left\{
 \begin{array}{l}
u_1^2 + ... + u_n^2 = \delta \\
\gamma_{1,r}u_1^2 + ... + \gamma_{n,r}u_{n}^2 = 0
 \end{array}
\right.
\\
r=1,...,n-m-1, \;\;\;\; \gamma_{j, r} \in \mathbb{Z}, \;\;\;\; \delta \in \mathbb{R}
\end{gathered}
\end{equation}
and define a smooth submanifold $\widetilde{\mathcal{R}} \subset \mathbb{R}^{n}$. In other words, $\widetilde{\gamma}_j = (1, \gamma_j)$.

\begin{lemma}\label{sumzero}
Let $P$ be the polytope associated to system $(\ref{eqmain})$ and denote by $a_1,..,a_n$ the vectors normal to the facets of $P$. Then, $a_1+....+a_n = 0$.

If $a_1+...+a_n = 0$, then the system of quadrics associated to $P$ can be written in the form $(\ref{eqmain})$.
\end{lemma}

\begin{proof}
From Section $\ref{mainconstructionpol}$ we know that the coefficients of the quadrics are coefficients of  linear relations between the vectors $a_1,...,a_n$. The first equation is equivalent to $a_1+...+a_n = 0$.
\end{proof}

Let $\Lambda \subset \mathbb{R}^{n-m-1}$ be a lattice generated by vectors
\begin{equation*}
\gamma_i - \gamma_{j}, \;\; \;\; i,j=1,...,n,
\end{equation*}
Assume that the rank of $\Lambda$ equals $n-m-1$. Since  $\gamma_i - \gamma_j = (\gamma_i - \gamma_n) - (\gamma_j - \gamma_n)  \in \Lambda$, we see that $\Lambda$ is generated by vectors $\gamma_1 - \gamma_n,...,\gamma_{n-1} - \gamma_n$. We define the dual lattice $\Lambda^{*}$, $D_{\Gamma}$, and $(n-m-1)-$torus $T_{\Gamma}$ as in the previous section
\begin{equation}\label{torus}
\begin{gathered}
\Lambda^{*}=\{\lambda^{*} \in \mathbb{R}^{n-m-1}| <\lambda^{*},\lambda> \in \mathbb{Z},  \;\; \forall \lambda \in \Lambda\}, \\
D_{\Gamma} = \Lambda^{*}/2\Lambda^{*}\approx \mathbb{Z}_2^{n-m-1},  \\
T_{\Gamma} = (e^{i\pi<\gamma_1 - \gamma_{n},\varphi>},...,e^{i\pi<\gamma_{n-1} - \gamma_{n},\varphi>}, 1) \subset \mathbb{C}^n.
\end{gathered}
\end{equation}
where $\varphi=(\varphi_1,...,\varphi_{n-m-1})\in \mathbb{R}^{n-m-1}$ and $<\cdot, \cdot>$ is the standard Euclidian product.

Let $\widetilde{\mathcal{R}}$ be the submanifold of $\mathbb{R}^n$ defined by system $(\ref{eqmain})$ and
\begin{equation}\label{mainquotient}
\mathcal{R} = \widetilde{\mathcal{R}}/\mathbb{Z}_2, \;\; (u_1,...,u_n) \sim (-u_1,...,-u_n).
\end{equation}

\vspace{.11in}

\textbf{Remark.}(see Section $\ref{sympalg}$ for more details). In fact, $\mathcal{R}$ is $m$-dimensional submanifold of $\mathbb{R}P^{n-1}$ defined by the system
\begin{equation}\label{projectiveq}
\begin{gathered}
\left\{
 \begin{array}{l}
\gamma_{1,r}u_1^2 + ... + \gamma_{n,r}u_{n}^2 = 0, \quad r=1,...,n-m-1.
 \end{array}
\right.
\end{gathered}
\end{equation}

\vspace{.18in}

Consider a map
\begin{equation*}
\begin{gathered}
\psi: \mathcal{R} \times T_{\Gamma} \rightarrow \mathbb{C}P^{n-1},\\
\psi(u_1,...,u_{n},\varphi) = [u_1e^{i\pi<\gamma_1 -\gamma_{n},\varphi>} : ... : u_{n-1}e^{i\pi<\gamma_{n-1}-\gamma_{n},\varphi>}, u_{n}].
\end{gathered}
\end{equation*}

Let $\varepsilon \in D_{\Gamma}$ be a nontrivial element. We define an action of $D_{\Gamma}$ on $\mathcal{R} \times T_{\Gamma}$ by
\begin{equation*}
\begin{gathered}
\varepsilon \cdot (u_1,...,u_{n}) = (u_1\cos(\pi<\gamma_1-\gamma_{n}, \varepsilon>),..,u_{n-1}\cos(\pi<\gamma_{n-1} - \gamma_{n}, \varepsilon>), u_{n}), \\
\varepsilon \cdot \varphi = \varphi + \varepsilon, \\
\varepsilon \cdot (u_1,...,u_{n},\varphi) = (\varepsilon \cdot (u_1,...,u_n), \varepsilon \cdot \varphi).
\end{gathered}
\end{equation*}
Then, we see that
\begin{equation*}
\psi(u_1,...,u_{n},\varphi) = \psi(\varepsilon \cdot (u_1,.., u_{n} , \varphi)).
\end{equation*}
As in the previous section, all points of an orbit of $D_{\Gamma}$ have the same image. So, we define
\begin{equation}\label{mainmap}
\begin{gathered}
\mathcal{N} = (\mathcal{R} \times T_{\Gamma})/{D_{\Gamma}}, \\
\psi: \mathcal{N} \rightarrow \mathbb{C}P^{n-1},\\
L =  \psi(u_1,...,u_{n},\varphi) = [u_1e^{i\pi(\gamma_1 - \gamma_{n},\varphi)}:...:u_{n-1}e^{i\pi(\gamma_{n-1} - \gamma_{n},\varphi)}: u_{n}] = \\
[u_1e^{i\pi(\gamma_1,\varphi)}:...:u_{n-1}e^{i\pi(\gamma_{n-1} ,\varphi)}, u_{n}e^{i\pi(\gamma_{n}, \varphi)}].
\end{gathered}
\end{equation}

Let us note that the action of $D_{\Gamma}$ is free on the second factor $T_{\Gamma}$. This implies that $L$ (or equivalently $\mathcal{N}$) fibers over $(n-m-1)-$dimensional torus
\begin{equation}\label{mainprojection}
L \xrightarrow{\mathcal{R}} T^{n-m-1} = T_{\Gamma}/D_{\Gamma}.
\end{equation}

\vspace{.08in}

\begin{theorem}\label{mironovtheorem2}(\cite{MironovCn})
The Lagrangian $L = \psi(\mathcal{N}) \subset \mathbb{C}P^{n-1}$ is immersed. If $\gamma_1 + ... +\gamma_{n} = 0$, then $L$ is minimal with respect to the Fubini-Study metric, i.e. the mean curvature vector of $L$ is equal to zero.
\end{theorem}

\begin{corollary}\label{monmin}
If $\gamma_1+...+\gamma_n = 0$, then $L$ is monotone Lagrangian (see formula $(\ref{generalmaslovformula})$).
\end{corollary}

Let $\widetilde{\Lambda}$ be the lattice generated by vectors $\widetilde{\gamma_j} = (1, \gamma_j)$. For any $u=(u_1,...,u_n) \in \mathcal{R}$ we have a sublattice
\begin{equation*}
\Lambda_u = \mathbb{Z}\langle \gamma_i -\gamma_j : u_i, u_j \neq 0 \rangle \subset \Lambda
\end{equation*}
and define $\widetilde{\Lambda}_u$ as in $(\ref{embcondition})$.  There is analogue of Theorem $\ref{mirpanemb0}$.

\begin{lemma}\label{embc1}
If $\Lambda_u = \Lambda$ for any $u  \in \mathcal{R}$, then the Lagrangian $L = \psi(\mathcal{N})$ is embedded.
\end{lemma}

\begin{proof}

Assume that $L$ is not embedded. Then
\begin{equation*}
[\hat{u}_1e^{i\pi<\gamma_1, \hat{\varphi}>}:...: \hat{u}_{n}e^{i\pi<\gamma_{n}, \hat{\varphi}>}] = [u_1e^{i\pi<\gamma_1, \varphi>}:...: u_ne^{i\pi<\gamma_n, \varphi>}]
\end{equation*}
for some  $\varphi, \hat{\varphi}$, $u=(u_1,...,u_n)$, $\hat{u}=(\hat{u}_1,...,\hat{u}_n)$. Suppose that $u_r \neq 0$. Then, we have
\begin{equation*}
\hat{u}_r = u_r, \;\; \hat{u}_je^{i\pi<\gamma_j - \gamma_r, \hat{\varphi}>} = u_je^{i\pi<\gamma_j - \gamma_r, \varphi>}, \;\; \forall u_j \neq 0, \; j\neq r.
\end{equation*}
This implies
\begin{equation*}
\begin{gathered}
e^{i\pi<\gamma_j - \gamma_r, \hat{\varphi} - \varphi>} \in \{1, -1 \} \; \Rightarrow \; <\gamma_j - \gamma_r, \hat{\varphi} - \varphi> \in \mathbb{Z} \; \Rightarrow \; \hat{\varphi} - \varphi \in \Lambda^{*}_u = \Lambda^{*}, \\
e^{i\pi<\gamma_j - \gamma_r, \hat{\varphi} - \varphi>} = \cos(\pi<\gamma_j - \gamma_r, \varepsilon>)
\end{gathered}
\end{equation*}
and $\hat{\varphi} - \varphi$ represents some element $\varepsilon \in D_{\Gamma} = \Lambda^{*}/2\Lambda^{*}$ (here we use that $\Lambda_u^{*} = \Lambda^{*}$). This means that $\hat{u}_j = u_j\cos(\pi<\gamma_j - \gamma_r, \varepsilon>)$. In other words, $\psi(\hat{u}, \hat{\varphi}) = \psi(u, \varphi)$ yields that $\hat{u} = \varepsilon \cdot u$ and $\hat{\varphi} = \varphi + \varepsilon$, but  we already identified these points. So, $L$ is embedded.

\end{proof}

\begin{lemma}\label{embc2}
If $\widetilde{\Lambda}_u = \widetilde{\Lambda}$, then $\Lambda_u = \Lambda$.
\end{lemma}

\begin{proof}
If $\gamma \in \Lambda$, then there are integers $\lambda_1,...,\lambda_{n-1}$  such that
\begin{equation*}
\lambda_1(\gamma_1 - \gamma_n) + ... + \lambda_{n-1}(\gamma_{n-1} - \gamma_n) = \gamma.
\end{equation*}
Let us recall that $\widetilde{\gamma_j} = (1, \gamma_j)$. Note that $\lambda_1(\widetilde{\gamma}_1 - \widetilde{\gamma}_n) + ... + \lambda_{n-1}(\widetilde{\gamma}_{n-1} - \widetilde{\gamma}_n) = (0, \gamma)$. So, if $\gamma \in \Lambda$, then $(0, \gamma) \in \widetilde{\Lambda}$.

Let $u=(u_1,...,u_n)$ be an arbitrary point of $\mathcal{R}$.  Then, $(0, \gamma) \in \widetilde{\Lambda} = \widetilde{\Lambda}_u$. Without loss of generality, assume $u_1,...,u_r \neq 0$. Then there are integers $\mu_1,...,\mu_r$ such that $\mu_1(1, \gamma_1) + ... + \mu_r(1, \gamma_r) = (0, \gamma)$. We see that
\begin{equation*}
\begin{gathered}
\mu_1 + ... +\mu_r = 0 \Rightarrow \\
(0,\gamma) = \mu_1\widetilde{\gamma}_1 + ... + \mu_{r-1}\widetilde{\gamma}_{r-1} + (-\mu_1 - ... -\mu_{r-1})\widetilde{\gamma_r} = \\
\mu_1(\widetilde{\gamma}_1 - \widetilde{\gamma}_r) + ... + \mu_r(\widetilde{\gamma}_{r-1} - \widetilde{\gamma}_r) = \mu_1(0, \gamma_1 - \gamma_r) + ... + \mu_{r-1}(0, \gamma_{r-1} - \gamma_r).
\end{gathered}
\end{equation*}
So, we get $\gamma = \mu_1(\gamma_1 - \gamma_r) + ... + \mu_{r-1}(\gamma_{r-1} - \gamma_r)$. This yields that $\gamma  \in \Lambda_u$. Since $\gamma$ is arbitrary element of $\Lambda$, we have $\Lambda \subseteq \Lambda_u$. Hence, $\Lambda = \Lambda_u$.

\end{proof}

Note that the following two conditions are equivalent
\begin{equation}\label{monconstant}
\begin{gathered}
\gamma_1 +... \gamma_n = 0  \Leftrightarrow \\
(1, \gamma_1) + ... + (1, \gamma_n) = \widetilde{\gamma}_1 + ... + \widetilde{\gamma}_n = \frac{n}{\delta}(\delta,0,...,0).
\end{gathered}
\end{equation}
\\
\\
Let $P$ be the polytope associated to system $(\ref{eqmain})$. Then from Lemma $\ref{equivmon}$ and formula $(\ref{monconstant})$  we get the following:
\begin{lemma}
If  $\gamma_1+...+\gamma_n = 0$, then the polytope associated to $P$ is Fano.
\end{lemma}

As a result, from Lemma $\ref{sumzero}$, Corollary $\ref{monmin}$, Lemmas $\ref{embdelzant0}$, $\ref{embc1}$, $\ref{embc2}$  we have

\begin{theorem}\label{monotonecondition}
Let $a_1,\ldots,a_n$ be the normal vectors to the facets of $P$.  The system of quadrics associated to $P$ can be written in the form $(\ref{eqmain})$ if and only if $a_1+\ldots+a_n = 0$. If $\widetilde{\Lambda}_u = \widetilde{\Lambda}$ for any $u \in \widetilde{\mathcal{R}}$ and $\gamma_1+...+\gamma_n = 0$, then the Lagrangian $L = \psi(\mathcal{N}) \subset \mathbb{C}P^{n-1}$ is embedded and  monotone.

In other words, if $P$ is Delzant, Fano, and $a_1+\ldots + a_n = 0$, then there exists monotone embedded Lagrangian $L  \subset \mathbb{C}P^{n-1}$ associated to $P$.
\end{theorem}

\textbf{Remark.} Actually, any equation $(\ref{eqmain0})$, after taking some linear combinations and changing of variables, can be written in the form $(\ref{eqmain})$. The problem is that the condition $\widetilde{\Lambda}_{u} = \widetilde{\Lambda}$ may not be preserved. Let us consider an example. Suppose $\widetilde{\mathcal{R}}$ is given by a quadric
\begin{equation*}
\begin{gathered}
\left\{
 \begin{array}{l}
2u_1^2 +u_2^2+u_3^2 + u_4^2 + u_5^2 = 5 \\
u_1^2 + u_4^2+u_5^2 = 3
 \end{array}
\right.
\end{gathered}
\end{equation*}
We see that $\widetilde{\Lambda}_{u}  = \widetilde{\Lambda}$ for any $u \in \widetilde{\mathcal{R}}$. After change of variables we get
\begin{equation*}
\begin{gathered}
\left\{
 \begin{array}{l}
u_1^2 +u_2^2+u_3^2 + u_4^2 + u_5^2 = 5 \\
u_1^2 + 2u_4^2+2u_5^2 = 6
 \end{array}
\right.
\end{gathered}
\end{equation*}
We see that for the new system $\widetilde{\Lambda} = \mathbb{Z}^2$, but $\widetilde{\Lambda}_{u_1=0} \neq \mathbb{Z}^2$.
\\
\\
\textbf{Example.} If $P$ is the standard $n-$simplex, then
\begin{equation*}
\begin{gathered}
\widetilde{\mathcal{R}} = \{ (u_1,\ldots, u_n) \in \mathbb{R}^n \; | \; u_1^2 + ... + u_n^2 = \delta \}
\\
\mathcal{R} = \widetilde{\mathcal{R}}/\mathbb{Z}_2 = \mathbb{R}P^{n-1}.
\end{gathered}
\end{equation*}
We get that $L = \mathbb{R}P^{n-1}$.

\section{Proof of main theorems}

\subsection{Minimal Maslov number}

In this section we find a formula for the minimal Maslov number of some Lagrangians.

From Sections $\ref{maincostruction0}$ and $\ref{maincostruction}$ we know that there are Lagrangians $L \subset \mathbb{C}P^{n-1}$ and $\widetilde{L} \subset \mathbb{C}^n$ associated to system $(\ref{eqmain})$. So, there are two Lagrangians associated to the same system. Assume that the Lagrangians are embedded and monotone (see Theorem $\ref{monotonecondition}$ and Theorem $\ref{fanocondition}$)

Recall that the Lagrangians $L $ and $\widetilde{L}$  fiber over $T^{n-m-1} = T_{\Gamma}/D_{\Gamma}$ and $T^{n-m}=T_{\widetilde{\Gamma}}/D_{\widetilde{\Gamma}}$, respectively. We have
\begin{equation*}
\begin{gathered}
L \xrightarrow{\mathcal{R}} T^{n-m-1}, \;\;\;\; \widetilde{L} \xrightarrow{\widetilde{\mathcal{R}}} T^{n-m},
\end{gathered}
\end{equation*}
where $\mathcal{R} = \widetilde{\mathcal{R}}/\mathbb{Z}_2$.

\begin{theorem}\label{minmaslovtheorem}
Suppose that $\widetilde{\mathcal{R}}$ is connected and
\begin{equation*}
\pi_1(\widetilde{\mathcal{R}}) = 0, \;\;\; \pi_1(L) = \pi_1(\mathcal{R}) \bigoplus \pi_1(T^{n-m-1}) = \mathbb{Z}_2 \bigoplus \mathbb{Z}^{n-m-1}.
\end{equation*}
Then
\begin{equation}
\begin{gathered}
H_2(\mathbb{C}P^{n-1}, L, \mathbb{Z}) = \mathbb{Z}^{n-m}, \\
N_L = N_{\widetilde{L}},
\end{gathered}
\end{equation}
where $L$, $\widetilde{L}$, are the Lagrangians associated to system $(\ref{eqmain})$, and $N_L$, $N_{\widetilde{L}}$ are the minimal Maslov numbers.

\end{theorem}

\begin{proof}
We use notations from Section $\ref{maincostruction}$. Let us recall that the lattice $\widetilde{\Lambda}$ is generated by vectors $\widetilde{\gamma}_j = (1, \gamma_j)$ and the lattice $\Lambda$ is generated by vectors $\gamma_j - \gamma_n$, where $j=1,\ldots,n$. Note that $\widetilde{\Lambda}$ is generated by vectors $\widetilde{\gamma}_1 - \widetilde{\gamma}_n, \ldots, \widetilde{\gamma}_{n-1} - \widetilde{\gamma}_n, \widetilde{\gamma}_n$. Therefore, we can choose the following basis for the dual lattice $\widetilde{\Lambda}^{*}$:
\begin{equation*}
\widetilde{\varepsilon}_{n-m}=(1,0...,0), \;\; \widetilde{\varepsilon}_{r}=(\varepsilon_{r,1},...,\varepsilon_{r, n-m}), \;\; r=1,...,n-m-1. \\
\end{equation*}
Denote by $\varepsilon_r$, an $(n-m-1)-$vector
\begin{equation*}
\varepsilon_r = (\varepsilon_{r, 2},..., \varepsilon_{r, n-m}).
\end{equation*}

\begin{lemma}\label{basisgen}
If vectors $\widetilde{\varepsilon}_{1},...,\widetilde{\varepsilon}_{n-m}$ form a basis for $\widetilde{\Lambda}^{*}$, then $\varepsilon_1,...,\varepsilon_{n-m-1}$ form a  basis for $\Lambda^{*}$.
\end{lemma}

\begin{proof}

Since $\Lambda$ is generated by vectors $\gamma_1 - \gamma_n, \ldots, \gamma_{n-1} - \gamma_n$, we have
\begin{equation*}
<\gamma_j - \gamma_n, \varepsilon_r> = <(1,\gamma_j) - (1, \gamma_n), \widetilde{\varepsilon}_r>  =<\widetilde{\gamma}_j - \widetilde{\gamma}_n, \widetilde{\varepsilon}_r> \in \mathbb{Z}.
\end{equation*}
Therefore, $\varepsilon_r \in \Lambda^{*}$. Let $\varepsilon$ be an arbitrary element of $\Lambda^{*}$. We see that
\begin{equation*}
<(1, \gamma_j), (-<\gamma_n, \varepsilon>, \varepsilon)> = -<\gamma_n, \varepsilon> + <\gamma_j, \varepsilon> = <\gamma_j - \gamma_n, \varepsilon> \in \mathbb{Z}.
\end{equation*}
Hence $(-<\gamma_n, \varepsilon>, \varepsilon) \in \widetilde{\Lambda}^{*}$. This means that there exist integers $\lambda_1, \ldots,\lambda_{n-m}$ such that
\begin{equation*}
\begin{gathered}
\lambda_1\widetilde{\varepsilon}_1 +... + \lambda_{n-m}\widetilde{\varepsilon}_{n-m} = (-<\gamma_n, \varepsilon>, \varepsilon) \Rightarrow \\
\lambda_1\varepsilon_1 + \ldots + \lambda_{n-m-1}\varepsilon_{n-m-1} = \varepsilon.
\end{gathered}
\end{equation*}

\end{proof}

If $\widetilde{\mathcal{R}}$ is connected, then for each $r=1,...,n-m$ there exists a curve $(v_{r,1}(s),...,v_{r,n}(s))$ such that
\begin{equation*}
\begin{gathered}
(v_{r,1}(0),...,v_{r,n}(0)) = (u_1,..,u_n), \;\; u_n \neq 0 \\
(v_{r,1}(1),...,v_{r,n}(1)) = (u_1\cos(\pi<\widetilde{\gamma}_1, \widetilde{\varepsilon}_r>), \ldots ,u_n\cos(\pi<\widetilde{\gamma}_n, \widetilde{\varepsilon}_r>))
\end{gathered}
\end{equation*}
for some $(u_1,\ldots, u_n) \in \widetilde{\mathcal{R}}$. Then we see that the curve
\begin{equation*}
\widetilde{\alpha}_r(s) = \widetilde{\psi}(v_{r,1}(s),...,v_{r,n}(s), s\widetilde{\varepsilon}_r) \subset \widetilde{L}, \;\;\; s\in [0,1]
\end{equation*}
is closed and there are elements $\widetilde{\beta}_1,...,\widetilde{\beta}_{n-m} \in \pi_2(\mathbb{C}^{n}, \widetilde{L})$ such that $\partial \widetilde{\beta}_r = \widetilde{\alpha}_r$.

Let us recall that the embedding of $\widetilde{L} \subset \mathbb{C}^n$ corresponding to system $(\ref{eqmain})$ is given by
\begin{equation*}
\begin{gathered}
\widetilde{L} = (u_1e^{i\pi\varphi_1 + i\pi<\gamma_1, \varphi>},..., u_ne^{i\pi\varphi_1 + i\pi<\gamma_n, \varphi>}) \subset \mathbb{C}^{n}, \quad \varphi \in \mathbb{R}^{n-m-1}.
\end{gathered}
\end{equation*}
Consider the Hopf projection $\sigma: \mathbb{C}^{n} \rightarrow \mathbb{C}P^{n - 1}$
\begin{equation*}
\begin{gathered}
\sigma(\widetilde{L}) = [u_1e^{i\pi\varphi_1 + i\pi<\gamma_1, \varphi>}:...: u_ne^{i\pi\varphi_1 + i\pi<\gamma_n, \varphi>}] = \\
 [u_1e^{i\pi<\gamma_1, \varphi>}:...: u_ne^{i\pi<\gamma_n, \varphi>}].
 \end{gathered}
\end{equation*}
We see that $\sigma(\widetilde{L})$ coincides with the Lagrangian $L \subset \mathbb{C}P^{n-1}$. Let $\omega_{FS}$, $\omega_{st}$ be the Fubini-Study form and the standard symplectic form of  $\mathbb{C}P^{n-1}$ and $\mathbb{C}^{n}$, respectively. Denote
\begin{equation*}
\beta_r = \sigma(\widetilde{\beta}_r), \quad  \alpha_r = \sigma(\widetilde{\alpha}_r), \quad r=1,...,n-m.
\end{equation*}
From Theorem $\ref{fanocondition}$ we have

\begin{equation*}
\omega_{st}(\widetilde{\beta}_r) = \frac{\pi}{2C}\mu_{\widetilde{L}}( \widetilde{\beta}_r),
\end{equation*}
where $C = \frac{n}{\delta}$ (see formula $(\ref{monconstant})$). Therefore
\begin{equation*}
\omega_{FS}(\beta_r) = \int\limits_{\sigma(\widetilde{\beta}_r)}\omega_{FS} = \frac{1}{\delta}\int\limits_{\widetilde{\beta}_r}\omega_{st} = \frac{1}{\delta}\omega_{st}(\widetilde{\beta}_r) = \frac{\pi}{2n}\mu_{\widetilde{L}}(\widetilde{\beta}_r).
\end{equation*}
We know that $L \subset \mathbb{C}P^{n-1}$ is monotone and from formula $(\ref{maslovclass})$ we get
\begin{equation}\label{mainmaslovformula2}
\mu_L(\beta_r) = \frac{2n}{\pi}\omega_{FS}(\beta_r) = \frac{2n}{\pi}\frac{\pi}{2n}\mu_{\widetilde{L}}(\widetilde{\beta}_r) = \mu_{\widetilde{L}}(\widetilde{\beta}_r).
\end{equation}

\begin{lemma}\label{linearind}
The cycles $\alpha_1 = \sigma(\widetilde{\alpha}_1), \ldots ,\sigma(\widetilde{\alpha}_{n-m-1}) = \alpha_{n-m-1}$ form a basis for $H_1(L, \mathbb{Z})/(torsion) = H_1(T^{n-m-1}, \mathbb{Z}) = \mathbb{Z}^{n-m-1}$.
\end{lemma}

\begin{proof}
We have
\begin{equation*}
\begin{gathered}
\widetilde{L} \xrightarrow{\widetilde{R}}  T_{\widetilde{\Gamma}}/D_{\widetilde{\Gamma}} = T^{n-m} = \widetilde{S}_1^1 \times \ldots \times \widetilde{S}_{n-m}^1, \;\;\; \widetilde{S}_k^1 = \mathbb{R}\langle \widetilde{\varepsilon}_k \rangle/\mathbb{Z}\langle \widetilde{\varepsilon}_k \rangle, \\
L \xrightarrow{\mathcal{R}} T_{\Gamma}/D_{\Gamma} = T^{n-m-1} = S_1^1 \times \ldots \times S_{n-m-1}^1, \;\; \; S_k^1 = \mathbb{R}\langle \varepsilon_k \rangle/\mathbb{Z}\langle \varepsilon_k \rangle
\end{gathered}
\end{equation*}
Since $\widetilde{\varepsilon}_1, \ldots,\widetilde{\varepsilon}_{n-m}$ form a basis for $\widetilde{\Lambda}$, we see that $\widetilde{\alpha}_1, \ldots,\widetilde{\alpha}_{n-m}$ form a basis for $H_1(\widetilde{L}, \mathbb{Z}) = H_1(T^{n-m}, \mathbb{Z})$. From Lemma $\ref{basisgen}$ we get that cycles $\alpha_1, \ldots,\alpha_{n-m-1}$ define a basis for $H_1(L, \mathbb{Z})/(torsion) = H_1(T^{n-m-1}, \mathbb{Z})$.
\end{proof}

Let us pay more attention to $\widetilde{\alpha}_{n-m}$. We assumed in the beginning of the proof that $\widetilde{\varepsilon}_{n-m} = (1,0, \ldots,0)$. We have
\begin{equation}\label{notintprojection}
\begin{gathered}
\partial (\widetilde{\beta}_{n-m}) = \widetilde{\alpha}_{n-m}(s) = (v_1(s)e^{i\pi s \varphi_1}, \ldots, v_n(s)e^{i\pi s \varphi_1}), \\
\alpha_{n-m} = \sigma(\widetilde{\alpha}_{n-m}) = [v_1(s): \ldots :v_n(s)]. \\
\end{gathered}
\end{equation}
From Formula $(\ref{mainmaslovformula2})$ and Theorem $\ref{fanocondition}$ we obtain
\begin{equation}\label{cpmaslovclass}
\mu_{L}(\beta_{n-m}) = \mu_{\widetilde{L}}(\widetilde{\beta}_{n-m}) = \int\limits_{\widetilde{\alpha}_{n-m}}nd\varphi_1 = \int\limits_{0}^{1}nds = n.
\end{equation}

From the long exact sequence of the pair $(\mathbb{C}P^{n-1}, L)$ we have
\begin{equation*}
H_2(L, \mathbb{Z}) \xrightarrow{i_{*}} \mathbb{Z} \xrightarrow{p_{*}} H_2(\mathbb{C}P^{n-1}, L, \mathbb{Z}) \xrightarrow{\partial_{*}} \mathbb{Z}_2 \bigoplus H_1(T^{n-m-1}, \mathbb{Z})  \rightarrow 0.
\end{equation*}
Then, from Lemma $\ref{linearind}$ we see that the elements $\partial_{*}(\beta_r) = \alpha_r$ define the basis for  $H_1(T^{n-m-1}, \mathbb{Z})$, where $r=1,...,n-m-1$. From formula $(\ref{notintprojection})$ we get $\partial_{*}(\beta_{n-m}) \in \mathbb{Z}_2$ (belongs to $\pi_1(\mathcal{R}))$.

\begin{lemma}
We have $H_2(\mathbb{C}P^{n-1}, L, \mathbb{Z}) = \mathbb{Z}^{n-m}$ and $p_{*}([\mathbb{C}P^1]) = 2[\beta_{n-m}]$.
\end{lemma}

\begin{proof}
Assume that the homomorphism $i_{*}$ is nonzero. Then there exists $m$ such that $p_{*}[m\mathbb{C}P^1] = 0$ and $0=\mu_L(p_{*}[m\mathbb{C}P^1]) = m\mu_L(p_{*}[\mathbb{C}P^1])$.  On the other hand, we have the disc $\beta_{n-m}$ such that $\partial_{*}\beta_{n-m} \in \mathbb{Z}_2 \subset H_1(L, \mathbb{Z})$. Hence, $[\partial_{*}(2\beta_{n-m})] = 0$  and for some number $s$ we have $p_{*}(s[\mathbb{C}P^1]) = 2[\beta_{n-m}]$. So, we get $\mu_{L}(p_{*}[\mathbb{C}P^1]) = 0 \neq \mu_{L}(\beta_{n-m})$ = n. We obtain that the homomorphism $i_{*}$ is trivial and the long exact sequence has the form
\begin{equation*}
0 \rightarrow \mathbb{Z} \xrightarrow{p_{*}} H_2(\mathbb{C}P^{n-1}, L, \mathbb{Z}) \xrightarrow{\partial_{*}} \mathbb{Z}^{n-m-1} \oplus \mathbb{Z}_2 \rightarrow 0.
\end{equation*}
As we already noticed, $[\partial_{*} (2\beta_{n-m})] = 0$. Hence, there exists a number $s$ such that $sp_{*}[\mathbb{C}P^1] = 2[\beta_{n-m}]$. From Formula $(\ref{maslovclass})$ and Fromula $(\ref{cpmaslovclass})$ we obtain
\begin{equation*}
s\pi = \omega_{FS}(p_{*}[s\mathbb{C}P^1]) = \frac{\pi}{2n} \mu_{L}(p_{*}(s[\mathbb{C}P^1]) = \frac{\pi}{2n} \mu_{L}(2[\beta_{n-m}]) = 2n\frac{\pi}{2n} = \pi.
\end{equation*}
So, $s=1$ and $2[\beta_{n-m}] = p_{*}[\mathbb{C}P^1]$. Therefore,
\begin{equation*}
H_2(\mathbb{C}P^{n-1}, L, \mathbb{Z}) = \mathbb{Z}^{n-m}.
\end{equation*}

\end{proof}

So, we have the basis $\widetilde{\beta}_1,...,\widetilde{\beta}_{n-m}$ for  $H_2(\mathbb{C}^n, L, \mathbb{Z})$ and the basis $\beta_1,...,\beta_{n-m}$ for $H_2(\mathbb{C}P^{n-1}, L, \mathbb{Z})$ such that $\mu_{\widetilde{L}}(\widetilde{\beta_i}) = \mu_L(\beta_i)$ for any $i$. This implies that $N_{\widetilde{L}} = N_{L}$.

\end{proof}

In general, there is no canonical basis for $H_1(T^{n-m}, \mathbb{Z})$. So, to compute the minimal Maslov number we need to choose some basis. We find some convenient basis below and give an explicit formula for the Maslov class.

Let us assume that $\widetilde{\Lambda} = \mathbb{Z}^{n-m}$ and $\widetilde{\Lambda}_u = \widetilde{\Lambda}$. Without loss of generality,  we can assume that system $(\ref{eqmain})$ can be written in the form (we take linear combinations of equations and use that $\widetilde{\Lambda}_{u_1=\ldots u_{m-1}=0} = \widetilde{\Lambda} = \mathbb{Z}^{n-m}$)
\begin{equation}\label{simplemaslov}
\begin{gathered}
\widetilde{\mathcal{R}} = \left\{
 \begin{array}{l}
\widetilde{\gamma}_{1,1}u_1^2 + ... + \widetilde{\gamma}_{m-1,1}u_{m-1}^2 + u_{m}^2 = \delta_1, \\
\widetilde{\gamma}_{1,2}u_1^2 + ... + \widetilde{\gamma}_{m-1,2}u_{m-1}^2 + u_{m+1}^2 = \delta_2 \\
...\\
\widetilde{\gamma}_{1,n-m}u_1^2 + ... + \widetilde{\gamma}_{m-1,n-m}u_{m-1}^2 + u_{n}^2 = \delta_{n-m}
 \end{array}
\right.
\end{gathered}
\end{equation}
Then the embedding of $\widetilde{L} \subset \mathbb{C}^n$ is given by
\begin{equation}\label{simplemaslovemb}
\widetilde{L} = (u_1e^{i\pi<\widetilde{\gamma}_1, \widetilde{\varphi}>},...,u_{m-1}e^{i\pi<\widetilde{\gamma}_{m-1}, \widetilde{\varphi}>}, u_me^{i\pi\varphi_{m}},...,u_me^{i\pi\varphi_{n}}).
\end{equation}

\vspace{.1in}

\begin{lemma}\label{simplemimasnumber}
Let $\widetilde{L} \subset \mathbb{C}^{n}$ be the Lagrangian corresponding to system $(\ref{simplemaslov})$ (equivalently, the embedding of $\widetilde{L}$ is given by formula $(\ref{simplemaslovemb})$). Assume that $\widetilde{\mathcal{R}}$ is connected, $\pi_1(\widetilde{\mathcal{R}}) = 0$, and $ \pi_1(L) = \pi_1(\mathcal{R}) \bigoplus \pi_1(T^{n-m-1})$. If $\widetilde{\gamma}_1+...+\widetilde{\gamma}_n = (C\delta_1,...,C\delta_{n-m})$ for some constant $C$, then the minimal Maslov number is given by
\begin{equation*}
N_{\widetilde{L}} = gcd(C\delta_1,...,C\delta_{n-m}),
\end{equation*}
where $gcd$ stands for the greatest common divisor.
\end{lemma}

\begin{proof}
We can choose vectors
\begin{equation*}
\widetilde{\varepsilon}_r = (\underbrace{0,...,0}_{r},1,0,...,0), \quad r=0,...,n-m-1\\
\end{equation*}
as basis for $\widetilde{\Lambda}^{*}$. As in the proof of the previous theorem, let $(v_{r,1}(s),...,v_{r,n}(s))$ be the curve such that
\begin{equation*}
\begin{gathered}
(v_{r,1}(0),...,v_{r,n}(0)) = (u_1,..,u_n),  \\
(v_{r,1}(1),...,v_{r,n}(1)) = (u_1\cos(\pi<\widetilde{\gamma}_1, \widetilde{\varepsilon}_r>),...,u_n\cos(\pi<\widetilde{\gamma}_n, \widetilde{\varepsilon}_r>)).
\end{gathered}
\end{equation*}
We get loops
\begin{equation*}
\widetilde{\alpha}_r(s) = \widetilde{\psi}(v_{r,1}(s),...,v_{r,n}(s), s\widetilde{\varepsilon}_r) \subset \widetilde{L}, \;\;\; s\in [0,1].
\end{equation*}
There is $\widetilde{\beta}_r \in \pi_2(\mathbb{C}^n, \widetilde{L})$ such that $\partial(\widetilde{\beta}_r) = \widetilde{\alpha}_r$. From Theorem $\ref{fanocondition}$ we have
\begin{equation*}
\mu_{\widetilde{L}}(\widetilde{\beta}_r) = \int\limits_{\widetilde{\alpha}_r}C\delta_1d\varphi_1+...+C\delta_{n-m}d\varphi_{n-m} = C\delta_r.
\end{equation*}
We proved in the previous theorem that $\widetilde{\alpha}_1,...,\widetilde{\alpha}_{n-m}$ form basis for $H_1(\widetilde{L}, \mathbb{Z}) = H_1(T^{n-m}, \mathbb{Z}) = \pi_2(\mathbb{C}^n, L)$. 
\end{proof}

\subsection{Some lemmas}\label{somelemmas}

Suppose $j_1,...,j_{\ell}$ are positive integers  and $j_1+...+j_{\ell} = n$. Let $\overrightarrow{u}_i$ be a vector $(u_{i,1},...,u_{i, j_i})$. Define $|u_i|^2$ by
\begin{equation*}
|u_i|^2 = \sum\limits_{s=1}^{j_i}u_{i,s}^2.
\end{equation*}
Then consider a system
\begin{equation}\label{minmaslovquadric2}
\begin{gathered}
\left\{
 \begin{array}{l}
|\overrightarrow{u}_{1}|^2 + \ldots + |\overrightarrow{u}_{\ell}|^2 = \delta \\
\gamma_{1,r}|\overrightarrow{u}_{1}|^2 + \ldots + \gamma_{n,r}|\overrightarrow{u}_{{\ell}}|^2 = 0
 \end{array}
\right.
\\
r=1,\ldots, n-m-1
\end{gathered}
\end{equation}
In fact, we applied the multiwedge operation, defined in Section $\ref{wedge}$, to  system
\begin{equation*}
\begin{gathered}
\left\{
 \begin{array}{l}
u_{1}^2 + \ldots + u_{\ell}^2 = \delta \\
\gamma_{1,r}u_{1}^2 + \ldots + \gamma_{n,r}u_{\ell}^2 = 0
 \end{array}
\right.
\\
r=1, \ldots,n-m-1
\end{gathered}
\end{equation*}
System $(\ref{minmaslovquadric2})$ defines Lagrangians $\widetilde{L} \subset \mathbb{C}^n$ and $L \subset \mathbb{C}P^{n-1}$.

\begin{lemma}\label{trivialfibration}
If $j_1,...,j_{\ell}$ are even numbers, then the Lagrangian $L$ is diffeomorphic to $\mathcal{R} \times T^{n-m-1}$, where $T^{n-m-1} = T_{\Gamma}/D_{\Gamma}$. In other words, the fibration $(\ref{mainprojection})$ is trivial.

Analogously, if $j_1,...,j_{\ell}$ are even, then $\widetilde{L}$ is diffeomorphic to $\widetilde{\mathcal{R}} \times T^{n-m}$, where $T^{n-m} = T_{\widetilde{\Gamma}}/D_{\widetilde{\Gamma}}$.
\end{lemma}
\begin{proof}
Note that
\begin{equation*}
(\cos(\phi)u_{s,i} - \sin(\phi)u_{s+1,i})^2 + (\sin(\phi)u_{s,i} + \cos(\phi)u_{s+1,i})^2 = u_{s,i}^2 + u_{s+1,i}^2
\end{equation*}
and there is an isotopy connecting $\overrightarrow{u}_i$ with $-\overrightarrow{u}_i$
\begin{equation*}
\begin{gathered}
(\cos(\phi)u_{1,i} - \sin(\phi)u_{2,i}, \sin(\phi)u_{1,i} + \cos(\phi)u_{2,i}, \ldots, \cos(\phi)u_{j_{i-1},i} - \sin(\phi)u_{j_i,i}, \\
 \sin(\phi)u_{j_{i-1},i} + \cos(\phi)u_{j_i,i}), \quad \;\; \phi \in [0,\pi].
\end{gathered}
\end{equation*}
Therefore, all transformation maps of the fibration are isotopic to identity.

We know that
\begin{equation*}
\begin{gathered}
T_{\Gamma} = \widetilde{S}_1^1 \times \ldots \times \widetilde{S}^1_{n-m-1}, \;\; \; \widetilde{S}^1_i = \mathbb{R}\langle\varepsilon_i\rangle/\mathbb{Z}\langle 2\varepsilon_i \rangle, \\
T^{n-m-1} = T_{\Gamma}/D_{\Gamma} = S_{1}^{1} \times \ldots\times S^1_{n-m-1}, \;\; \; S^1_i = \mathbb{R}\langle\varepsilon_i\rangle/\mathbb{Z}\langle \varepsilon_i \rangle.
\end{gathered}
\end{equation*}
Let us consider a fibration over $S^1$ defined by composition of projections $\pi_1: L \rightarrow T^{n-m-1} \rightarrow S_1^1$. Since all transformation maps of the fibration are isotopic to identity, we get that  $\pi_1: L \rightarrow S_1^1$ is trivial. This means that $L = L_1 \times S^1_1$, where $L_1$ fibers over $T^{n-m-2} = S_2^2 \times ... \times S_{n-m-1}^1$ with fiber $\mathcal{R}$.

In the same way, we can prove that $L_1 = L_2 \times S^1$, where $L_2$ fibers over $T^{n-m-3} = S_3^1 \times ... \times S_{n-m-1}^1$. Repeating this $n-m-1$ times we get $L = \mathcal{R} \times T^{n-m-1}$.
\\
\\
The same proof works for $\widetilde{L} \subset \mathbb{C}^n$.
\end{proof}

Let us study the fundamental group of Lagrangians.

\begin{lemma}\label{fundgrouplagr}
If $j_1,...,j_n \geqslant 2$, $\pi_1(\widetilde{\mathcal{R}}) = 0$, and $\gamma_1+...+\gamma_n = 0$, then $\pi_1(L) = \mathbb{Z}_2 \bigoplus \mathbb{Z}^{n-m-1}$.
\end{lemma}

\begin{proof}
If $\pi_1(\widetilde{\mathcal{R}}) = 0$, then $\pi_1(\mathcal{R}) = \mathbb{Z}_2$. From the exact sequence of the fibration $L \xrightarrow{\mathcal{R}} T^{n-m-1}$ we have
\begin{equation*}
0 \rightarrow \mathbb{Z}_2 \rightarrow \pi_1(L) \rightarrow \mathbb{Z}^{n-m-1} \rightarrow 0.
\end{equation*}
Note that under our conditions the fibartion $L \xrightarrow{\mathcal{R}} T^{n-m-1}$ has a section
\begin{equation*}
\begin{gathered}
s(\varphi) = \\
 (\underbrace{\cos(\pi<\gamma_1 - \gamma_{\ell}, \varphi>), \sin(\pi<\gamma_1 - \gamma_{\ell}, \varphi>,0,...0}_{j_1},...,\\
\underbrace{\cos(\pi<\gamma_{\ell-1} - \gamma_{\ell}, \varphi>), \sin(\pi<\gamma_{\ell-1} - \gamma_{\ell}, \varphi>, 0,...,0}_{j_{\ell}}, \overrightarrow{u}_{\ell}, \varphi), \\
\overrightarrow{u}_{\ell} = const.
\end{gathered}
\end{equation*}
Hence, the exact sequence above splits and $\pi_1(L) = \mathbb{Z}_2 \rtimes \mathbb{Z}^{n-m-1}$, where $\rtimes$ stands for the semidirect product. Easy to see that  $\mathbb{Z}_2 \rtimes \mathbb{Z}^{n-m-1}$ is isomorphic to $\mathbb{Z}_2 \oplus \mathbb{Z}^{n-m-1}$.
\end{proof}

Let us mention the following almost obvious lemma. Denote by $\widetilde{\mathcal{R}}_P^{(j_1,\ldots, j_n)}$ the real moment-angle manifold obtained from $\widetilde{\mathcal{R}}_P$ by applying the multiwedge operation $(j_1,\ldots, j_n)$ (see Section $\ref{wedge}$). Let $P^{(j_1,\ldots,j_n)}$ be the polytope associated to $\widetilde{\mathcal{R}}_P^{(j_1,\ldots, j_n)}$.

\begin{lemma}\label{multiwedgedelzant}
If $P$ is Delzant, then $P^{(j_1, \ldots, j_n)}$ is Delzant.
\end{lemma}

\begin{proof}
Let $\widetilde{\Lambda}$ and $\widetilde{\Lambda}^{(j_1,\ldots, j_n)}$ be the lattices generated by columns of $\widetilde{\mathcal{R}}_P$ and $\widetilde{\mathcal{R}}_P^{(j_1,\ldots, j_n)}$, respectively (see Section $\ref{maincostruction0}$). We see that $\widetilde{\Lambda} = \widetilde{\Lambda}^{(j_1,\ldots, j_n)}$ (because we repeat columns). Since $P$ is Delzant, we have $\widetilde{\Lambda} = \widetilde{\Lambda}_u$ for any $u \in \widetilde{\mathcal{R}}_P$ (see Theorem $\ref{embdelzant0})$. Easy to see that $\widetilde{\Lambda}^{(j_1,\ldots, j_n)}_u = \widetilde{\Lambda}^{(j_1,\ldots, j_n)}$ for any $u \in \widetilde{\mathcal{R}}_P^{(j_1,\ldots, j_n)}$ (because we have the same repeated columns). Then, Theorem $\ref{embdelzant0}$ says that $P^{(j_1, \ldots, j_n)}$ is Delzant.
\end{proof}

\vspace{.1in}

\emph{Proof of Lemma $\ref{helplemma}$}. Assume that $\widetilde{\mathcal{R}}_P = \widetilde{\mathcal{Z}}_Q$. This means that $P = Q^{(2,\ldots,2)}$. By the construction of the multiwedge operation, the polytope $Q$ has dimension $\frac{m}{2}$ and has $\frac{n}{2}$ facets. It follows from Lemma $\ref{trivialfibration}$ that $L = \widetilde{\mathcal{R}}_P/\mathbb{Z}_2 \times T^{n-m-1}$. Arguing as in Lemma $\ref{trivialfibration}$, we see that the action of $\mathbb{Z}_2$ on $\widetilde{\mathcal{R}}_P$ is isotopic to identity; therefore the action preserves the orientation. So, $L$ is orientable.

We mentioned that $\mathcal{R}_P = \mathcal{R}_P/\mathbb{Z}_2$ can be defined by the system of quadrics $(\ref{projectiveq})$ in $\mathbb{R}P^{n-1}$. Therefore, the normal bundle of $\mathcal{R}_P \subset \mathbb{R}P^{n-1}$ has $n-m-1$ nonzero sections. We get that the normal bundle of $\mathcal{R}_P \subset \mathbb{R}P^{n-1}$ is trivial and $T\mathcal{R} \oplus \mathbb{R}^{n-m-1} = T\mathbb{R}P^{n-1}$. We have $w_2(\mathcal{R}) = \frac{n(n-1)}{2} \; mod \; 2$. This proves the Lemma.

\subsection{Proof of Theorem $\ref{vanishquantum}$}

Let $P$ be  Delzant and Fano polytope. Let $a_1,\ldots, a_n$ be the vectors normal to the facets of $P$. We assume that $a_1+\ldots + a_n = 0$. From Section $\ref{maincostruction}$ we know that there is embedded monotone Lagrangian $L \subset \mathbb{C}P^{n-1}$ associated to $P$.

Let us use notations from Section $\ref{wedge}$. We suppose that $P = Q^{(2,\ldots, 2)}$. Therefore, $\widetilde{\mathcal{R}}_P = \widetilde{\mathcal{Z}}_Q$.  By Lemma $\ref{helplemma}$, the Lagrangian $L$ is diffeomorphic to $\widetilde{\mathcal{R}}_P/\mathbb{Z}_2 \times T^{n-m-1}$. Moreover, $L$ is orientable and the action of $\mathbb{Z}_2$ on $\widetilde{\mathcal{R}}_P$ is isotopic to identity.

Let us study the face ring of $Q$ and the complex $[R(Q), d]$ defined in Section $\ref{torsiondefinitions}$. Denote the facets of $Q$ by $v_1,\ldots,v_{\frac{n}{2}}$ and denote the exterior variables of  $[R(Q), d]$ by $y_1,\ldots,y_{\frac{n}{2}}$. Let us recall that
\begin{equation*}
m(Q) = min\{k \in \mathbb{N} \; | \; v_{i_1}\cap \ldots \cap v_{i_k} = \emptyset  \}.
\end{equation*}

Suppose that there exists a nontrivial element $a \in H^{r}(\widetilde{\mathcal{Z}}_Q, \mathbb{Z})$, where $0 < r < 2m(Q) -1$. Then, there are facets $v_1,\ldots, v_{\ell}$ such that $v_1\cap \ldots \cap v_{\ell} = \emptyset$. From our definition it follows that $\ell \geqslant m(Q)$. This contradiction proves that $H^{r}(\widetilde{\mathcal{Z}}_Q, \mathbb{Z}) = 0$ for $0 < r<2m(Q)-1$.

The manifold $\widetilde{\mathcal{R}}_P$ is given by system $(\ref{eqmain})$. Since $\widetilde{\mathcal{R}}_P = \widetilde{\mathcal{Z}}_Q=\widetilde{\mathcal{R}}_Q^{(2,\ldots,2)}$, we see that system $(\ref{eqmain})$ can be written in the following form:

\begin{equation}\label{quantumvanisheq1}
\begin{gathered}
\widetilde{\mathcal{R}}_P = \left\{
 \begin{array}{l}
u_1^2 + ... + u_n^2 = \delta \\
\gamma_{1,r}(u_1^2 + u_2^2) + \gamma_{2,r}(u_3^2 + u_4^2) + ... + \gamma_{\frac{n}{2},r}(u_{n-1}^2 + u_n^2) = 0
 \end{array}
\right.
\\
r=1,...,n-m-1.
\end{gathered}
\end{equation}

Denote $\widetilde{\mathcal{R}}_P/\mathbb{Z}_2$ by $\mathcal{R}_P$.  The embedding of $\mathcal{R}_P \times T^{n-m-1}$ into $\mathbb{C}P^{n-1}$ is given by the formula
\begin{equation*}\label{quantumvanisheq}
\begin{gathered}
\mathcal{N} = (\mathcal{R}_P \times T_{\Gamma})/{D_{\Gamma}},  \;\;\; \psi: \mathcal{N} \rightarrow \mathbb{C}P^{n-1},\\
L =  \psi(u_1, u_2,...,u_{n-1}, u_{n}, \varphi) = \\
[u_{1}e^{i\pi(\gamma_1,\varphi)}: u_{2}e^{i\pi(\gamma_1,\varphi)}:...:u_{n-1}e^{i\pi(\gamma_{\frac{n}{2}}, \varphi)}: u_{n}e^{i\pi(\gamma_{\frac{n}{2}}, \varphi)}].
\end{gathered}
\end{equation*}
Denote homogeneous coordinates of $\mathbb{C}P^{n-1}$ by $w_1,\ldots,w_n$. Let $H_i$ be a hyperplane defined by $w_i + iw_{i+1} = 0$. If $L \cap H_1 \neq \emptyset$, then
\begin{equation*}
\begin{gathered}
L \cap H_1 = [0:0:u_{3}e^{i\pi(\gamma_2,\varphi)}: u_{4}e^{i\pi(\gamma_2,\varphi)}:...:u_{n-1}e^{i\pi(\gamma_{\frac{n}{2}} :\varphi)}: u_{n}e^{i\pi(\gamma_{\frac{n}{2}}, \varphi)}]
\end{gathered}
\end{equation*}
Since $P$ is Delzant, we have that $\Lambda_u = \Lambda$ for any $u \in \widetilde{\mathcal{R}}_P$ (see Lemma $\ref{embdelzant0}$). Therefore, we can argue as in Lemma $\ref{trivialfibration}$ and show that
\begin{equation}\label{diffintersect}
L \cap H_1 = (\mathcal{R}_P\cap \{u_1=u_2 =0\})\times T^{n-m-1} \;\;\; or \;\;\; \emptyset.
\end{equation}
Denote by $h^0 \in H^{2}(L, \mathbb{Z}_2)$ the Poincare dual to $[L \cap H_1] \in H_{n-3}(L, \mathbb{Z}_2)$. It follows from formula $(\ref{diffintersect})$ that
\begin{equation*}
\begin{gathered}
[L \cap H_1] \in H_{m-2}(\mathcal{R}_P, \mathbb{Z}_2) \times [T^{n-m-1}] \subset H_{n-3}(L, \mathbb{Z}_2) = H_{n-3}(\mathcal{R}_P \times T^{n-m-1}, \mathbb{Z}_2) \; \Rightarrow \\
h^{0} \in H^2(\mathcal{R}_P, \mathbb{Z}_2) \times 1 \subset H^2(L, \mathbb{Z}_2).
\end{gathered}
\end{equation*}

\begin{lemma}\label{generatorzero}
We have $(h^0)^{m(Q)} = 0$.
\end{lemma}

\begin{proof}
Assume that the polytope $Q$ is given by a system
\begin{equation}\label{vanishquantumpol}
Q=\{x\in \mathbb{R}^{\frac{m}{2}}: <\hat{a}_i,x> + \hat{b}_i \geqslant  0 \quad \;\;\; i=1,...,\frac{n}{2}  \}.
\end{equation}
Denote $\widetilde{\mathcal{Z}}_Q/\mathbb{Z}_2$ by $\mathcal{Z}_Q$.  System $(\ref{quantumvanisheq1})$ can be written in the form
\begin{equation*}
\begin{gathered}
\widetilde{\mathcal{R}}_P = \widetilde{\mathcal{Z}}_Q = \left\{
 \begin{array}{l}
|z|_1^2 + ... + |z|_{\frac{n}{2}}^2 = \delta \\
\gamma_{1,r}|z|^2_1 + \gamma_{2,r}|z|_2^2 + ... + \gamma_{\frac{n}{2},r}|z|_{\frac{n}{2}}^2 = 0
 \end{array}
\right.
\end{gathered}
\end{equation*}
Then $\widetilde{\mathcal{R}}_P \cap H_1 = \widetilde{\mathcal{Z}}_Q \cap H_1 = \widetilde{\mathcal{Z}}_Q \cap \{z_1 = 0\}$. By our construction (see Section $\ref{torsiondefinitions}$), the variable $z_1$ corresponds to the first inequality of system $(\ref{vanishquantumpol})$, i.e. the complex moment angle manifold associated to the polytope $Q \cap \{<\hat{a}_1,x> + \hat{b}_1 =  0\}$ is $\widetilde{\mathcal{Z}}_Q \cap \{z_1 = 0\}$. Since $Q$ is irredundant, we have $Q \cap \{<\hat{a}_1,x> + \hat{b}_1 =  0\} \neq \emptyset$ and $\widetilde{\mathcal{Z}}_Q \cap \{z_1 = 0\} \neq \emptyset$.   Let us recall that we denote facets of $Q$ by $v_1, \ldots, v_{\frac{n}{2}}$. Arguing in the same way, we see that the complex moment angle manifold associated to the intersection of facets $v_1\cap\ldots\cap v_r$ is $\widetilde{\mathcal{Z}}_Q \cap \{z_1 = 0\}\cap \ldots \cap \{z_r = 0\}$.

Without loss of generality, assume that $v_{1} \cap \ldots \cap v_{m(Q)} = \emptyset$. This implies that
\begin{equation*}
\widetilde{\mathcal{Z}}_Q \cap \{z_1=\ldots z_{m(Q)} = 0\} = \emptyset \Rightarrow \mathcal{Z}_Q \cap \{z_1=\ldots z_{m(Q)} = 0\} = \emptyset.
\end{equation*}

Let us consider $H_{*}(L, \mathbb{Z}_2)$ equipped with the intersection product $\cap$. Since $L\cap H_1, L \cap H_3, \ldots, L\cap H_{n-1}$ represent the same homology class in $H_{n-3}(L, \mathbb{Z}_2)$, we get
\begin{equation*}
\begin{gathered}
\underbrace{[L\cap H_1] \cap \ldots \cap [L\cap H_1]}_{m(Q)} = [L\cap H_1] \cap [L\cap H_3]\cap \ldots \cap [L\cap H_{2m(Q)-1}]] = \\
[(\mathcal{Z}_Q \cap \{z_1 = \ldots = z_{m(Q)} = 0\}) \times T^{n-m-1}] = \emptyset.
\end{gathered}
\end{equation*}
Since, $h^{0}$ is the Poincare dual to $L\cap H_1$ and we work over $\mathbb{Z}_2$, we obtain that $(h^0)^{m(Q)} = 0$

\end{proof}

Let us use Lemma $\ref{zeroquantum}$ to prove that $QH^{*}(L, \mathbb{Z}_2[T, T^{-1}]) = 0$. Since the action of $\mathbb{Z}_2$ is free on $\widetilde{\mathcal{R}}_P$ and the Eilenberg-MacLane space $K(\mathbb{Z}_2, 1) = \mathbb{R}P^{\infty}$ we get the Cartan-Leray spectral sequence (see \cite[p. 206]{Weibel})
\begin{equation*}
\begin{gathered}
E_2^{p,q}= H^{p}(\mathbb{R}P^{\infty}, H^{q}(\widetilde{\mathcal{R}}_P, \mathbb{Z}_2)) \Rightarrow H^{*}(\mathcal{R}_P, \mathbb{Z}_2),\\
\mathcal{R}_P = \widetilde{\mathcal{R}}_P/\mathbb{Z}_2
\end{gathered}
\end{equation*}

The part of $E_{2m(Q)}$  is shown in Figure $\ref{fig:s1}$.
\begin{figure}[h]
\centering
\includegraphics[width=0.6\linewidth]{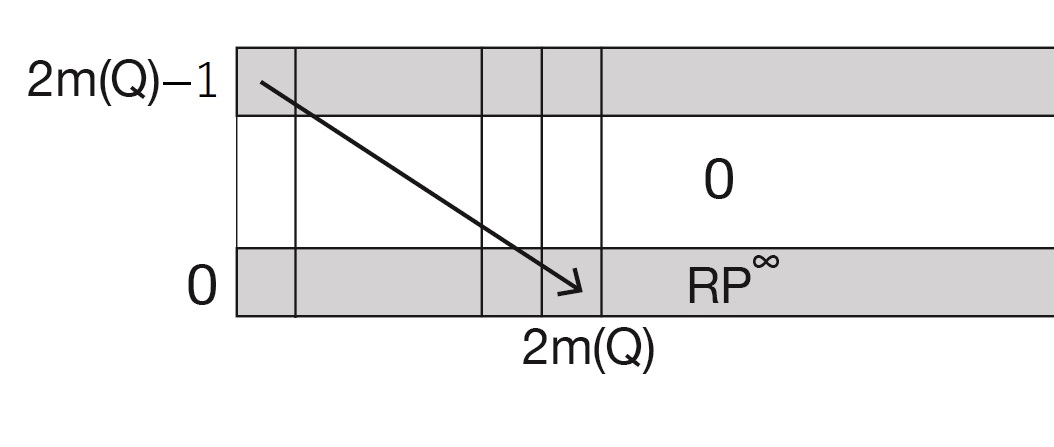}
\caption{$E_{2m(Q) }$ page}
  \label{fig:s1}
\end{figure}

Let us denote by $\alpha$ the generator of $H^{*}(\mathbb{R}P^{\infty}, \mathbb{Z}_2)$, i.e. $H^{*}(\mathbb{R}P^{\infty}, \mathbb{Z}_2) = \mathbb{Z}_2[\alpha]$, where $deg(\alpha) = 1$. Since $\partial_{2m(Q) }$ is the only differential that can kill $E_{2m(Q)}^{2m(Q),0}$, we get that

\begin{equation*}
E_{\infty}^{2m(Q),0} = \mathbb{Z}_2 \;\; \Leftrightarrow \;\; \partial_{2m(Q) }(E_{2m(Q)}^{0,2m(Q)-1}) = 0
\end{equation*}

\begin{lemma}
If $\partial_{2m(Q) }(E_{2m(Q)}^{0,2m(Q)-1}) = 0$, then $QH^{*}(L, \mathbb{Z}_2[T, T^{-1}]) = 0$.
\end{lemma}

\begin{proof}
Assume that $h^0 = \alpha$. If $\partial_{2m(Q) }(E_{2m(Q)}^{0,2m(Q)-1}) = 0$, then $\alpha^{2m(Q)}$ is well defined nonzero element of $H^{2m(Q)}(\mathcal{R}_P, \mathbb{Z}_2)$. On the other hand, from Lemma $\ref{generatorzero}$ we get that $(h^0)^{m(Q)} = 0$. Since $H^{2}(\mathcal{R}_P, \mathbb{Z}_2) = \mathbb{Z}_2$, we see that $h^{0} = 0$. Note that $m(Q)>1$ because $Q$ is irredundant.

It follows from Lemma $\ref{zeroquantum}$ that $QH^{*}(L, \mathbb{Z}_2[T, T^{-1}]) = 0$ ( we assumed that $N_L > 2$).
\end{proof}

Let us consider cohomology ring with $\mathbb{Z}$ coefficients. Let us denote by $h^{0} \in H^2(\mathcal{R}_P, \mathbb{Z})$ the Poincate dual to $[L \cap H_1] \in H_{n-3}(L, \mathbb{Z})$. Since $Q$ is irredundant, we see from Theorem $\ref{torcohom}$ that $H^1(\widetilde{\mathcal{R}}_P, \mathbb{Z}) = H^2(\widetilde{\mathcal{R}}_P, \mathbb{Z}) = 0$. From the Cartan-Leray spectral sequence
\begin{equation*}
\begin{gathered}
E_2^{p,q}= H^{p}(\mathbb{R}P^{\infty}, H^{q}(\widetilde{\mathcal{R}}_P, \mathbb{Z})) \Rightarrow H^{*}(\mathcal{R}_P, \mathbb{Z})
\end{gathered}
\end{equation*}
we get that $H^{2}(\mathcal{R}_P, \mathbb{Z}) = \mathbb{Z}_2$.  Therefore, $h^0 \in \mathbb{Z}_2$. From Section $\ref{zeroquantum}$ we know that there is the isomorphism $h: QH^{*}(L, \mathbb{Z}[T, T^{-1}]) \rightarrow QH^{*+2}(L, \mathbb{Z}[T, T^{-1}]) = 0$. Moreover, $h(x^{min}) = h^0$ whenever $N_L > 2$, where $x^{min}$ is a unique minimum of a generic Morse function.  Since $H^{0}(L, \mathbb{Z}) = \mathbb{Z}$ and $2h^{0} = 0$, we get that there exists $x$ such that $d_F^{f}(x) = \ell x^{min}$, where $\ell \in \{1, 2\}$ and $d_F^{f}$ is the Floer differential. So, $QH^{*}(L, \mathbb{Z}[T, T^{-1}])$ is never isomorphic to $H^{*}(L, \mathbb{Z}[T, T^{-1}])$.
\\

Let $G$ be a ring in which $2$ is invertible. Then $H^{r}(\mathbb{R}P^{\infty}, G) = 0$ for all $r>0$. Therefore, from the Cartan-Leray spectral sequnce
\begin{equation*}
\begin{gathered}
E_2^{p,q}= H^{p}(\mathbb{R}P^{\infty}, H^{q}(\widetilde{\mathcal{R}}_P, G)) \Rightarrow H^{*}(\mathcal{R}_P, G)
\end{gathered}
\end{equation*}
we see that $H^{2}(\mathcal{R}_P, G) = 0$. Hence, the element $h^{0} = 0$. It follows from Lemma $\ref{zeroquantum}$  that $QH^{*}(L, G[T, T^{-1}]) = 0$.

\subsection{Proof of Theorem $\ref{twospheretheorem}$}

Let $P$ be product of two standard simplices of dimensions $p-1$ and $n-p-1$, i.e. $P$ is defined by inequalities
\begin{equation}\label{twoprodpolytope}
\begin{gathered}
\left\{
 \begin{array}{l}
x_1 + 1 \geqslant 0 \\
\ldots \\
x_{p-1} +1 \geqslant 0 \\
-x_1 - \ldots - x_{p-1} + 1 \geqslant 0 \\
x_p +1 \geqslant 0 \\
\ldots \\
x_{n-2} + 1 \geqslant 0 \\
-x_p - \ldots - x_{n-2} + 1 \geqslant 0
 \end{array}
\right.
\end{gathered}
\end{equation}

Then, the real moment-angle manifold  $\widetilde{\mathcal{R}}_P \subset \mathbb{R}^{n}$ associated to $P$ is given by
\begin{equation}\label{sphereprod}
\begin{gathered}
\left\{
 \begin{array}{l}
u_1^2 + \ldots + u_{n}^2 = n \\
(n-p)u_1^2 + \ldots +(n-p)u_p^2 - pu_{p+1}^2 - ... - pu_{n}^2 = 0
 \end{array}
\right.
\\
\Leftrightarrow
\\
\left\{
 \begin{array}{l}
u_1^2 + \ldots + u_p^2 = p \\
u_{p+1} + \ldots u_{n}^2 = n-p
 \end{array}
\right.
\end{gathered}
\end{equation}
We see that $\widetilde{\mathcal{R}}_P = S^{p-1} \times S^{n-p-1}$.  We define
\begin{equation*}
\begin{gathered}
\mathcal{R}_P = (S^{p-1} \times S^{n-p-1})/\mathbb{Z}_2, \\
(x,y) \sim (-x,-y), \quad x \in S^{p-1}, \;\; y \in S^{n-p-1}.
\end{gathered}
\end{equation*}

Easy to see that $\mathcal{R}_P$ is orientable if and only if $p-1$, $n-p-1$ are both either odd, or even. In other words, $\mathcal{R}_P$ is orientable if and only if $n$ is even.

Since the action of $\mathbb{Z}_2$ is free on the first factor of $S^{p-1} \times S^{n-p-1}$, we see that $\mathcal{R}_P$ fibers over $\mathbb{R}P^{p-1}$, where the fiber is $S^{n-p-1}$, i.e.
\begin{equation*}
\begin{gathered}
\pi: \mathcal{R}_P \xrightarrow{S^{n-p-1}} \mathbb{R}P^{p-1}, \\
\pi(u_1,,,.u_n) = [u_1: \ldots :u_p],
\end{gathered}
\end{equation*}

\begin{lemma}\label{cohomtwosphere}
Without loss of generality, assume that $p \leqslant n-p$. If $p=n-p = 2$, then $\mathcal{R} = T^2$. If $p = 2$ and $n - p\geqslant 3$, then $\pi_1(\mathcal{R}_P) = \mathbb{Z}$. If $p \geqslant 3$, then $\pi_1(\mathcal{R}_P) = \mathbb{Z}_2$. Moreover,
\begin{equation*}
\begin{gathered}
H^{*}(\mathcal{R}_P, \mathbb{Z}_2) = \mathbb{Z}_2[x, y]/(x^p, x^{p-1}y + y^2, y^3)  \;\; if \;\; p=n-p \;\; and \;\; p = 1 \;\; mod \;\; 2 \\
H^{*}(\mathcal{R}_P, \mathbb{Z}_2) = H^{*}(\mathbb{R}P^{p-1}, \mathbb{Z}_2) \otimes H^{*}(S^{n-p-1}, \mathbb{Z}_2) = \mathbb{Z}_2[x, y]/(x^p, y^2) \;\;\; otherwise,
\end{gathered}
\end{equation*}
where $deg(x) = 1$, $deg(y) = n-p-1$.
\end{lemma}

\textbf{Example.} If $p=n-p=3$, then $\mathcal{R}_P = (S^2 \times S^2)/\mathbb{Z}_2 $ is diffeomorphic to Grassmanian $Gr(2,4)$. It is known (see ~\cite[Chapter $7$, Problem $7-$B]{Milnor}) that $H^{*}(Gr(2,4), \mathbb{Z}_2)$ is generated by Stifel-Whitney classes $w_1, w_2$ and their dual classes $\bar{w}_1, \bar{w}_2$. They satisfy the following condition:
\begin{equation*}
\begin{gathered}
(1+w_1+w_2)(1+ \bar{w}_1 +\bar{w}_2 ) = 1 \Leftrightarrow \\
w_1 = \bar{w}_1, \; \; w_1^2 = w_2 + \bar{w}_2, \; \; w_1(w_2 + \bar{w}_2) =0, \; \; w_2\bar{w}_2 = 0.
\end{gathered}
\end{equation*}
If we denote $x=w_1$, $y=w_2$, then the conditions above are equivalent to
\begin{equation*}
\overline{w}_2 = x^2 + y, \;\; 0 =  w_1(w_2 + \bar{w}_2) = x^3, \;\; 0 = w_2\bar{w}_2= y(x^2 + y) = x^2y + y^2
\end{equation*}

If $p=n-p=4$, Then $(S^3 \times S^3)/\mathbb{Z}_2$ is diffeomorphic to $SO(4) = SO(3) \times S^3 = \mathbb{R}P^3 \times S^3$.

\begin{proof}
Let us recall that we have the firbration
\begin{equation*}
\begin{gathered}
\pi: \mathcal{R}_P \xrightarrow{S^{n-p-1}} \mathbb{R}P^{p-1}, \\
\pi(u_1,,,.u_n) = [u_1: \ldots :u_p].
\end{gathered}
\end{equation*}

If $p=n-p=2$, then $\mathcal{R}_P$ is orientable with zero Euler characteristic. Therefore, $\mathcal{R}_P$ is diffeomorphic to $2-$torus.

If $p=2$ and $n-p \geqslant 3$, then from the exact sequence of the fibration we get $\pi_1(\mathcal{R}_P) = \pi_1(S^1) = \mathbb{Z}$.

If $p \geqslant 3$, $n-p \geqslant 3$, then from the exact sequence of the fibration we obtain $\pi_1(\mathcal{R}_P) = \pi_1(\mathbb{R}P^{p-1}) = \mathbb{Z}_2$.

Instead of cohomology ring let us study homology groups equipped with the intersection product $\cap$. The homology Serre spectral sequence of the fibration $\pi: \mathcal{R}_P \rightarrow \mathbb{R}P^{p-1}$ collapses at $E_2$. Here we consider homology groups over $\mathbb{Z}_2$. In our case, all homology groups  are isomorphic to $\mathbb{Z}_2$ and there is no nontrivial action of the fundamental group. Then, we have

\begin{equation*}
H_{*}(\mathcal{R}_P, \mathbb{Z}_2) =  H_{*}(\mathbb{R}P^{p-1}, \mathbb{Z}_2) \otimes H_{*}(S^{n-p-1}, \mathbb{Z}_2)
\end{equation*}
as groups, not rings (with respect to the intersection product). Let $i_{r}: \mathbb{R}P^{r} \hookrightarrow \mathbb{R}P^{p-1}$ be an embedding given by
\begin{equation*}
i_r([u_1:...:u_{r+1}]) = [u_1:...:u_{r+1}:0...0].
\end{equation*}
Note that the fibration $i_{r}^{*}\mathcal{R}_P$ has a section defined by the following formula:
\begin{equation*}
\begin{gathered}
s_{r} : \mathbb{R}P^{r} \rightarrow i_{r}^{*}\mathcal{R}_P \\
s_{r}([u_1:...:u_{r+1}]) = (\underbrace{u_1\sqrt{p},...,u_{r+1}\sqrt{p}, 0,...,0}_p, u_1\sqrt{n-p},...,u_{r+1}\sqrt{n-p},0...,0).
\end{gathered}
\end{equation*}
Let us define the following closed submanifolds of $\mathcal{R}_P$:
\begin{equation*}
\begin{gathered}
A_{r} = \pi^{-1}([0:...:0,u_{r+1}:...:u_p]), \;\;\; B_{r} = s_r([u_1:...:u_{r+1}:0...:0]), \\
r=0,...,p-1.
\end{gathered}
\end{equation*}
Note that $A_0 = \mathcal{R}_P$, $A_{p-1}$ is the fiber $S^{n-p-1}$, $B_0$ is a point. We have $dim(A_{r}) = p-r - 1 + n-p-1 = n-r-2$ and $dim(B_{r}) = r$. Then, $dim(A_{r}) = codim(B_{r})$ (recall that $dim(\mathcal{R}_P) = n-2)$ and $\pi(A_r)$ intersects $\pi(B_{r})$ at a single point $[0:,...,0,u_{r+1},0,...,0]$. Since $B_r$ is section, we get that $A_r$ intersects $B_r$ at a single point. Moreover, their intersection is transversal. As a result, we obtain that $A_r$ and $B_r$ define nontrivial elements in homology groups. Comparing dimensions we see that elements $A_0,...,A_{p-1}, B_0,...,B_{p-1}$ form an additive basis for $H_{*}(\mathcal{R}_P, \mathbb{Z}_2)$.

Let us study the ring structure. Let $M_r, M_{s} \hookrightarrow \mathbb{R}P^{p-1}$ be transversally embedded $\mathbb{R}P^{p-r-1}$ and  $\mathbb{R}P^{p-s-1}$, respectively. Therefore, the intersection of $\pi^{-1}(M_r)$ and $\pi^{-1}(M_s)$ is transversal too. Arguing as before, we can prove that $\pi^{-1}(M_r)$ and $\pi^{-1}(M_s)$ represent nontrivial elements in homology group (we need to intersect it with appropriately modified section $B_r$). Since $H_r(\mathcal{R}_P, \mathbb{Z}_2)$ has dimension $1$ for $r<p$, we see that  $\pi^{-1}(M_r)$ and $A_r$ represent the same homology class. So,
\begin{equation*}
[A_r] \cap [A_{s}] = [\pi^{-1}M_r] \cap [\pi^{-1}M_s]  = [A_{r+s}].
\end{equation*}
Here we assume that $[A_{r+s}] = 0$ whenever $r+s >p$. By definition we see that $\pi(A_r) \cap \pi(B_s) = \mathbb{R}P^{s-r}$ and
\begin{equation*}
[A_r] \cap [B_s] = [B_{s-r }].
\end{equation*}
In the formula above we assume that $[B_{s-r}] = 0$ whenever $s-r <0$. Moreover, $dim([B_r] \cap [B_s]) = r + s - (n-2) \leqslant (p-1) + (p-1) - (n-2)  \leqslant 2p - n  \leqslant 0$.

Assume that $p< n-p$. Then $dim([B_r] \cap [B_s]) < 0$. Hence, $[B_r] \cap [B_s] = 0$.

Suppose $p=n-p$. Then $[B_r] \cap [B_s] = 0$ whenever $r+s < 2p - 2$. Let us find $[B_{p-1}] \cap [B_{p-1}]$.  Consider the composition of diagonal embedding and the projection
\begin{equation*}
S^{p-1} \xrightarrow{\Delta} S^{p-1} \times S^{p-1} \xrightarrow{pr} (S^{p-1} \times S^{p-1})/\mathbb{Z}_2 = \mathcal{R}_P.
\end{equation*}
We see that $pr(\Delta(S^{p-1})) = B_{p-1}$ and the tangent bundle of $B_{p-1}$ is isomorphic to the normal bundle. Therefore,
\begin{equation*}
\# ([B_{p-1}] \cap [B_{p-1}]) \; \; mod \; \; 2 \; = \; \chi(B_{p-1}) \;\; mod \;\; 2,
\end{equation*}
where $\chi$ is the Euler characteristic. Since $B_{p-1}$ is section of the fibration, we have $\chi(B_{p-1}) = \chi(\mathbb{R}P^{p-1}) $. So,  $[B_{p-1}] \cap [B_{p-1}]= 1$ if $p$ is odd and $[B_{p-1}] \cap [B_{p-1}]=0$ if $p$ is even. In other words,
\begin{equation*}
[B_{p-1}] \cap [B_{p-1}] = (p \;\; mod  \;\; 2)[B_0].
\end{equation*}

We get that elements $[\mathcal{R}_P] = [A_0]$, $[A_1]$, $[B_{p-1}]$, $[B_0]$ generate $(H_{*}(\mathcal{R}_P, \mathbb{Z}_2), \cap)$ as a ring (with respect to the intersection product). Let $x$, be the Poincare dual to $[A_1]$ and $y$ be the Poincare dual to $B_{p}$.  The lemma is proved.

\end{proof}
Let us use notations from Section $\ref{maincostruction}$. We have
\begin{equation*}
\begin{gathered}
\gamma_1=\ldots =\gamma_p = n-p, \quad \gamma_{p+1} = \ldots =\gamma_n = - p, \\
\gamma_1 - \gamma_n = \ldots = \gamma_p - \gamma_n = n, \;\; \gamma_{p+1} - \gamma_n = \ldots = \gamma_{n-1} - \gamma_n = 0,
\end{gathered}
\end{equation*}
The lattice $\Lambda \subset \mathbb{R}_P$ is generated by  number $n$ and the dual lattice $\Lambda^{*}$ is generated by $\frac{1}{n}$. The $1-$dimensional torus $T_{\Gamma}$ is given by
\begin{equation*}
\begin{gathered}
T_{\Gamma} = (e^{i\pi n\varphi}, \ldots, e^{i\pi n\varphi},1,...,1), \\
\end{gathered}
\end{equation*}
We see that $D_{\Gamma} = \Lambda^{*}/2\Lambda^{*}$ has only one nontrivial element $\frac{1}{n}$. Then from formulas $(\ref{mainmap})$ we have
\begin{equation*}
\begin{gathered}
\mathcal{N}(n,p) = ((S^{p-1} \times S^{n-p-1})/\mathbb{Z}_2 \times T_{\Gamma})/D_{\Gamma} = (\mathcal{R}_P \times T_{\Gamma})/D_{\Gamma}, \\
(u_1,...,u_n, \varphi) \sim (-u_1, \ldots, -u_p, u_{p+1},...,u_n, \varphi + \frac{1}{n}),
\end{gathered}
\end{equation*}
We see that the action of $D_{\Gamma}$ is free on $T_{\Gamma}$. Therefore $\mathcal{N}$ fibers over $T_{\Gamma}/D_{\Gamma} = S^1$
\begin{equation*}
\begin{gathered}
\mathcal{N}(n,p) \xrightarrow{\mathcal{R}_P } S^1, \;\; S^1 = \mathbb{R}_P/\mathbb{Z}\langle\frac{1}{n}\rangle.
\end{gathered}
\end{equation*}
From formula $(\ref{mainmap})$ we get the Lagrangian
\begin{equation*}
\psi(\mathcal{N}(n,p)) = [u_1e^{i\pi n\varphi}:...:u_pe^{i\pi n\varphi}:u_{p+1}:...:u_n].
\end{equation*}
Let us reparametrize $\varphi$ in the formula above (to get rid of $n$) and obtain
\begin{equation}\label{embtwosphere}
\begin{gathered}
\mathcal{N}(n,p) \xrightarrow{\mathcal{R}_P} S^1, \;\; S^1 = \mathbb{R}/\mathbb{Z}, \\
L(n,p) = \psi(\mathcal{N}(n,p)) = [u_1e^{i\pi\varphi}:...:u_pe^{i\pi\varphi}:u_{p+1}:...:u_n].
\end{gathered}
\end{equation}

Denote by $L(n,p)$ the image $\psi(\mathcal{N}(n,p)) \subset \mathbb{C}P^{n-1}$. From theorem $(\ref{mironovtheorem2})$ we get that $L(n,p)$ is immersed Lagrangian.

\begin{lemma}\label{maslovclasstwo}
The Lagrangian $L(n,p)$ is embedded monotone Lagrangian. The minimal Maslov number of $L(n,p)$ is equal to $gcd(p,n)$, where $gcd$ stands for the greatest common divisor.
\end{lemma}

\begin{proof}
Since $P$ is Delzant and monotone, we see that $L(n,p)$ is embedded and monotone ( see Theorem $\ref{monotonecondition}$).

Lemma $\ref{simplemaslov}$ can be applied to system $(\ref{sphereprod})$. So, Lemma $\ref{simplemaslov}$ and Theorem $\ref{minmaslovtheorem}$ say that the minimal Maslov number $N_L = gcd(p, n-p) = gcd(p,n)$.

\end{proof}

\textbf{Remark.} Note that $|z+1|^2 + |z-1|^2 = 2|z|^2 + 2$ and $\frac{i(e^{i\varphi}+1)}{e^{i\varphi}-1} \in \mathbb{R}$. This implies that $arg(i(e^{i\varphi}+1)) = arg(e^{i\varphi}-1)$. Therefore $H_2(\mathbb{C}P^{n-1}, L) = \mathbb{Z}^2$ is generated by the following holomorphic discs:
\begin{equation*}
\begin{gathered}
\beta_1, \beta_2: (D^2, S^1) \rightarrow (CP^{n-1}, L),\\
\beta_1(z) = [\underbrace{\sqrt{p}\frac{i(z+1)}{2} : \sqrt{p}\frac{z-1}{2}: 0: \ldots :0}_p : 1 : ... : 1], \\
\beta_2(z) = [\underbrace{i\sqrt{p}(z+1): \sqrt{p}(z-1): 0: \ldots :0}_p : \\
 i\sqrt{n-p}(z+1): \sqrt{n-p}(z-1):0: \ldots :0].
\end{gathered}
\end{equation*}

\vspace{.08in}

\begin{lemma}
Suppose that $p$ is odd and $p>1$. Then $QH(L(2p,p), \; \mathbb{Z}_2[T, T^{-1}])$ is isomorphic to $H^{*}(L(2p,p), \; \mathbb{Z}_2[T,T^{-1}])$. This yields that the Lagrangians $L(2p,p)$ are not displaceable.
\end{lemma}

\begin{proof}
We know that $L(2p,p)$ fibers over $S^1$ with fiber $\mathcal{R}_P$. Let us consider the Serre spectral sequence.  Denote the generator of $E_{2}^{1,0}= \mathbb{Z}_2$ by $\hat{z}$ and denote by $\hat{x}$ the generator of $E_2^{0,1} = \mathbb{Z}_2$. Let $\hat{y}$ be an element of $H^{p-1}(\mathcal{R}_P, \mathbb{Z}_2)$ such that $E_2^{0,p-1} = \mathbb{Z}_2^2\langle \hat{x}^{p-1}, \hat{y} \rangle$. We see that $\hat{x}^{p-1}\hat{y}\hat{z} = \hat{y}^2\hat{z}$ (see Lemma $\ref{cohomtwosphere}$). Let us denote by $x,y,z$ the elements of $H^{*}(L(2p,p), \mathbb{Z}_2)$ corresponding to $\hat{x}, \hat{y}, \hat{z}$, respectively. We get that $x^{p-1}yz = y^2z$.

It follows from Lemma $\ref{maslovclasstwo}$ that the minimal Maslov number of $L(2p,p)$ is equal to $p$. Let us consider the spectral sequence of Oh (see Section $\ref{quantumdefinitions}$). The operator $\delta_1$ has degree $p-1$. Since $deg(x)=deg(z) = 1$, we see that $\delta_1(x) =\delta(z) = 0$. Then $\delta_1(y)\in H^{0}(L(2p,p), \mathbb{Z}_2)$ equals either 1, or 0.

We have
\begin{equation*}
\delta_1(y^2z) = \delta_1(y^2)z = 2\delta_1(y)z= 0 = \delta_1(x^{p-1}yz) = x^{p-1}z\delta_1(y)
\end{equation*}
So, we obtain that $\delta_1(y) = 0$. Since the cohomology ring of $L(2p,p)$ is generated by $x,y,z$, we get that the spectral sequence of Oh collapses at $E_1$.
\end{proof}

If $p, n$ are even, then from Lemma $\ref{trivialfibration}$ we get that $L(n,p)$ is diffeomorphic to $\mathcal{R}_P \times S^1$. It follows from Lemma $\ref{helplemma}$ that $L(n,p)$ is spin if and only if $\frac{n}{2}$ is even (and $p$ is even). Let $G$ be a ring where $2$ is invertible. From the Serre spectral sequence we get $H^{2}(L(n,p), G) = 0$. Assume that the minimal Maslov number of $L(n,p)$ is strictly greater than $2$. Let $f: L(n,p) \rightarrow \mathbb{R}$ be a Morse function with a single minimum $x^{min}$. We know from Section $\ref{quantumdefinitions}$ that there exists the isomorphism $h: QH^{*} \rightarrow QH^{*+2}$ such that $h(x^{min}) = h^0 \in H^{2}(L(n,p), G)$, where $h^0$ is Poicare dual to $(L(n,p) \cap H) \in H_{n-3}(L(n,p), G)$. Since $H^{2}(L(n,p), G) = 0$, we see that $h^0 = 0$. Then Lemma $\ref{zeroquantum}$ implies that $QH^{*}(L(n,p), G[T, T^{-1}]) = 0$.

\subsection{Proof of Theorem $\ref{masseylagr}$}

Let us consider $3-$dimensional cube and cut off three edges, i.e. we consider a polytope $P$ defined by inequalities (see Figure $\ref{fig:truncube}$)
\begin{equation*}
P = \left\{
 \begin{array}{l}
x_1 + 2 \geqslant 2 \;\;\; x_2 + 2 \geqslant 0 \;\;\; x_3 + 2 \geqslant 0 \\
-x_1 + 2 \geqslant 2 \;\;\; -x_2 + 2 \geqslant 0 \;\;\; -x_3 + 2 \geqslant 0 \\
x_1 + x_2 + 3 \geqslant 0 \\
-x_1 - x_3 + 3 \geqslant 0 \\
-x_2 + x_3 + 3 \geqslant 0
 \end{array}
\right.
\end{equation*}

\textbf{Remark.} The polytope $P$ and its generalizations were considered by Limonchenko in \cite{Limon2}. Limonchenko constructed a rich set of complex moment-angle manifolds with nontrivial Massey products.
\\

Easy to see that $P$ is Delzant. From Section $\ref{maincostruction}$ and Lemma $\ref{sumzero}$ we know that there is embedded Lagrangian in $\mathbb{C}P^9$ associated to $P$ (but not monotone because $P$ is not Fano). We denote by $\widetilde{\mathcal{R}}_P$ and $\widetilde{\mathcal{Z}}_P$ the real and complex moment-angle manifolds associated to $P$, respectively.  It is discussed in Section $\ref{wedge}$ that $\widetilde{\mathcal{R}}_P^{(2,\ldots,2)} = \widetilde{\mathcal{Z}}_P$.

\begin{lemma}
The cohomology ring $H^{*}(\widetilde{\mathcal{Z}}_P, \mathbb{Q})$ contains nonzero triple Massey product.
\end{lemma}

\begin{proof}
Let us study the face ring $\mathbb{Q}[P]$ and the complex $[R(P),d]$, defined in Section $\ref{torsiondefinitions}$.  We denote the facets of $P$ as in Figure $\ref{fig:truncube}$. We denote the exterior variables of  $[R(P),d]$  by $y_1,\ldots, y_6$ such that $dy_i = v_i$.

\begin{figure}[h]
\centering
\includegraphics[width=0.4\linewidth]{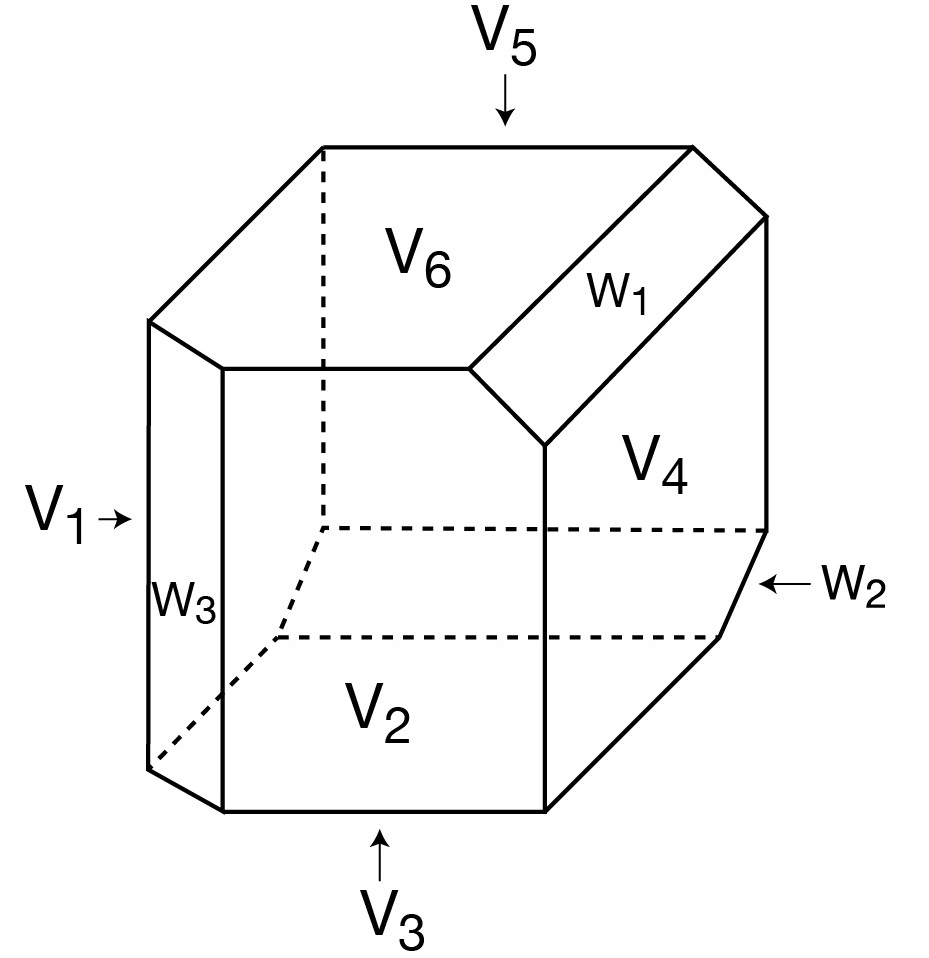}
\caption{truncted cube}
  \label{fig:truncube}
\end{figure}

Consider three classes $\alpha, \beta, \gamma \in H^{3}(\widetilde{\mathcal{Z}}_P, \mathbb{Q})$ represented by the following elements of $[R(P),d]$:
\begin{equation*}
\alpha = y_3v_6, \quad \beta = y_1v_4, \quad \delta = y_5v_2.
\end{equation*}
Since $v_4 \cap v_6 = v_1 \cap v_2 = v_5 \cap v_2 = \emptyset$, we get that
\begin{equation*}
\begin{gathered}
\alpha \beta = y_3y_1v_6v_4 = 0, \quad \beta\delta = y_1y_5v_4v_2 = d(y_1y_5y_4v_2) = d(h), \\
h = y_1y_5y_4v_2.
\end{gathered}
\end{equation*}
Hence, the triple Massey product $<\alpha, \beta, \gamma> \in H^{8}(\mathcal{Z}_P, \mathbb{Q})$ is defined. We have
\begin{equation*}
<\alpha, \beta, \delta> = \alpha h + 0\delta = y_3y_1y_5y_4v_6v_2 \neq 0.
\end{equation*}
So, the Massey product $<\alpha, \beta, \delta>$ is nontrivial.

\end{proof}

The real moment-angle manifold $\widetilde{\mathcal{R}}_P$ is given by a system

\begin{equation}\label{systemmassey}
\widetilde{\mathcal{R}}_P = \left\{
 \begin{array}{l}
u_1^2 + u_4^2 = 4 \\
u_2^2 + u_5^2 = 4 \\
u_3^2 + u_6^2 = 4 \\
u_1^2 + u_2^2 - u_7^2 = 1 \\
u_1^2 + u_3^2 + u_8^2 = 7 \\
u_2^2 - u_3^2 + u_9^2 = 3
 \end{array}
\right.
\end{equation}

We need to apply the simplicial multiwedge operation defined in Section $\ref{wedge}$, i.e. we repeat $j_i$ times the variable $u_i$. By $\overrightarrow{u}_i$ denote the vector $(u_{i,1},\ldots, u_{i, j_i})$ and by $|\overrightarrow{u}_i|^2 = u_{i,1}^2 + \ldots + u_{i,j_i}^2$ the squared length of the vector. Let us consider

\begin{equation}\label{systemformassey2}
\begin{gathered}
\widetilde{\mathcal{R}}_P^{(j_1,\ldots, j_9)} = \left\{
 \begin{array}{l}
|\overrightarrow{u}_1|^2 + |\overrightarrow{u}_4|^2 = 4 \\
|\overrightarrow{u}_2|^2 + |\overrightarrow{u}_5|^2 = 4 \\
|\overrightarrow{u}_3|^2 + |\overrightarrow{u}_6|^2 = 4 \\
|\overrightarrow{u}_1|^2 + |\overrightarrow{u}_2|^2 - |\overrightarrow{u}_7|^2 = 1 \\
|\overrightarrow{u}_1|^2 + |\overrightarrow{u}_3|^2 + |\overrightarrow{u}_8|^2 = 7 \\
|\overrightarrow{u}_2|^2 - |\overrightarrow{u}_3|^2 + |\overrightarrow{u}_9|^2 = 3
 \end{array}
\right.
\\
j_1=\ldots = j_6=4k, \;\; j_7 = 6k, \;\; j_8 = 6k, \;\; j_9 = 6k
\end{gathered}
\end{equation}
where $k$ is an arbitrary positive integer. Let $P^{(j_1, \ldots, j_9)}$ be the polytope associated to system $(\ref{systemformassey2})$. Since $P$ is Delzant, we obtain that $P^{(j_1,\ldots, j_9)}$ is Delzant (see Lemma $\ref{multiwedgedelzant}$). We see that system $(\ref{systemformassey2})$ satisfies the conditions of Lemma $\ref{equivmon}$ for $C = 2k$, i.e. the sum of columns is equal to the right-hand side multiplied by $2k$. This implies that $P^{(j_1, \ldots, j_9)}$ is Fano.

\begin{lemma}
The cohomology ring $H^{*}(\widetilde{\mathcal{R}}_P^{(j_1,\ldots, j_9)}, \mathbb{Q})$ contains nonzero triple Massey product.
\end{lemma}

\vspace*{.05in}

\textbf{Remark.} The lemma above follows immediately from Theorem $\ref{limonch}$,  but we also give an explicit proof. The reader can skip the proof of the lemma.

\begin{proof}

Note that $\widetilde{\mathcal{R}}_P^{(j_1,\ldots, j_9)} = \widetilde{\mathcal{Z}}_P^{(\frac{j_1}{2},\ldots, \frac{j_9}{2})}$. So, we study the topology of complex moment-angle manifold $\widetilde{\mathcal{Z}}_P^{(\frac{j_1}{2},\ldots, \frac{j_9}{2})}$. We denote facets of $P^{(\frac{j_1}{2},\ldots,\frac{j_9}{2})}$ and the corresponding variables in $[R(P^{(j_1,\ldots, j_9)}),d]$ by $v_{1,1},\ldots, v_{1,\frac{j_1}{2}}, \ldots, v_{9, 1}, \ldots, v_{9, \frac{j_9}{2}}$. Lemma $\ref{interfacetswedge}$ says that
\begin{equation*}
\begin{gathered}
d(y_{3,1}v_{3,2}\ldots v_{3,2k}v_{6,1}\ldots v_{6,2k}) = v_{3,1}v_{3,2}\ldots v_{3,2k}v_{6,1}\ldots v_{6,2k} = 0,\\
d(y_{1,1}v_{1,2}\ldots v_{1,2k}v_{4,1}\ldots v_{4,2k}) = v_{1,1}v_{1,2}\ldots v_{1,2k}v_{4,1}\ldots v_{4,2k} = 0, \\
d(y_{5,1}v_{5,2}\ldots v_{5,2k}v_{2,1}\ldots v_{2,2k}) = v_{5,1}v_{5,2}\ldots v_{5,2k}v_{2,1}\ldots v_{2,2k} = 0.
\end{gathered}
\end{equation*}
This means that we have three classes $\alpha', \beta', \delta' \in H^{16k-1}(\widetilde{\mathcal{R}}_P^{(j_1,\ldots, j_9)}, \mathbb{Q})$ represented in $[R(P^{(j_1,\ldots, j_9)}),d]$ by elements
\begin{equation*}
\begin{gathered}
\alpha' = y_{3,1}v_{3,2}\ldots v_{3,2k}v_{6,1}\ldots v_{6,2k}, \quad \beta' = y_{1,1}v_{1,2}\ldots v_{1,2k}v_{4,1}\ldots v_{4,2k}, \\
\delta' =  y_{5,1}v_{5,2}\ldots v_{5,2k}v_{2,1}\ldots v_{2,2k},
\end{gathered}
\end{equation*}
respectively. Since $v_4 \cap v_6 = v_1 \cap v_2 = v_5 \cap v_2 = \emptyset$, from Lemma $\ref{interfacetswedge}$ we have
\begin{equation*}
\begin{gathered}
\alpha'\beta' = 0, \\
\beta'\delta' = d(y_{1,1}y_{5,1}y_{4,1}v_{1,2}\ldots v_{1, 2k}v_{5,2}\ldots v_{5, 2k}v_{4,2}\ldots v_{4, 2k}v_{2,1}\ldots v_{2,2k}) = d(h'), \\
h' = y_{1,1}y_{5,1}y_{4,1}v_{1,2}\ldots v_{1, 2k}v_{5,2}\ldots v_{5, 2k}v_{4,2}\ldots v_{4, 2k}v_{2,1}\ldots v_{2,2k}.
\end{gathered}
\end{equation*}

Therefore, the triple Massey product $<\alpha', \beta', \delta'> \in H^{48k-4}(\mathcal{Z}_P, \mathbb{Q})$ is defined. We have
\begin{equation*}
\begin{gathered}
<\alpha', \beta', \delta'> = \alpha' h' + 0\delta' \neq 0 \\
\end{gathered}
\end{equation*}
We get that $<\alpha', \beta', \delta'>$ is nontrivial.

\end{proof}

Denote by $I$ the segment $[0,1]$. Let us consider a polytope $Q = P \times I$ given by inequalities
\begin{equation*}
Q = \left\{
 \begin{array}{l}
x_1 + 2 \geqslant 2 \;\;\; x_2 + 2 \geqslant 0 \;\;\; x_3 + 2 \geqslant 0 \\
-x_1 + 2 \geqslant 2 \;\;\; -x_2 + 2 \geqslant 0 \;\;\; -x_3 + 2 \geqslant 0 \\
x_1 + x_2 + 3 \geqslant 0 \\
-x_1 - x_3 + 3 \geqslant 0 \\
-x_2 + x_3 + 3 \geqslant 0 \\
x_4 + \frac{1}{2k} \geqslant , \;\;\; -x_4 + \frac{1}{2k} \leqslant 0
 \end{array}
\right.
\end{equation*}
We apply multiwedge operation to $Q$ and consider $Q^{(j_1,\ldots,j_9,1,1)}$. The real moment angle manifold associated to $Q^{(j_1,\ldots,j_9,1,1)}$ is given by
\begin{equation}\label{systemformassey3}
\begin{gathered}
\widetilde{\mathcal{R}_Q}^{(j_1,\ldots,j_9,1,1)} = \left\{
 \begin{array}{l}
|\overrightarrow{u}_1|^2 + |\overrightarrow{u}_4|^2 = 4 \\
|\overrightarrow{u}_2|^2 + |\overrightarrow{u}_5|^2 = 4 \\
|\overrightarrow{u}_3|^2 + |\overrightarrow{u}_6|^2 = 4 \\
|\overrightarrow{u}_1|^2 + |\overrightarrow{u}_2|^2 - |\overrightarrow{u}_7|^2 = 1 \\
|\overrightarrow{u}_1|^2 + |\overrightarrow{u}_3|^2 + |\overrightarrow{u}_8|^2 = 7 \\
|\overrightarrow{u}_2|^2 - |\overrightarrow{u}_3|^2 - |\overrightarrow{u}_9|^2 = -3 \\
|\overrightarrow{u}_{10}|^2 + |\overrightarrow{u}_{11}|^2 = \frac{1}{k}
 \end{array}
\right.
\\
j_1=\ldots j_6=4k, \;\; j_7 = 6k, \;\; j_8 = 6k, \;\; j_9 = 6k, \;\; j_{10}=j_{11} = 1
\end{gathered}
\end{equation}
It follows from Lemma $\ref{multiwedgedelzant}$ that $P^{(j_1,\ldots,j_9)}$ is Delzant. This yields that $Q$ is Delzant. Since conditions of Lemma $\ref{equivmon}$ are satisfied, we get that $Q$ is Fano. Note that the last equation of system $(\ref{systemformassey3})$ is independent of the other equations. This means that $ \widetilde{\mathcal{R}}_Q^{(j_1,\ldots,j_9,1,1)} = \widetilde{\mathcal{R}}_P^{(j_1,\ldots,j_9)} \times S^1$. Taking linear combinations of the equations of system $(\ref{systemformassey3})$ we can write the system in the form $(\ref{eqmain})$. Therefore, there is monotone Lagrangian embedding of
\begin{equation*}
L = \widetilde{\mathcal{R}}_Q^{(j_1,\ldots,j_9,1,1)}/\mathbb{Z}_2 \times T^6 = (\widetilde{\mathcal{R}}_P^{(j_1,\ldots,j_9)} \times S^1)/\mathbb{Z}_2 \times T^6.
\end{equation*}
into $\mathbb{C}P^{42k +1}$. Since the action of $\mathbb{Z}_2$ is free on $S^1$, we get that $\widetilde{\mathcal{R}}_Q^{(j_1,\ldots,j_9,1,1)}$ fibers over $S^1/\mathbb{Z}_2$. By Lemma $\ref{trivialfibration}$, the action of $\mathbb{Z}_2$ on $\widetilde{\mathcal{R}}_P^{(j_1,\ldots,j_9)} $ is isotopic to identity. Hence, the fibration over $S^1$ is trivial and $(\widetilde{\mathcal{R}}_P^{(j_1,\ldots,j_9)} \times S^1)/\mathbb{Z}_2$ is diffeomorphic to $\widetilde{\mathcal{R}}_P^{(j_1,\ldots,j_9)} \times S^1$. We already know that $\widetilde{\mathcal{R}}_P^{(j_1,\ldots,j_9)}$ has nontrivial Massey product. Since $L$ is diffeomorphic to
\begin{equation*}
\widetilde{\mathcal{R}}_P^{(j_1,\ldots,j_9)} \times  T^7,
\end{equation*}
we see that $L$ has nontrivial Massey product.

\subsection{Proof of Theorem $\ref{wide-narrow}$}

In this section we consider the standard $2-$simplex and cut off all its vertices. Then, we use Theorem $\ref{vertexcutoff}$ to study the corresponding complex moment-angle manifold.
\\

Let us consider a $6-$gon defined by inequalities
\begin{equation*}
\begin{gathered}
P = \left\{
 \begin{array}{l}
x_1 + 1 \geqslant 0 \quad x_2 + 1 \geqslant 0\\
-x_2 + 1 \geqslant 0 \quad -x_2+1 \geqslant 0 \\
x_1+x_2 + 1 \geqslant 0 \quad -x_1-x_2 + 1 \geqslant 0
 \end{array}
\right. \\
\end{gathered}
\end{equation*}
Denote the polytope above by $P$. The real moment-angle manifold $\widetilde{\mathcal{R}}_P$ associated to $P$ is defined by a system
\begin{equation*}
\widetilde{\mathcal{R}}_P = \left\{
 \begin{array}{l}
u_1^2  +  u_3^2 = 2\\
u_2^2 + u_4^2  =   2 \\
u_1^2 + u_2^2 + u_5^2 = 3 \\
u_1^2 + u_2^2 - u_6^2  =   1
 \end{array}
 \right.
\end{equation*}
Let us apply the simplicial multiwedge operation (see Section $\ref{wedge}$) to $P$. We consider $P^{(2,\ldots,2)}$ and the corresponding moment-angle manifolds
\begin{equation*}
\widetilde{\mathcal{R}}_P^{(2,\ldots,2)} = \widetilde{\mathcal{Z}}_P = \left\{
 \begin{array}{l}
|z_1|^2 + |z_3|^2 = 2 \\
|z_2|^2 + |z_4|^2 = 2 \\
|z_1|^2 + |z_2|^2 + |z_5|^2 = 3\\
|z_1|^2 + |z_2|^2 - |z_6|^2 = 1
 \end{array}
 \right.
\end{equation*}
Note that $\widetilde{\mathcal{Z}}_P$ is associated to the polytope $P$. Let us consider a polytope
\begin{equation*}
Q = P^{(k, k, k, k, k, k)},
\end{equation*}
where $k$ is an arbitrary positive integer. Denote by $\widetilde{\mathcal{Z}}_Q$ the complex moment-angle manifold associated to $Q$. We have $\widetilde{\mathcal{R}}^{(2,\ldots,2)}_Q = \widetilde{\mathcal{Z}}_Q$, i.e. $\widetilde{\mathcal{R}}^{(2,\ldots,2)}_Q$ is associated to the polytope
\begin{equation*}
Q^{(2,\ldots,2)} = P^{(2k, 2k, 2k, 2k, 2k, 2k)}.
\end{equation*}
We repeat $j_i$ times the variable $z_i$. By $\overrightarrow{z}_i$ denote the vector $(z_{i,1},\ldots, z_{i, j_i})$ and by $|\overrightarrow{z}_i|^2 = z_{i,1}^2 + \ldots + z_{i,j_i}^2$ the squared length of the vector. We have
\begin{equation}\label{widenarrow1}
\begin{gathered}
 \widetilde{\mathcal{Z}}_Q = \left\{
 \begin{array}{l}
|\overrightarrow{z}_1|^2 + |\overrightarrow{z}_3|^2 = 2 \\
|\overrightarrow{z}_2|^2 + |\overrightarrow{z}_4|^2 = 2 \\
|\overrightarrow{z}_1|^2 + |\overrightarrow{z}_2|^2 + |\overrightarrow{z}_5|^2 = 3\\
|\overrightarrow{z}_1|^2 + |\overrightarrow{z}_2|^2 - |\overrightarrow{z}_6|^2 = 1
 \end{array}
 \right.
 \\
 j_1 = \ldots = j_6 = k.
\end{gathered}
\end{equation}

\begin{lemma}\label{consumquantum}
The rings
\begin{equation*}
\begin{gathered}
H^{*}(\widetilde{\mathcal{Z}}_Q^, \mathbb{Z}_2), \\
H^{*}(\#9(S^{4k-1} \times S^{8k-3}) \#8(S^{6k-2} \times S^{6k-2}), \; \mathbb{Z}_2)
\end{gathered}
\end{equation*}
are isomorphic.
\end{lemma}

\begin{proof}
Theorem $\ref{cohomisotopic}$ says that $H^{*}(\widetilde{\mathcal{Z}}_Q, \mathbb{Z})$ and $H^{*}(\widetilde{\mathcal{Z}}_P, \mathbb{Z})$ are isomorphic as ungraded rings. Denote by $h$ the isomorphism. It follows from Theorem $\ref{vertexcutoff}$) that
\begin{equation*}
\widetilde{\mathcal{Z}}_P = \#9(S^{3} \times S^{5}) \#8(S^{4} \times S^{4})
\end{equation*}
and denote generators of $H^{*}(\widetilde{\mathcal{Z}}_P, \; \mathbb{Z})$ by
\begin{equation*}
s^3_i, \;\; s_i^5, \;\; s^4_j, \;\; \hat{s}_j^4     \quad i=1,\ldots, 9, \quad j=1,\ldots, 8.
\end{equation*}
To study the ring $H^{*}(\widetilde{\mathcal{Z}}_Q, \mathbb{Z})$ it is enough to find  $h(s^3_i), h(s^4_j)$.

Let us study the face ring of $P$ and use Theorem $\ref{torcohom}$. We denote by $v_1,...,v_6$ the variables of $[R(P), d]$ coming from the edges of $6-$gon. Denote the corresponding exterior variables of $[R(P), d]$ by $y_1, \ldots, y_6$ (see Figure $\ref{fig:6gon}$).

\begin{figure}[h]
\centering
\includegraphics[width=0.5\linewidth]{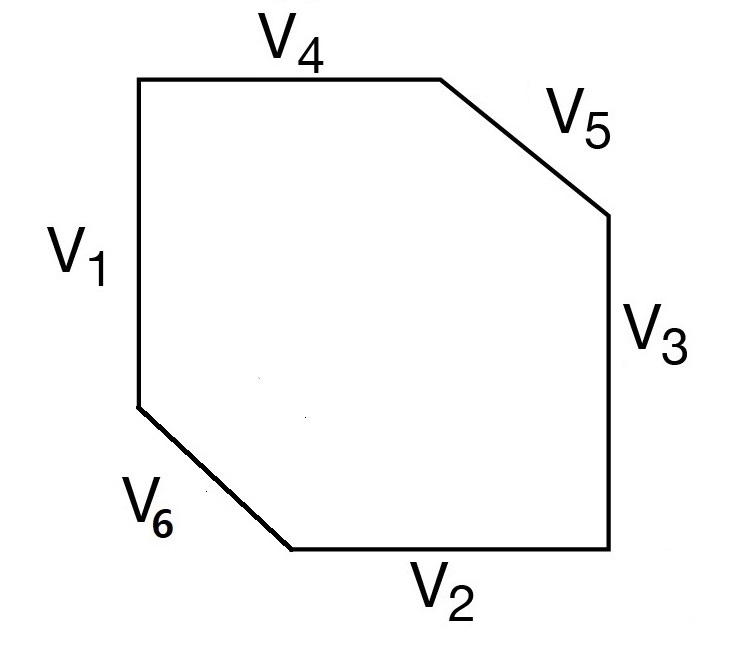}
\caption{6-gon}
  \label{fig:6gon}
\end{figure}

Since $v_1 \cap v_5 = \emptyset$, we see that $d(y_1v_5) = v_1v_5 = 0$. Therefore, $y_1v_5$ represent nontrivial element. Analogously,
\begin{equation*}
\begin{gathered}
y_1v_{5}, \;\; y_1v_{3}, \;\; y_1v_{2},  \;\; y_4v_{3}, \;\; y_4v_{2}, \;\; y_4v_{6}, \\
y_5v_{2}, \;\; y_5v_{6}, \;\; y_3v_{6} \quad  \in H^{3}(\widetilde{\mathcal{Z}}_P, \mathbb{Z})
\end{gathered}
\end{equation*}
represent nontrivial elements in $H^{3}(\widetilde{\mathcal{Z}}_P, \mathbb{Z})$. In the same way
\begin{equation*}
\begin{gathered}
y_1y_4v_3, \;\; y_1y_4v_2, \;\; y_1y_5v_2, \;\; y_4y_5v_2, \\
y_4y_5v_6, \;\; y_4y_3v_6, \;\; y_5y_3v_6, \;\; y_3y_2v_1 \in H^{4}(\widetilde{\mathcal{Z}}_P, \mathbb{Z})
\end{gathered}
\end{equation*}
represent nontrivial elements in $H^{4}(\widetilde{\mathcal{Z}}_P, \mathbb{Z})$. For example, since $v_1 \cap v_3 = v_3 \cap v_4 = \emptyset$, we have $d(y_1y_4v_3) = y_4v_1v_3 - y_1v_4v_3 = 0$.

Now, let us study the face ring of $Q = P^{(k, k, k, k, k, k)}$. We denote the new variables corresponding to the new facets by $v_{i,1}, \ldots, v_{i,k}$. New exterior variables are denoted by $y_{i,1}, \ldots, y_{i,k}$. Theorem $\ref{interfacetswedge}$ says that
\begin{equation*}
d(y_{1,1}v_{1,2} \ldots v_{1,k}v_{5,1}\ldots v_{5,k}) = v_{1,1} \ldots v_{1,k}v_{5,1}\ldots v_{5,k} =0
\end{equation*}
because $v_{1,1} \cap \ldots \cap v_{1,2k} \cap v_{3,1}\ldots v_{3,2k} = \emptyset$. Analogously, we see that

\begin{equation*}
\begin{gathered}
y_{1,1}v_{1,2}\ldots v_{1,k}v_{5,1}\ldots v_{5,k}, \;\; y_{1,1}v_{1,2}\ldots v_{1,k}v_{3,1}\ldots v_{3,k}, \;\; y_{1,1}v_{1,2}\ldots v_{1,k}v_{2,1}\ldots v_{2,k}, \\
y_{4,1}v_{4,2}\ldots v_{4,k}v_{3,1}\ldots v_{3,k}, \;\; y_{4,1}v_{4,2}\ldots v_{4,k}v_{2,1}\ldots v_{2,k}, \;\; y_{4,1}v_{4,2}\ldots v_{4,k}v_{6,1}\ldots v_{6,k}, \\
y_{5,1}v_{5,2}\ldots v_{5,k}v_{2,1}\ldots v_{2,k}, \;\; y_{5,1}v_{5,2}\ldots v_{5,k}v_{6,1}\ldots v_{6,k}, \;\; y_{3,1}v_{3,2}\ldots v_{3,2k}v_{6,1}\ldots v_{6,k} \\
  \in \;\;  H^{4k-1}(\widetilde{\mathcal{Z}}_Q, \mathbb{Z})
\end{gathered}
\end{equation*}
So, we have
\begin{equation*}
\begin{gathered}
h(s^3_i) \in H^{4k-1}(\widetilde{\mathcal{Z}}_Q, \mathbb{Z}).
\end{gathered}
\end{equation*}
Arguing as before, we show that
\begin{equation*}
\begin{gathered}
h(s^4_j) \in H^{6k-2}(\widetilde{\mathcal{Z}}_Q, \mathbb{Z}).
\end{gathered}
\end{equation*}
As a result, we obtain that rings
\begin{equation*}
H^{*}(\widetilde{\mathcal{Z}}_Q, \; \mathbb{Z}), \quad H^{*}(\#9(S^{4k-1} \times S^{8k-3}) \#8(S^{6k-2} \times S^{6k-2}), \; \mathbb{Z})
\end{equation*}
are isomorphic as ungraded rings. Since there is no torsion, we see that
\begin{equation*}
H^{*}(\widetilde{\mathcal{Z}}_Q, \; \mathbb{Z}_2), \quad H^{*}(\#9(S^{4k-1} \times S^{8k-3}) \#8(S^{6k-2} \times S^{6k-2}), \; \mathbb{Z}_2)
\end{equation*}
are isomorphic as ungraded rings.
\end{proof}

As we already noticed
\begin{equation*}
 \widetilde{\mathcal{Z}}_Q = \widetilde{\mathcal{R}}^{(2,\ldots,2)}_Q,
\end{equation*}
i.e. $\widetilde{\mathcal{R}}_P^{(2k,\ldots,2k)} $ is associated to the polytope $P^{(2k,\ldots,2k)}= Q^{(2,\ldots,2)}$. We see that $\widetilde{\mathcal{R}}_Q^{(2,\ldots,2)}$ is given by the system
\begin{equation}\label{widenarrow2}
\begin{gathered}
\widetilde{\mathcal{R}}_Q^{(2,\ldots,2)} = \left\{
 \begin{array}{l}
|\overrightarrow{u}_2|^2 + |\overrightarrow{u}_3|^2 = 2 \\
|\overrightarrow{u}_2|^2 + |\overrightarrow{u}_4|^2 = 2 \\
|\overrightarrow{u}_1|^2 + |\overrightarrow{u_2}|^2 + |\overrightarrow{z}_5|^2 = 3\\
|\overrightarrow{u}_1|^2 + |\overrightarrow{u}_2|^2 - |\overrightarrow{u}_6|^2 = 1
 \end{array}
 \right.
 \\
\overrightarrow{u}_r = (u_{r,1}, \ldots, u_{r,2k}) \in \mathbb{R}^{2k}\\
|\overrightarrow{u}_r|^2 = u_{r,1}^2 + \ldots + u_{r,2k}^2.
\end{gathered}
\end{equation}
Let us write the system above in the following form:

\begin{equation}\label{widenarrow3}
\begin{gathered}
\widetilde{\mathcal{Z}}_Q = \widetilde{R}_Q^{(2,\ldots,2)} = \left\{
 \begin{array}{l}
|\overrightarrow{u}_1|^2 + |\overrightarrow{u}_2|^2 +|\overrightarrow{u}_3|^2   +  |\overrightarrow{u}_4|^2 + |\overrightarrow{u}_5|^2 + |\overrightarrow{u}_6|^2   = 6\\
|\overrightarrow{u}_2|^2 + |\overrightarrow{u}_4|^2 = 2 \\
|\overrightarrow{u}_1|^2 + |\overrightarrow{u_2}|^2 + |\overrightarrow{z}_5|^2 = 3\\
|\overrightarrow{u}_1|^2 + |\overrightarrow{u}_2|^2 - |\overrightarrow{u}_6|^2 = 1
 \end{array}
 \right.
\end{gathered}
\end{equation}

We see that system $(\ref{widenarrow3})$ can be written in the form $(\ref{eqmain})$. From Section $\ref{maincostruction}$ we know that there exists a Lagrangian $L \subset \mathbb{C}P^{12k-1}$ associated to system $\ref{widenarrow3}$ (or equivalently to the polytope $P^{(2k,\ldots,2k)}$). From Lemma $\ref{trivialfibration}$ we get that
\begin{equation*}
L = \mathcal{R}_Q^{(2,\ldots, 2)} \times T^3, \quad \;\; \mathcal{R}_Q^{(2,\ldots, 2)} = \widetilde{\mathcal{R}}_Q^{(2,\ldots,2)}/\mathbb{Z}_2.
\end{equation*}

\begin{lemma}
The Lagrangian $L$ is embedded and monotone. The minimal Maslov number $N_L$ is equal to $2k$.
\end{lemma}

\begin{proof}
The polytope $P$ is Delzant and Lemma $\ref{multiwedgedelzant}$ says that $P^{(2k,\ldots, 2k)}=Q^{(2,\ldots,2)}$ is Delzant. Moreover, the conditions of Lemma $\ref{equivmon}$ are satisfied for $C = 2k$. Therefore  $P^{(2k,\ldots, 2k)}$ is  Fano. It follows from Theorem $\ref{monotonecondition}$ that  $L$ is embedded and monotone. From Lemma $\ref{simplemimasnumber}$ and Theorem $\ref{minmaslovtheorem}$ we get that $N_L = 2k$.
\end{proof}

Dimension of the polytope $P^{(2k,\ldots,2k)}$ equals $12k$ and dimension of $Q = P^{(k,\ldots,k)}$ is equal to 6k. It follows from Lemma $\ref{helplemma}$ that $L$ is spin.

Let $G$ be a ring in which $2$ is invertible. Lemma $\ref{vanishquantum}$ says that $QH(L, G[T, T^{-1}]) = 0$.
\\

Let us prove that $QH(L, \mathbb{Z}_2[T, T^{-1}]) = 0$. Recall that
\begin{equation*}
 m(Q) = min\{\ell \in \mathbb{N} \; | \; v_{i_1}\cap \ldots \cap v_{i_\ell} = \emptyset  \}.
\end{equation*}
We see that $m(P) = 2$. It follows from Theorem $\ref{interfacetswedge}$ that
\begin{equation*}
m(Q) = m(P^{(k,\ldots,k)}) = 2k.
\end{equation*}
For example, $v_{1,1}\cap \ldots \cap v_{1,k} \cap v_{3,1} \cap \ldots \cap v_{3,k} = \emptyset$.  By Lemma $\ref{consumquantum}$, $H^{*}(\widetilde{\mathcal{Z}}_Q, \mathbb{Z}_2)$ is isomorphic to
\begin{equation*}
H^{*}(\#9(S^{4k-1} \times S^{8k-3}) \#8(S^{6k-2} \times S^{6k-2}), \; \mathbb{Z}_2).
\end{equation*}
We have the Cartan-Leray spectral sequence
\begin{equation*}
\begin{gathered}
E_2^{p,q} = H^{p}(\mathbb{R}P^{\infty}, H^{q}(\widetilde{\mathcal{Z}}_Q, \mathbb{Z}_2)) \Rightarrow H^{*}(\mathcal{Z}_Q, \mathbb{Z}_2),\\
\mathcal{Z}_Q = \widetilde{\mathcal{Z}}_Q/\mathbb{Z}_2 = \widetilde{\mathcal{R}}_Q^{(2,\ldots,2)}/\mathbb{Z}_2.
\end{gathered}
\end{equation*}
The second page of the spectral sequence is shown in figure $\ref{fig:s3}$.

\begin{figure}[h]
\centering
\includegraphics[width=0.5\linewidth]{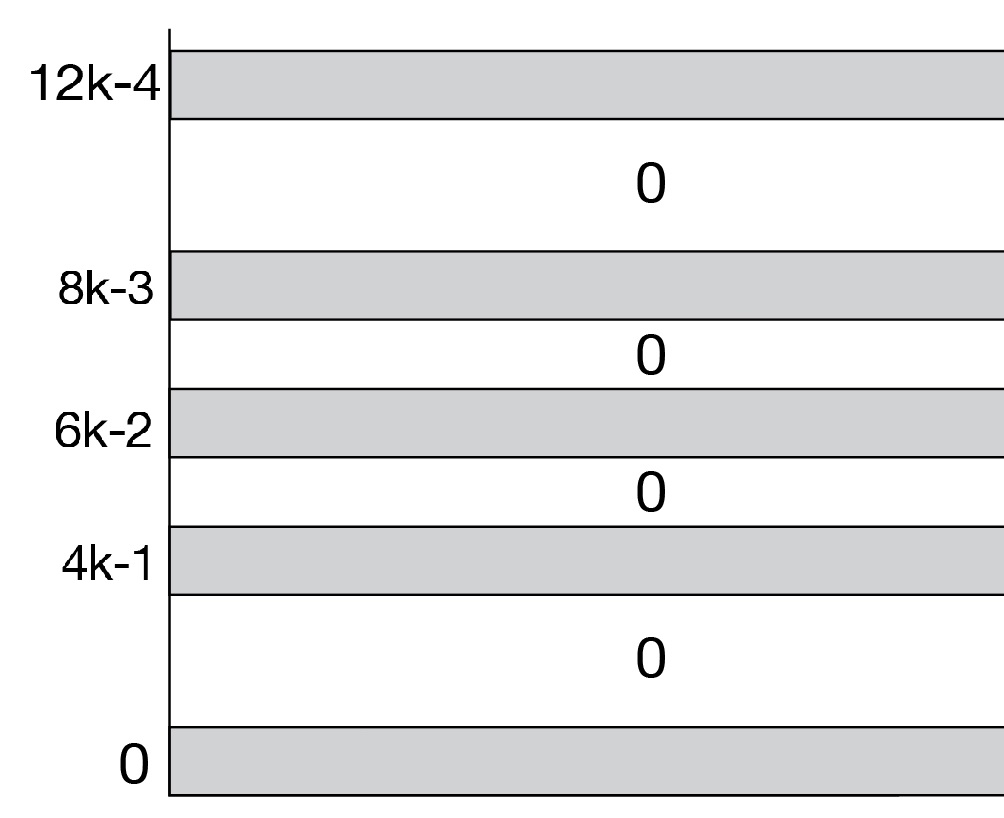}
\caption{$E_2$ page}
  \label{fig:s3}
\end{figure}
Let us denote
\\
\\
The nonzero element of $E^{0, 12k-4}_2$ by $h$;
\\
Generators of $E^{0, 8k-3}_2$ by $s^{8k-3}_i$;
\\
Generators of $E^{0,6k-1}_2$ by $s^{6k-2}_j$ and $\hat{s}^{6k-2}_j$, where $s_j^{6k-2}$ is the Poincare dual to  $\widehat{s}_j^{6k-2}$;
\\
Generators of $E^{0, 4k-1}_2$ by $s^{4k-1}_i$;
\\
Generator of $H^{*}(\mathbb{R}P^{\infty}, \mathbb{Z}_2)$ by $\alpha$;
\\
Assume that $s^{4k-1}_i$ is Poicare dual to $s^{8k-3}_i$.

The first nontrivial differential is $\partial_{2k}$ and $E_2 = E_{2k}$. We have
\begin{equation*}
\begin{gathered}
\partial_{2k}(s_i^{8k-3}) = \alpha^{2k}(\sum\limits_{r=1}^{8}a_{i,r}s_r^{6k-2} + \sum\limits_{r=1}^{8}b_{i,r}\widehat{s}_r^{6k-2}), \\
\partial_{2k}(s_j^{6k-2}) = \alpha^{2k}\sum\limits_{r=1}^{9}c_{j,r}s_r^{4k-1}, \quad \partial_{2k-1}(\widehat{s}_j^{6k-2}) = \alpha^{2k}\sum\limits_{r=1}^{9}\widehat{c}_{j,r}s_r^{4k-1}, \\
a_{i,r}, \; b_{i,r}, \; c_{j,r}, \; \widehat{c}_{j,r} \; \in \mathbb{Z}_2.
\end{gathered}
\end{equation*}
In other words,
\begin{equation*}
\begin{gathered}
\partial_{2k}\begin{pmatrix}
s^{8k-3}_1 \\
.\\
.\\
.\\
s^{8k-3}_9
\end{pmatrix}
=
\begin{pmatrix}
A & B
\end{pmatrix}
\begin{pmatrix}
\alpha^{2k}s^{6k-2}_1 \\
.\\
.\\
.\\
\alpha^{2k}s^{6k-2}_8\\
\alpha^{2k}\widehat{s}^{6k-2}_1 \\
.\\
.\\
.\\
\alpha^{2k}\widehat{s}^{6k-2}_8
\end{pmatrix}
\end{gathered}
\end{equation*}
\begin{equation*}
\begin{gathered}
\partial_{2k}\begin{pmatrix}
s^{6k-2}_1 \\
.\\
.\\
.\\
s^{6k-2}_8
\end{pmatrix}
=
C\begin{pmatrix}
\alpha^{2k}s^{4k-1}_1 \\
.\\
.\\
.\\
\alpha^{2k}s^{4k-1}_9
\end{pmatrix}
, \quad \;\;
\partial_{2k}\begin{pmatrix}
\widehat{s}^{6k-2}_1 \\
.\\
.\\
.\\
\widehat{s}^{6k-2}_8
\end{pmatrix}
=
\widehat{C}\begin{pmatrix}
\alpha^{2k}s^{4k-1}_1 \\
.\\
.\\
.\\
\alpha^{2k}s^{4k-1}_9
\end{pmatrix}
,
\end{gathered}
\end{equation*}
where $A = (a_{i,r})$, $B=(b_{i,r})$, $C=(c_{j,r})$, $\widehat{C}=(\widehat{c}_{j,r})$ are matrices.  Since $s^{8k-3}_{i}  s^{6k-2}_j = 0$, we have
\begin{equation*}
\begin{gathered}
0=\partial_{2k}(s^{8k-3}_{i} s^{6k-2}_j) = (b_{i,j} + c_{j,i})\alpha^{2k}h, \\
0=\partial_{2k}(s^{8k-3}_{i}  \widehat{s}^{6k-2}_j) =  (a_{i,j} + \widehat{c}_{j,i})\alpha^{2k}h, \\
\Rightarrow b_{i,j} = c_{j,i}, \quad a_{i,j} = \widehat{c}_{j,i},
\end{gathered}
\end{equation*}
This means that $A = \widehat{C}^T$, $B = C^T$, where $T$ stands for the transposition. Therefore,
\begin{equation}\label{spec5}
dim(\partial_{2k}(E_{2k}^{0, 8k-3})) = dim(\partial_{2k}(E_{2k}^{0, 6k-2})).
\end{equation}
Also, since $dim(\partial_{2k}(E_{2k}^{0, 8k-3})) \subset ker(\partial_{2k}(E_{2k}^{2k, 6k-2}))$ and $dim(E_{2k}^{2k, 6k-2})=16$, we get that
\begin{equation*}
dim(\partial_{2k}(E_{2k}^{0, 6k-2})) \leqslant 8.
\end{equation*}

Let us note that $E_{2k+1}^{*, 6k-2} = E_{6k-1}^{*, 6k-2}$ and $E_{6k}^{*, 6k-2} = E_{\infty}^{*, 6k-2}$. We see that only $\partial_{6k-1}$ may affect the $(6k-2)$th row of $E_{6k-1}$. The manifold $\mathcal{Z}_Q$ is finite-dimensional and for any $r > 12k-4$ we must have $E_{\infty}^{r, 6k-2} = 0$. Since, $dim(E_{6k-1}^{0, 12k-4}), dim(E_{6k-1}^{6k-1, 0}) \leqslant  1$, we get that $dim(E_{6k-1}^{r, 6k-2}) \leqslant 2$ for all $r \geqslant 2k$. Hence,
\begin{equation*}
dim(\partial_{2k}(E_{2k}^{0, 6k-2})) \geqslant 7.
\end{equation*}

We have two cases:
\\
\\
\textbf{1.} Assume that $dim(\partial_{2k}(E_{2k}^{0, 8k-3})) = 7$. The next nontrivial differential may be $\partial_{4k-1}$ and $E_{2k+1}=E_{4k-1}$.  There exist  elements $x_1, x_2 \in E_{4k-1}^{0, 4k-1}$, $y_1, y_2 \in E_{4k-1}^{0, 8k-3}$, and $t_1, t_2 \in E_{4k-1}^{0, 6k-2}$ such that for any $r \geqslant 2k$ and $m\geqslant 0$ we have $E_{4k-1}^{r, 4k-1} = \mathbb{Z}_2\langle \alpha^rx_1, \alpha^rx_2 \rangle$, $E_{4k-1}^{r, 6k-2} = \mathbb{Z}_2\langle \alpha^rt_1, \alpha^rt_2 \rangle$, and $E_{4k-1}^{m, 8k-3} = \mathbb{Z}_2\langle \alpha^ry_1, \alpha^ry_2 \rangle$. Note that $\partial_{4k-1}$ can not change $E_{4k-1}^{*, 6k-2}$. Since the row $E_{\infty}^{*, 6k-2}$ is finite ($\mathcal{Z}_Q$ is finite-dimensional), we see that the differential $\partial_{6k-1}$ must be nonzero on $E_{6k-1}^{*, 12k-4}$ and $E_{6k-1}^{*, 6k-2}$ (in particular $E_{6k-1}^{0, 12k-4} \neq 0$). Hence, rows $E_{\infty}^{*, 4k-1}$, $E_{\infty}^{*, 8k-3}$ are finite if and only if the map $\partial_{4k-1}: E_{4k-1}^{0, 8k-3} \rightarrow E_{4k-1}^{4k-1, 4k-1}$ is isomorphism. Therefore, $E_{4k}^{4k-1, 4k-1} = 0$. Let us consider $E_{4k}$ and let $x$ be an arbitrary element of $E_{4k}^{0, 4k-1}$. If $\partial_{4k}(x) = \alpha^{4k}$, then  $\partial_{4k}(\alpha^{4k-1}  x) = \alpha^{8k-1} \neq 0$. On the other hand, $\alpha^{4k-1}  x = 0$ because it belongs to  $E_{4k}^{4k-1, 4k-1}=0$. This contradiction shows that
\begin{equation*}
dim(\partial_{2k}(E_{2k}^{0, 8k-3})) = 7 \; \Rightarrow \; \partial_{4k}(E_{4k-1}^{0, 4k-1}) = 0.
\end{equation*}

\textbf{2.} Assume that $dim(\partial_{2k}(E_{2k}^{0, 8k-3})) = 8$. Then, $E_{\infty}^{r, 6k-2} = 0$ for any $r \geqslant 2k$. As in the previous case, the next nontrivial differential may be $\partial_{4k-1}$ and $E_{2k+1}=E_{4k-1}$.  There exist  elements $x \in E_{4k-1}^{0, 4k-1}$ and $y \in E_{4k-1}^{0, 8k-3}$ such that  $E_{4k-1}^{r, 4k-1} = \mathbb{Z}_2\langle \alpha^rx \rangle$, $E_{4k-1}^{m, 8k-3} = \mathbb{Z}_2\langle \alpha^my\rangle$ for any $r \geqslant 2k$ and $m \geqslant 0$.

Suppose that $\partial_{4k-1}(y) = \alpha^{4k-1}x$. Then, $E_{4k}^{r, 4k-1}= 0$ for all $r \geqslant 4k-1$. Let $w$ be an element of $E_{4k}^{0, 4k-1}$ such that $\partial_{4k}(w) = \alpha^{4k}$. We know that $\alpha^{4k-1}  w = 0$   and $ 0 = \partial_{4k}(\alpha^{4k-1}  w) = \alpha^{8k-1} \neq 0$. This yields that
\begin{equation*}
dim(\partial_{2k}(E_{2k-1}^{0, 8k-3})) = 8, \;\;\; \partial_{4k-1}(y) = \alpha^{4k-1}x \; \Rightarrow \; \partial_{4k}(E_{4k-1}^{0, 4k-1}) = 0.
\end{equation*}

Suppose that $\partial_{4k-1}(y) = 0$. We see that $E_{4k}^{m, 8k-3} = \mathbb{Z}_2\langle\alpha^m y\rangle$ for any $m \geqslant 0$. Since $E_{4k}^{0,4k-1} = E_{2}^{0,4k-1}$, we have an element $x \in E_{4k}^{0,4k-1}$ such that $xy = 0$. Then,
\begin{equation*}
0 = \partial_{4k}(xy) = y\partial_{4k}(x) \; \Rightarrow \; \partial_{4k}(x) \neq \alpha^{4k}.
\end{equation*}
We get the following:
\begin{equation*}
dim(\partial_{2k}(E_{2k}^{0, 8k-3})) = 8, \;\;\; \partial_{4k-1}(y) = 0 \; \Rightarrow \; \partial_{4k-1}(E_{4k-1}^{0, 4k-1}) = 0.
\end{equation*}

As a result, we see that $\partial_{4k}(E_{4k-1}^{0,4k-1}) = 0$. It follows from Theorem $\ref{vanishquantum}$ that $QH(L, \mathbb{Z}_2[T, T^{-1}]) = 0$.

\subsection{Proof of Theorem $\ref{wide-narrow2}$}

Let us consider three dimensional cube $I^3$ defined by
\begin{equation*}
I^3= \left\{
 \begin{array}{l}
x_1+1 \geqslant 0 \;\;\; x_2+1 \geqslant 0 \;\;\; x_3 + 1 \geqslant 0 \\
-x_1 + 1 \geqslant 0 \;\;\; -x_2+1 \geqslant 0 \;\;\; -x_3 + 1\geqslant 0 \\
 \end{array}
 \right.
\end{equation*}
Let us cut off the vertices $(1,1,1)$, $(-1,-1,-1)$ and consider the polytope $P$ defined by inequalities (see Figure $\ref{fig:truncube3}$)

\begin{equation*}
P = \left\{
 \begin{array}{l}
x_1+1 \geqslant 0 \;\;\; x_2+1 \geqslant 0 \;\;\; x_3 + 1 \geqslant 0 \\
-x_1 + 1 \geqslant 0 \;\;\; -x_2+1 \geqslant 0 \;\;\; -x_3 + 1\geqslant 0 \\
-x_1 - x_2 - x_3 + 2 \geqslant 0 \quad x_1+x_2+x_3+2 \geqslant 0
 \end{array}
 \right.
\end{equation*}
The complex moment-angle manifold $\widetilde{\mathcal{Z}}_P \subset \mathbb{C}^{8}$ associated to $P$ is given by a system
\begin{equation*}
\widetilde{\mathcal{Z}}_P = \left\{
 \begin{array}{l}
|z_1|^2  +  |z_4|^2 = 2\\
|z_2|^2 + |z_5|^2  =   2 \\
|z_3|^2 + |z_6|^2 = 2  \\
|z_1|^2 + |z_2|^2 + |z_3|^2 + |z_7|^2  =   5 \\
|z_1|^2 + |z_2|^2 + |z_3|^2 - |z_8|^2  =   1
 \end{array}
 \right.
\end{equation*}
Let us apply the multiwedge operation defined in Section $\ref{wedge}$ and consider
\begin{equation*}
Q = P^{(k, \ldots, k, 2k, 2k)}
\end{equation*}
We have the associated complex moment angle manifold $\widetilde{\mathcal{Z}}_Q$. We know that
\begin{equation*}
\widetilde{\mathcal{Z}}_Q = \widetilde{\mathcal{R}}_Q^{(2,\ldots, 2)}
\end{equation*}
and $\widetilde{\mathcal{R}}_Q^{(2,\ldots, 2)} \subset \mathbb{R}^{20k}$ is associated to the polytope $Q^{(2,\ldots,2)}=P^{(2k, \ldots, 2k, 4k, 4k)}$. We repeat $j_i$ times all variables (as in the proof of Theorem $\ref{wide-narrow}$). By $\overrightarrow{u}_i$ denote the vector $(u_{i,1},\ldots, u_{i, j_i})$ and by $|\overrightarrow{u}_i|^2 = u_{i,1}^2 + \ldots + u_{i,j_i}^2$ the squared length of the vector. We obtain that the real moment-angle manifold $\widetilde{\mathcal{R}}_Q^{(2,\ldots, 2)}  \subset \mathbb{R}^{20k}$ is defined by
\begin{equation}\label{widenarrow4}
\begin{gathered}
 \widetilde{\mathcal{R}}^{(2,\ldots,2)}_Q =
 \left\{
 \begin{array}{l}
|\overrightarrow{u}_1|^2  +  |\overrightarrow{u}_4|^2 = 2\\
|\overrightarrow{u}_2|^2 + |\overrightarrow{u}_5|^2  =   2 \\
|\overrightarrow{u}_3|^2 + |\overrightarrow{u}_6|^2 = 2  \\
|\overrightarrow{u}_1|^2 + |\overrightarrow{u}_2|^2 + |\overrightarrow{u}_3|^2 + |\overrightarrow{u}_7|^2  =   5 \\
|\overrightarrow{u}_1|^2 + |\overrightarrow{u}_2|^2 + |\overrightarrow{u}_3|^2 - |\overrightarrow{u}_8|^2  =   1
 \end{array}
 \right.
 \\
 j_1=\ldots=j_6 = 2k, j_7=j_8=4k
\end{gathered}
\end{equation}
Since $P$ is Dlzant, we get that $Q^{(2,\ldots,2)}$ is Delzant (see Lemma $\ref{multiwedgedelzant}$). The conditions of Lemma $\ref{equivmon}$ are satisfied (for $C = 2k$) and we get that $Q^{(2,\ldots,2)}$ is Fano. Therefore, there is embedded monotone Lagrangian $L \subset \mathbb{C}P^{20k-1}$ associated to  $Q^{(2,\ldots,2)}$ (see Theorem $\ref{monotonecondition})$. It follows from Lemma $\ref{trivialfibration}$ that
\begin{equation*}
L = \widetilde{\mathcal{R}}_Q^{(2,\ldots,2)}/\mathbb{Z}_2 \times T^4 = \widetilde{\mathcal{Z}}_Q/\mathbb{Z}_2 \times T^4.
\end{equation*}
From Lemma $\ref{simplemimasnumber}$ we have
\begin{equation*}
N_L = gcd(4k, 10k, 2k) = 2k.
\end{equation*}

Then, Theorem $\ref{cohomisotopic}$ says that $H^{*}(\widetilde{\mathcal{Z}}_Q, \mathbb{Z})$ and $H^{*}(\widetilde{\mathcal{Z}}_P, \mathbb{Z})$ are isomorphic as ungraded rings. Moreover, by Theorem $\ref{vertexcutoff}$ we have
\begin{equation}\label{cubecomplexangle}
\widetilde{\mathcal{Z}}_P = \mathcal{G}(\mathcal{G}((S^3 \times S^3 \times S^3)) \#7(S^3 \times S^8) \#12(S^4 \times S^7) \#10(S^5 \times S^6).
\end{equation}

\vspace*{.1in}

Let us prove that $QH(L, \mathbb{Z}_2[T,T^{-1}]) = 0$.
\\

\begin{lemma}\label{acunacohom}
We have: \\
$H^3(\mathcal{G}(\mathcal{G}(S^3 \times S^3 \times S^3)), \mathbb{Z}_2) = \mathbb{Z}_2^3$; \\
$H^{4}(\mathcal{G}(\mathcal{G}(S^3 \times S^3 \times S^3)), \mathbb{Z}_2) = \mathbb{Z}^6_2$;\\
$H^5(\mathcal{G}(\mathcal{G}(S^3 \times S^3 \times S^3)), \mathbb{Z}_2) = \mathbb{Z}_2^3$.
\end{lemma}

\begin{proof}
By $(S^3 \times S^3 \times S^3)_{-}$ we denote $S^3 \times S^3 \times S^3$ minus a small $9-$dimensional disk. From Lefschetz duality
\begin{equation*}
H^r((S^3 \times S^3 \times S^3)_{-} \times D^2, \mathcal{G}(S^3 \times S^3 \times S^3), \mathbb{Z}_2) \approx H_{11 - r}((S^3 \times S^3 \times S^3)_{-} \times D^2, \mathbb{Z}_2)
\end{equation*}
and from the exact sequence of the pair $((S^3 \times S^3 \times S^3)_{-} \times D^2, \;\; \mathcal{G}(S^3 \times S^3 \times S^3))$ we get
\begin{equation*}
\begin{gathered}
H_{11-r}((S^3 \times S^3 \times S^3)_{-} \times D^2, \mathbb{Z}_2) \rightarrow \\
 H^{r}((S^3 \times S^3 \times S^3)_{-} \times D^2, \mathbb{Z}_2) \rightarrow H^{r}(\mathcal{G}(S^3 \times S^3 \times S^3), \mathbb{Z}_2) \rightarrow \\
H_{10-r}((S^3 \times S^3 \times S^3)_{-} \times D^2, \mathbb{Z}_2) \rightarrow H^{r+1}((S^3 \times S^3 \times S^3)_{-} \times D^2, \mathbb{Z}_2) \rightarrow \\
\end{gathered}
\end{equation*}
and
\begin{equation*}
\begin{gathered}
0 \rightarrow H^3((S^3 \times S^3 \times S^3)_{-} \times D^2, \mathbb{Z}_2) \rightarrow H^3(\mathcal{G}(S^3 \times S^3 \times S^3), \mathbb{Z}_2) \rightarrow 0 \\
0 \rightarrow H^6((S^3 \times S^3 \times S^3)_{-} \times D^2, \mathbb{Z}_2) \rightarrow H^6(\mathcal{G}(S^3 \times S^3 \times S^3), \mathbb{Z}_2) \rightarrow 0 \\
0 \rightarrow H^4(\mathcal{G}(S^3 \times S^3 \times S^3), \mathbb{Z}_2) \rightarrow H_6((S^3 \times S^3 \times S^3)_{-} \times D^2, \mathbb{Z}_2) \rightarrow 0 \\
0 \rightarrow H^7(\mathcal{G}(S^3 \times S^3 \times S^3), \mathbb{Z}_2) \rightarrow H_3((S^3 \times S^3 \times S^3)_{-} \times D^2, \mathbb{Z}_2) \rightarrow 0
\end{gathered}
\end{equation*}
So, we obtain
\begin{equation*}
\begin{gathered}
H_3(\mathcal{G}(S^3 \times S^3 \times S^3), \mathbb{Z}_2) = \mathbb{Z}^3_2, \quad H_4(\mathcal{G}(S^3 \times S^3 \times S^3), \mathbb{Z}_2) = \mathbb{Z}_2^3, \\
H_6(\mathcal{G}(S^3 \times S^3 \times S^3), \mathbb{Z}_2) = \mathbb{Z}^3_2, \quad H_7(\mathcal{G}(S^3 \times S^3 \times S^3), \mathbb{Z}_2) = \mathbb{Z}_2^3, \\
H_{10}(\mathcal{G}(S^3 \times S^3 \times S^3), \mathbb{Z}_2) = \mathbb{Z}_2, \\
H_k(\mathcal{G}(S^3 \times S^3 \times S^3), \mathbb{Z}_2) = 0 \quad if \;\; k\neq 0,3,4,6,7,10
\end{gathered}
\end{equation*}
Repeating this procedure one more time for the pair $(\mathcal{G}(S^3 \times S^3 \times S^3)_{-} \times D^2, \;\; \mathcal{G}(\mathcal{G}(S^3 \times S^3 \times S^3)))$ we prove the lemma.
\end{proof}

Let us study the face ring of $P$ and use Theorem $\ref{torcohom}$. We denote by $v_1,...,v_6$ the variables of $[R(P), d]$ coming from the cube $I^3$ and by $w_1, w_2$ the variables coming from the new facets. Denote the corresponding exterior variables of $[R(P), d]$ by $y_1, \ldots, y_6, t_1, t_2$ such that $dy_i=v_i$, $dt_j = w_j$ (see Figure $\ref{fig:truncube3}$).

\begin{figure}[h]
\centering
\includegraphics[width=0.6\linewidth]{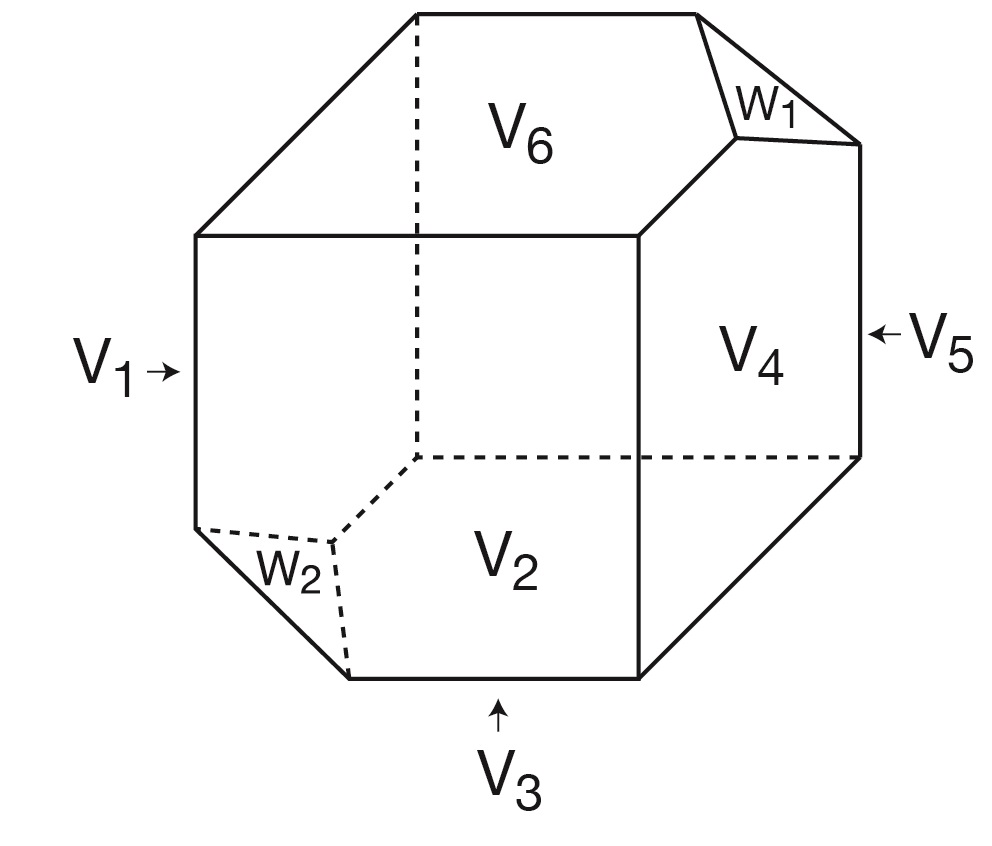}
\caption{truncted cube}
  \label{fig:truncube3}
\end{figure}

We see that
\begin{equation*}
\begin{gathered}
\widetilde{a}_1 = y_1v_{4}, \;\; \widetilde{a}_2 = y_2v_{5}, \;\; \widetilde{a}_3 = y_3v_{6}, \;\; \widetilde{s} = t_2w_1
\end{gathered}
\end{equation*}
represent nontrivial elements in $H^{3}(\widetilde{\mathcal{Z}}_P, \mathbb{Z})$. Since $v_4 \cap v_5 \cap v_6 = \emptyset$, we have
\begin{equation*}
\begin{gathered}
\widetilde{a}_1\widetilde{a}_2\widetilde{a}_3 = y_1y_2y_3v_4v_5v_6 = 0, \\
\widetilde{a}_1\widetilde{a}_2, \;\widetilde{a}_1\widetilde{a}_2, \; \widetilde{a}_2\widetilde{a}_3 \neq 0.
\end{gathered}
\end{equation*}
Therefore, $\widetilde{a}_1,\widetilde{a}_2,\widetilde{a}_3$ represent classes of $H^{*}(\mathcal{G}(\mathcal{G}(S^3\times S^3\times S^3)), \mathbb{Z})$  (see formula $(\ref{cubecomplexangle}))$. Then, elements
\begin{equation*}
\begin{gathered}
\widetilde{b}_1 = y_1t_2v_4, \; \widetilde{b}_2 = y_2t_2v_5, \; \widetilde{b}_3 = y_3t_2v_6, \\
\widetilde{c}_1 = y_4t_1v_1, \; \widetilde{c}_2 = y_5t_1v_2, \; \widetilde{c}_3 = y_6t_1v_3
\end{gathered}
\end{equation*}
belong to $H^{4}(\widetilde{\mathcal{Z}}_P, \mathbb{Z})$. Moreover,
\begin{equation*}
\widetilde{a}_i\widetilde{a}_j\widetilde{b}_m, \; \widetilde{a}_i\widetilde{a}_j\widetilde{c}_m \neq 0 \;\;\; if \; and \; only \; if \;\; i \neq j, i\neq m, j\neq m,
\end{equation*}
Hence, classes $\widetilde{a}_i, \widetilde{b}_j$ belong to cohomology ring  $H^{*}(\mathcal{G}(\mathcal{G}(S^3\times S^3\times S^3)), \mathbb{Z})$.

Since $\widetilde{s} = t_2w_1$ and $t_1w_2$ represent the same classes (because $d(t_1t_2)= t_2w_1 - t_2w_1$), we see that
\begin{equation*}
\widetilde{s}\widetilde{b}_i = \widetilde{s}\widetilde{c}_i = 0 \;\;\; \forall i. \\
\end{equation*}
Let us consider
\begin{equation*}
\widetilde{g}_1 = y_1y_2y_3w_1, \;\; \widetilde{g}_2 =  y_4y_5y_6w_2.
\end{equation*}
Since $v_1 \cap w_1 = v_2 \cap w_1 = v_3 \cap w_1 = v_4 \cap w_2 = v_5 \cap w_2 = v_6 \cap w_2 = \emptyset$, we have
\begin{equation*}
\widetilde{a}_i\widetilde{g}_1 = \widetilde{b}_i\widetilde{g}_1 = \widetilde{b}_i\widetilde{g}_2 = \widetilde{b}_i\widetilde{g}_2 = 0 \;\;\; \forall i.
\end{equation*}

Let us study the face ring of $Q = P^{(k, \ldots, 2k, 2k)}$. Theorem $\ref{cohomisotopic}$ says that $H^{*}(\widetilde{\mathcal{Z}}_P, \mathbb{Z})$ and $H^{*}(\widetilde{\mathcal{Z}}_Q, \mathbb{Z})$ are isomorphic as ungraded rings. We denote the variables of $[R(Q), d]$ by $v_{1,1},\ldots, v_{1,k},\ldots, v_{6,1}, \ldots, v_{6,k}$, $v_{7,1}, \ldots, v_{7,2k}, v_{8,1}, \ldots, v_{8,2k}$. From Lemma $\ref{interfacetswedge}$ and from the compuations above we get that
\begin{equation*}
\begin{gathered}
a_1 = y_{1,1}v_{1,2}\ldots v_{4,1}\ldots v_{4,k}, \;\; a_2 = y_{2,1}v_{2,2}\ldots v_{5,1}\ldots v_{5,k}, \\
a_3 = y_{3,1}v_{3,2}\ldots v_{6,1}\ldots v_{6,k}, \;\; b_1 = y_{1,1}t_{2,1}v_{1,2}\ldots v_{1,k}w_{2,2}\ldots, w_{2,2k}v_{4,1}\ldots v_{4,k}, \\
b_2 = y_{2,1}t_{2,1}v_{2,2}\ldots v_{2,k}w_{2,2}\ldots, w_{2,2k}v_{5,1}\ldots v_{5,k}, \\
b_3 = y_{3,1}t_{2,1}v_{3,2}\ldots v_{3,k}w_{2,2}\ldots, w_{2,2k}v_{6,1}\ldots v_{6,k}, \\
c_1 = y_{4,1}t_{1,1}v_{4,2}\ldots v_{4,k}w_{1,2}\ldots, w_{1,2k}v_{1,1}\ldots v_{1,k}, \\
c_2 = y_{5,1}t_{1,1}v_{5,2}\ldots v_{5,k}w_{1,2}\ldots, w_{1,2k}v_{2,1}\ldots v_{2,k}, \\
c_3 = y_{6,1}t_{1,1}v_{6,2}\ldots v_{6,k}w_{1,2}\ldots, w_{1,2k}v_{3,1}\ldots v_{3,k}, \\
g_1 = y_{1,1}y_{2,1}y_{3,1}v_{1,2}\ldots v_{1,k}v_{2,2}\ldots, v_{2,k}v_{3,2}\ldots, v_{3,k}w_{1,1}, \ldots, w_{1,2k}, \\
g_2 = y_{4,1}y_{5,1}y_{6,1}v_{4,2}\ldots v_{4,k}v_{5,2}\ldots, v_{5,k}v_{6,2}\ldots, v_{6,k}w_{2,1}, \ldots, w_{2,2k}, \\
s = t_{2,1}w_{2,2}\ldots w_{2,2k}w_{1,1}\ldots w_{1,2k}.
\end{gathered}
\end{equation*}
represent nontrivial elements. Moreover,
\begin{equation}\label{widenarrowrelations}
\begin{gathered}
a_1, a_2, a_3 \; \in H^{4k-1}(\widetilde{\mathcal{Z}}_Q, \mathbb{Z}), \\
b_1,b_2, b_3, c_1,c_2, c_3 \; \in H^{8k-2}(\widetilde{\mathcal{Z}}_Q, \mathbb{Z}), \\
g_1, g_2 \; \in H^{10k-3}(\widetilde{\mathcal{Z}}_Q, \mathbb{Z}), \\
a_1a_2, \; a_1a_3, \; a_2a_3 \neq  0, \;\;\; a_1a_2a_3 = 0, \\
a_ia_jb_m, \; a_ia_jc_m \neq 0 \;\;\; if \; and \; only \; if \;\; i \neq j, i\neq m, j\neq m, \\
a_ig_1 = b_ig_1 = b_ig_2 = b_ig_2 = 0, \; sa_i = sb_i = 0 \;\;\; \forall i.
\end{gathered}
\end{equation}

\begin{lemma}\label{vanishcohomth} Assume that $r \leqslant 12k-3$, and $r \neq 0, 4k-1, 6k-1, 8k-2, 8k-1, 10k-3, 12k-3$. If $k \geqslant 2$, then we have: \\
$H^{r}(\widetilde{\mathcal{Z}}_Q, \mathbb{Z}) = 0$;\\
$H^{4k-1}(\widetilde{\mathcal{Z}}_Q, \mathbb{Z})$ is generated by $a_1, a_2, a_3$; \\
$H^{8k-1}(\widetilde{\mathcal{Z}}_Q, \mathbb{Z})$ is generated by $s$; \\
$H^{10k-3}(\widetilde{\mathcal{Z}}_Q, \mathbb{Z})$ is generated by $g_1, g_2$.
\end{lemma}

\begin{proof}
We need to consider all possible combinations of the variables of $[R(Q), d]$. For example, $deg(t_2w_{2,2}\ldots w_{2,2k} w_{1,1}\ldots w_{1,2k} v_{5,1}\ldots v_{5,k}) = 10k-1$ but
\begin{equation*}
\begin{gathered}
d(y_{5,1}t_{2,1}w_{2,2}\ldots w_{2,2k} w_{1,1}\ldots w_{1,2k} \ldots v_{5,k}) = \\
t_2w_{2,2}\ldots w_{2,2k} w_{1,1}\ldots w_{1,2k} v_{5,1}\ldots v_{5,k}
\end{gathered}
\end{equation*}
because $w_{1,1}\cap\ldots w_{1,2k} \cap w_{2,1} \cap \ldots w_{2,2k} = \emptyset$. This means that the considered element is exact and represent trivial element in the cohomology group. Arguing in the same way we can prove the lemma.
\end{proof}

\vspace*{.1in}

Since $\widetilde{\mathcal{Z}}_Q$ does not have torsion we get that
\begin{equation*}
H^{*}(\widetilde{\mathcal{Z}}_Q, \mathbb{Z}_2) = H^{*}(\widetilde{\mathcal{Z}}_Q, \mathbb{Z}) \otimes \mathbb{Z}_2.
\end{equation*}
Let us denote elements in $H^{*}(\widetilde{\mathcal{Z}}_Q, \mathbb{Z}_2)$ corresponding to $a_i, b_i, c_i,g_i, s$ by the same symbols. Denote
\begin{equation*}
 \mathcal{Z}_Q = \widetilde{\mathcal{Z}}_Q/\mathbb{Z}_2.
\end{equation*}

We have the Cartan-Leray spectral sequence
\begin{equation*}
E_2^{p,q} = H^{p}(\mathbb{R}P^{\infty}, H^{q}(\widetilde{\mathcal{Z}}_Q, \mathbb{Z}_2)) \Rightarrow H^{*}(\mathcal{Z}_Q, \mathbb{Z}_2),
\end{equation*}
Denote the differential of $E_{k}$ by $\partial_k$. Recall that
\begin{equation*}
 m(Q) = min\{\ell \in \mathbb{N} \; | \; v_{i_1}\cap \ldots \cap v_{i_\ell} = \emptyset  \}.
\end{equation*}
We see that $m(P) = 2$. It follows from Theorem $\ref{interfacetswedge}$ that
\begin{equation*}
m(Q) = m(P^{(k,\ldots,k, 2k,2k)}) = 2k.
\end{equation*}

Let $\beta$ be the generator of $H^{*}(\mathbb{R}P^{\infty}, \mathbb{Z}_2) = E_{2}^{*, 0}$. Let us show that $\partial_{4k}(E_{4k}^{0,4k-1}) = 0$, First, note that $E_{2}^{0,4k-1} = E_{4k}^{0,4k-1}$. Then, from formulas $(\ref{widenarrowrelations})$ we have (we work over $\mathbb{Z}_2$)
\begin{equation}\label{vanishprovethe}
\begin{gathered}
a_1a_2a_3 = 0 \Rightarrow \\
 0 = \partial_{4k}(a_1a_2a_3) = \partial_{4k}(a_1)a_2a_3 + \partial_{4k}(a_2)a_1a_3 + \partial_{4k}(a_1)a_2a_3,  \\
\partial_{4k}(a_1), \partial_{4k}(a_1), \partial_{4k}(a_1) \in \{0, \beta^{4k} \}.
\end{gathered}
\end{equation}
If $\beta^{4k}a_1a_2, \beta^{4k}a_1a_3, \beta^{4k}a_2a_3$ are nontrivial in $E_{4k}^{4k, 8k-2}$, then $\partial_{4k}(a_i) = 0$ for all $i$.

The second differential $\partial_2: E_{2}^{0, 8k-1} \rightarrow E_{2}^{2, 8k-2}$ may be nontrivial. By Lemma $\ref{vanishcohomth}$, $H^{8k-1}(\widetilde{\mathcal{Z}}_Q, \mathbb{Z}_2)$ is generated by $s$.  Assume that
\begin{equation*}
\partial_{2}(s) = \lambda_1\beta^2a_1a_2 + \lambda_2\beta^2a_1a_3 + \lambda_3\beta^2a_2a_3 + \ldots, \quad \lambda_1,\lambda_2,\lambda_3 \in \mathbb{Z}_2.
\end{equation*}
Since $b_1s = 0$ and $b_1$ represents the cohomology class of $\mathcal{G}(\mathcal{G}(S^3 \times S^3 \times S^3))$, we have
\begin{equation*}
\begin{gathered}
\partial_{2}(b_1) \in E_{2}^{2,8k-3} = 0  \Rightarrow \\
0 = \partial_{2}(b_1s) = b_1\partial_2(s) = \lambda_3\beta^{2}b_1a_2a_3 \Rightarrow \lambda_3 = 0.
\end{gathered}
\end{equation*}
In the same way, using $b_2, b_3$ we can prove that $\lambda_1 = \lambda_2 = 0$. So, $\partial_2$ does not kill $\beta^{4k}a_ia_j$. Since $\partial_2(a_i) = 0$, we get that $\partial_2(a_ia_j) = 0$ for any $i, j$. This means that elements $a_ia_j$ are not trivial in $E_3$.

It follows from Lemma $\ref{vanishcohomth}$ that $E_{3}^{4k, 8k-2} = E_{2k}^{4k, 8k-2}$. Since $H^{10k-3}(\widetilde{\mathcal{Z}}_Q, \mathbb{Z}_2) = E_{2k}^{0, 10k-3}$ is generated by $g_1, g_2$ and $b_ig_j = 0$ for all $i, j$, we have (recall that $b_i \in H^{*}(\mathcal{G}(\mathcal{G}(S^3 \times S^3 \times S^3)), \mathbb{Z}_2)$)
\begin{equation*}
\begin{gathered}
\partial_{2k}(g_1) = \lambda_1\beta^{2k}a_1a_2 + \lambda_2\beta^{2k}a_1a_3 + \lambda_3\beta^{2k}a_2a_3 + \ldots, \quad \lambda_1,\lambda_2,\lambda_3 \in \mathbb{Z}_2,\\
0 = \partial_{2}(b_1g_1) = b_1\partial_2(g_1) = \lambda_3\beta^{2k}b_1a_2a_3 \Rightarrow \lambda_3 = 0.
\end{gathered}
\end{equation*}
Analogously, $\lambda_1=\lambda_2 = 0$. Since $\partial_{2k}(a_i) \in E_{2k}^{2k, 2k} = 0$, we get that $\partial_{2k}(a_ia_j) = 0$. This yields that $\partial_{2k}$ does not kill $\beta^{4k}a_ia_j$.

Finally, we see that  $a_1a_2, a_1a_3, a_2a_3$ are nontrivial in $E_{4k}^{0, 8k-1}$. So, we get from formulas $(\ref{vanishprovethe})$ that
\begin{equation*}
\partial_{4k}(E_{4k}^{0, 4k-1}) = 0.
\end{equation*}
It follows from Theorem $\ref{vanishquantum}$ that $QH(L, \mathbb{Z}_2[T, T^{-1}]) = 0$.
\\

Let $G$ be a ring in which $2$ is invertible. Lemma $\ref{vanishquantum}$ says that $QH(L, G[T, T^{-1}]) = 0$.

\subsection{Proof of Theorem $\ref{fundgroup}$}

Let $P$ be a $6-$gon defined by inequalities
\begin{equation*}
P = \left\{
 \begin{array}{l}
x_1+1 \geqslant 0 \;\;\; x_2 + 1 \geqslant 0 \\
-x_1 + 1\geqslant 0 \;\;\; -x_2 + 1 \geqslant 0 \\
x_1 + x_2  + 1 \geqslant 0 \;\;\; -x_1-x_2 + 1 \geqslant 0
 \end{array}
\right.
\end{equation*}
The real moment-angle manifold associated to $P$ is given by
\begin{equation}\label{systemfundgroup}
\widetilde{\mathcal{R}}_P = \left\{
 \begin{array}{l}
u_1^2 + u_3^2 = 2 \\
u_2^2 + u_4^2 = 2 \\
u_1^2 + u_2^2 - u_5^2 = 1 \\
u_1^2 + u_2^2 + u_6^2 = 3
 \end{array}
\right.
\end{equation}

The manifold $\widetilde{\mathcal{R}}_P = S_{17}$, where $S_{17}$ is an orientable surface of genus $17$ (see \cite{torictop} Proposition $4.1.8$). The sum of normal vectors to the facets is equal to zero. Moreover, $P$ is Fano and Delzant. From Section $\ref{maincostruction}$ and Lemma $\ref{sumzero}$ we get that there is monotone embedded Lagrangian $L \subset \mathbb{C}P^6$ associated to $P$. By our construction $L$ fibers over $T^3$
\begin{equation*}
\begin{gathered}
L \xrightarrow{S_{17}/\mathbb{Z}_2} T^3, \\
(u_1,\ldots, u_6) \sim (-u_1, \ldots, -u_6), \quad (u_1,\ldots,u_6) \in S_{17}.
\end{gathered}
\end{equation*}
From the exact sequence of the fibration we have
\begin{equation*}
1 \rightarrow \pi_1(S_{17}/\mathbb{Z}_2) \rightarrow \pi_1(L) \rightarrow \pi_1(T^3) \rightarrow 1.
\end{equation*}

\subsection{Proof of Theorem $\ref{threespheretheo}$}

We argue as in the previous theorems. Let us consider a submanifold $\widetilde{\mathcal{R}}$ of $\mathbb{R}^n$ defined by a system
\begin{equation}\label{threesphereeq}
\begin{gathered}
\widetilde{\mathcal{R}} = \left\{
 \begin{array}{l}
\sum\limits_{r=1}^{k}u_r^2 + \sum\limits_{r=k+1}^{p_1}u_r^2  = p_1 \\
\sum\limits_{r=1}^{k}u_r^2 + \sum\limits_{r=p_1+1}^{p_2}u_r^2 = p_2 - p_1 + k\\
- \sum\limits_{r=1}^{k}u_r^2 + \sum\limits_{r=p_2+1}^{n}u_r^2 = n-p_2 - k \\
 \end{array}
\right.
\\
k < p_1-1, \quad  p_2 - p_1 + k > p_1, \quad n-p_2 - k > p_1, \\
p_1, \; p_2-p_1, \; n-p_2 > 2
\end{gathered}
\end{equation}

Since $ p_2 - p_1 + k> p_1$ and $n-p_2 - k> p_1$, we have a diffeomorphism
\begin{equation*}
\begin{gathered}
f: \widetilde{\mathcal{R}} \rightarrow S^{p_1-1} \times S^{p_2-p_1 - 1} \times S^{n-p_2-1}, \\
f(u_1,\ldots, u_n) = \bigg(u_1, \ldots, u_{p_1}, \frac{u_{p_1+1}}{\sqrt{p_2 - p_1 + k - \sum\limits_{r=1}^{k}u_r^2} }, \ldots, \frac{u_{p_2}}{\sqrt{p_2 - p_1 + k - \sum\limits_{r=1}^{k}u_r^2 }}, \\
 \frac{u_{p_2+1}}{\sqrt{n-p_2 - k + \sum\limits_{r=1}^{k}u_r^2 }}, \ldots, \frac{u_{n}}{\sqrt{n-p_2 - k + \sum\limits_{r=1}^{k}u_r^2} }\bigg).
\end{gathered}
\end{equation*}
Denote by $\widetilde{\Lambda}$ the lattice generated by columns of $(\ref{threesphereeq})$. Direct computation shows that $\widetilde{\Lambda}_u = \widetilde{\Lambda}$ and conditions of Lemma $\ref{equivmon}$ are satisfied (for $C=1$). This means that the polytope $P$ associated to system $(\ref{threesphereeq})$ is Delzant and Fano. Theorem $\ref{monotonecondition}$ says that there exists an embedded monotone Lagrangian $L \subset \mathbb{C}P^{n-1}$ associated to the polytope $P$.

If $k, p_1, p_2, n$ are even, then from Lemma $\ref{trivialfibration}$ we get that
\begin{equation*}
L = (S^{p_1-1} \times S^{p_2-p_1 - 1} \times S^{n-p_2-1})/\mathbb{Z}_2 \times T^2.
\end{equation*}
It follows from Lemma $\ref{simplemimasnumber}$ that the minimal Maslov number
\begin{equation*}
N_L = gcd(p_1, p_2-p_1+k, n-p_2-k).
\end{equation*}

\section{From symplectic geometry to algebraic topology}\label{sympalg}

In this section we explain how to use symplectic geometry and Lagrangian quantum cohomology to study the singular cohomology of manifolds. We plan to prove a general theorem in a further paper. In this Section we consider an example and show how the new method works.
\\

Let $\mathcal{R} \subset \mathbb{R}P^{n-1}$ be a smooth manifold defined by a system of quadrics
\begin{equation*}
\begin{gathered}
\mathcal{R} = \left\{
 \begin{array}{l}
\gamma_{1,r}u_1^2 + ... + \gamma_{n,r}u_{n}^2 = 0, \quad r=1,...,n-m-1.
 \end{array}
\right.
\\
[u_1:\ldots:u_n] \in \mathbb{R}P^{n-1}.
\end{gathered}
\end{equation*}
We would like to know $H^{*}(\mathcal{R}, \mathbb{Z}_2)$. In general, the topology of $\mathcal{R}$ can be highly complicted (see \cite{Krasnov} for examples). The Agrachev-Lerario spectral sequence( see \cite{Agrachev}) can be used to study the Betti numbers of $\mathcal{R}$. Unfortunately, the Agrachev-Lerario spectral sequence often requires extremely heavy computations.

The manifold $\mathcal{R}$ can be considered as a real toric space. There are explicit formulas for $H^{*}(\mathcal{R}, \mathbb{Q})$ (see \cite{real2}), but formula for $H^{*}(\mathcal{R}, \mathbb{Z}_2)$ is not known.

Let us study a manifold $\widetilde{\mathcal{R}} \subset \mathbb{R}^n$ given by
\begin{equation*}
\begin{gathered}
\widetilde{\mathcal{R}} = \left\{
 \begin{array}{l}
 u_1^2 + \ldots + u_n^2 = 1 \\
\gamma_{1,r}u_1^2 + ... + \gamma_{n,r}u_{n}^2 = 0, \quad r=1,...,n-m-1.
 \end{array}
\right.
\\
(u_1, \ldots, u_n) \in \mathbb{R}^n.
\end{gathered}
\end{equation*}
Easy to see that
\begin{equation*}
\mathcal{R} = \widetilde{\mathcal{R}}/\mathbb{Z}_2, \quad (u_1,\ldots, u_n) \sim (-u_1, \ldots, -u_n).
\end{equation*}
Since the Eilenberg-MacLane space $K(\mathbb{Z}_2, 1) = \mathbb{R}P^{\infty}$, we have the Cartan-Leray spectral sequence (see \cite[p. 206]{Weibel})
\begin{equation*}
E_2^{p,q} = H^{p}(\mathbb{R}P^{\infty}, H^{q}(\widetilde{\mathcal{R}}, \mathbb{Z}_2)) \Rightarrow H^{*}(\mathcal{R}, \mathbb{Z}_2).
\end{equation*}
In general, it is not  easy to find $E_{\infty}$. Moreover, $E_{\infty}$ does not give enough information to find the singular cohomology ring of $\mathcal{R}$. So, we need some luck or extra information to get $H^{*}(\mathcal{R}, \mathbb{Z}_2)$ from $E_{\infty}$.

From Section $\ref{mainconstructionpol}$ we know that there is a polytope $P$ associated to $\widetilde{\mathcal{R}}$. From Section $\ref{maincostruction}$ we know that there is a Lagrangian $L \subset \mathbb{C}P^{n-1}$ associated to the polytope $P$. If $P$ is Delzant and Fano, then the Lagrangian is embedded and monotone. Therefore, $QH^{*}(L, \mathbb{Z}_2[T, T^{-1}])$ is defined. We have the spectral sequence of Oh, where the first page is defined by $H^{*}(L, \mathbb{Z}_2)$. Since the spectral sequence of Oh is multiplicative and converges to $QH^{*}(L, \mathbb{Z}_2[T, T^{-1}])$, we get some restrictions on $\mathcal{R}$ (because $L$ is constructed using $\mathcal{R}$). It follows from Lemma $\ref{helplemma}$ that for a rich set of polytopes we have
\begin{equation*}
L = \mathcal{R} \times T^{n-m-1}.
\end{equation*}
If $N_L > 2$, then we can ``ignore" the torus in the spectral sequence of Oh and work with $H^{*}(\mathcal{R}, \mathbb{Z}_2)$. Theorem $\ref{vanishquantum}$ says that $QH^{*}(L, \mathbb{Z}_2[T, T^{-1}]) = 0$ for some set of polytopes. Hence, the spectral sequence of Oh converges to zero and this gives us some extra information for studying $E_{\infty}$ and $H^{*}(\mathcal{R}, \mathbb{Z}_2)$.
\\

\textbf{Remark.} We consider Delzant and Fano polytopes, but as we can see from proofs of the previous theorems many polytopes can be made Fano by applying simplicial multiwedge operation. So, we can make the polytope Fano, study cohomology ring of its real toric space, and use Theorem $\ref{cohomisotopic}$ to study the cohomology ring of the original real toric space.
\\

First, let us consider a simple artificial example. Assume that $L = \mathcal{R} \times S^1$, where $dim(M) = 6$, $N_L = 4$, and $QH(L, \mathbb{Z}_2[T, T^{-1}]) = 0$. Assume that $E_{\infty}$ page of the Cartan-Leray spectral sequence has the form shown in the figure $\ref{fig:s007}$  Then, $y^2$ equals either $yx^3$, or $0$. We see that $H^{*}(L, \mathbb{Z}_2)$ is generated by $x, y, c$, where $c \in H^{1}(S^1, \mathbb{Z}_2)$. The spectral sequence of Oh is defined in the end of Section $\ref{quantumdefinitions}$. Let $\delta_1$ be the differential on the first page of Oh's spectral sequence. Since $N_L = 4$, we get that the degree of $\delta_1$ equals $3$ and $\delta_1(x)=\delta_1(c) = 0$. If $y^2 = x^3y$, then
\begin{equation*}
0 = 2\delta_1(y) = \delta_1(y^2) = \delta_1(x^3y) = x^3\delta_1(y) \;\; \Rightarrow \;\; \delta_1(y) = 0.
\end{equation*}
This yields that $QH(L, \mathbb{Z}_2[T, T^{-1}])$ is isomorphic to $H^{*}(L, \mathbb{Z}_2)$, but we know that $QH(L, \mathbb{Z}_2[T, T^{-1}]) = 0$. Hence, $y^2 = 0$.

\newpage

\begin{figure}[h]
\includegraphics[width=0.6\linewidth]{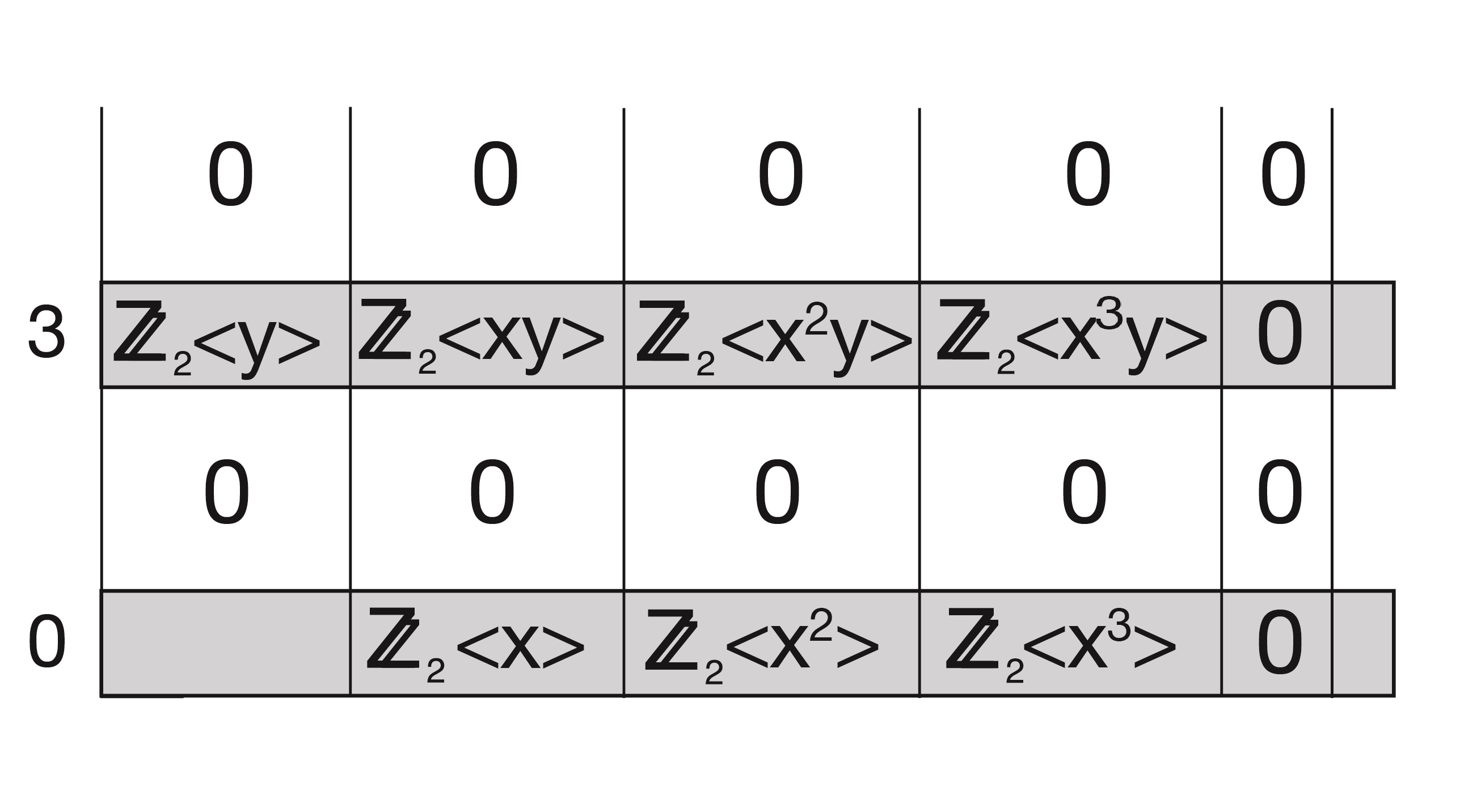}
\caption{ $E_{\infty}$ page }
  \label{fig:s007}
\end{figure}

Let us consider a real example. In Section $\ref{wide-narrow}$ we constructed the Lagrangian (see Lemma \ref{consumquantum})
\begin{equation*}
\begin{gathered}
L = \mathcal{R}_P \times T^{3}, \quad \mathcal{R}_P = \widetilde{\mathcal{R}}_P/\mathbb{Z}_2, \\
H^{*}(\widetilde{\mathcal{R}}_P, \mathbb{Z}_2) \cong H^{*}(\#9(S^{4k-1} \times S^{8k-3}) \#8(S^{6k-2} \times S^{6k-2}), \; \mathbb{Z}_2), \\
QH^{*}(L, \mathbb{Z}_2[T, T^{-1}]) = 0.
\end{gathered}
\end{equation*}
In the end of Section $\ref{wide-narrow}$ we show that $E_{\infty}$ page of the Cartan-Leray spectral sequence has one of the following forms (see the figures below).

\begin{figure}[H]
\includegraphics[width=0.7\linewidth]{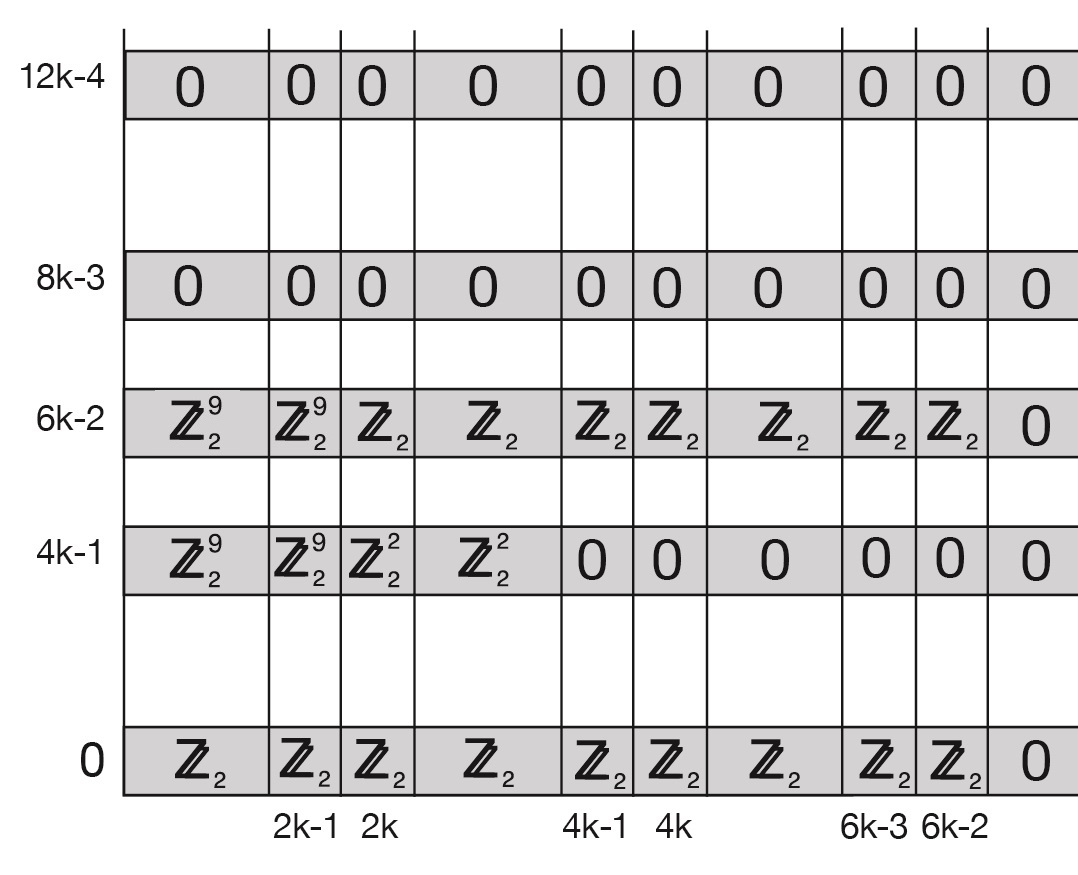}
\caption{ $E_{\infty}$ page }
  \label{fig:s4}
\end{figure}

\begin{figure}[H]
\includegraphics[width=0.7\linewidth]{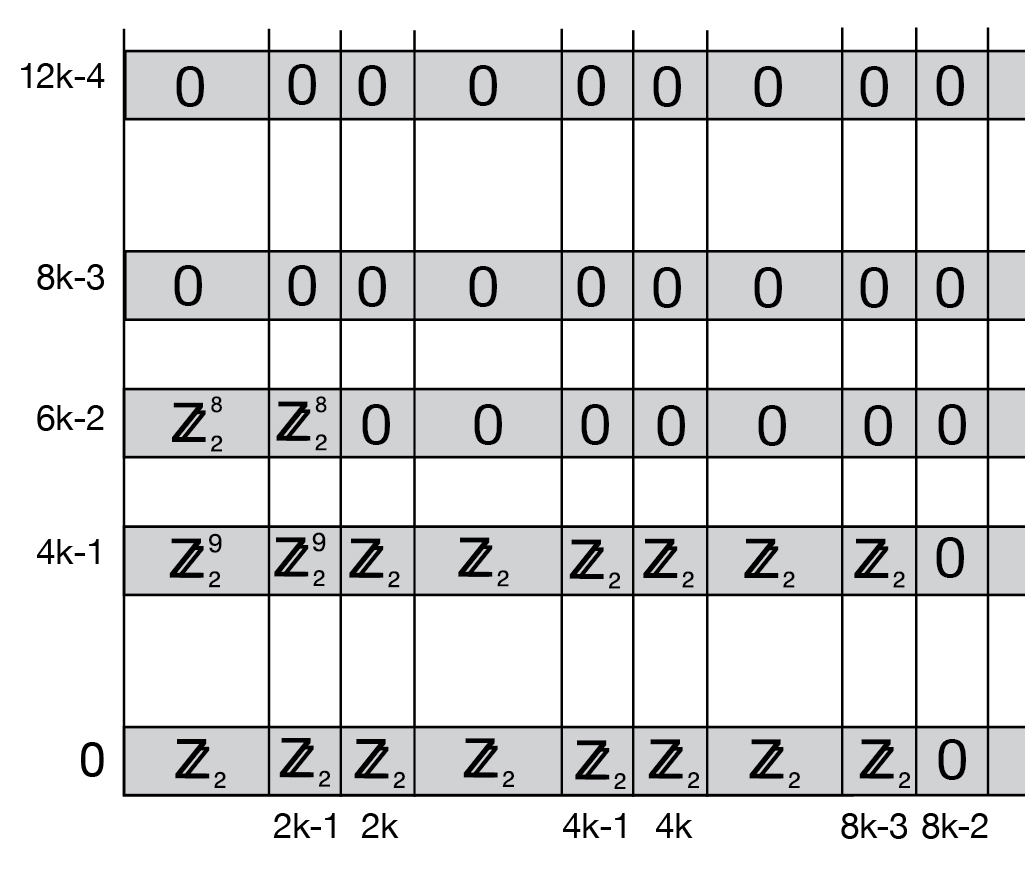}
\caption{ $E_{\infty}$ page}
  \label{fig:s2}
  \end{figure}

  \begin{figure}[H]
  \includegraphics[width=0.7\linewidth]{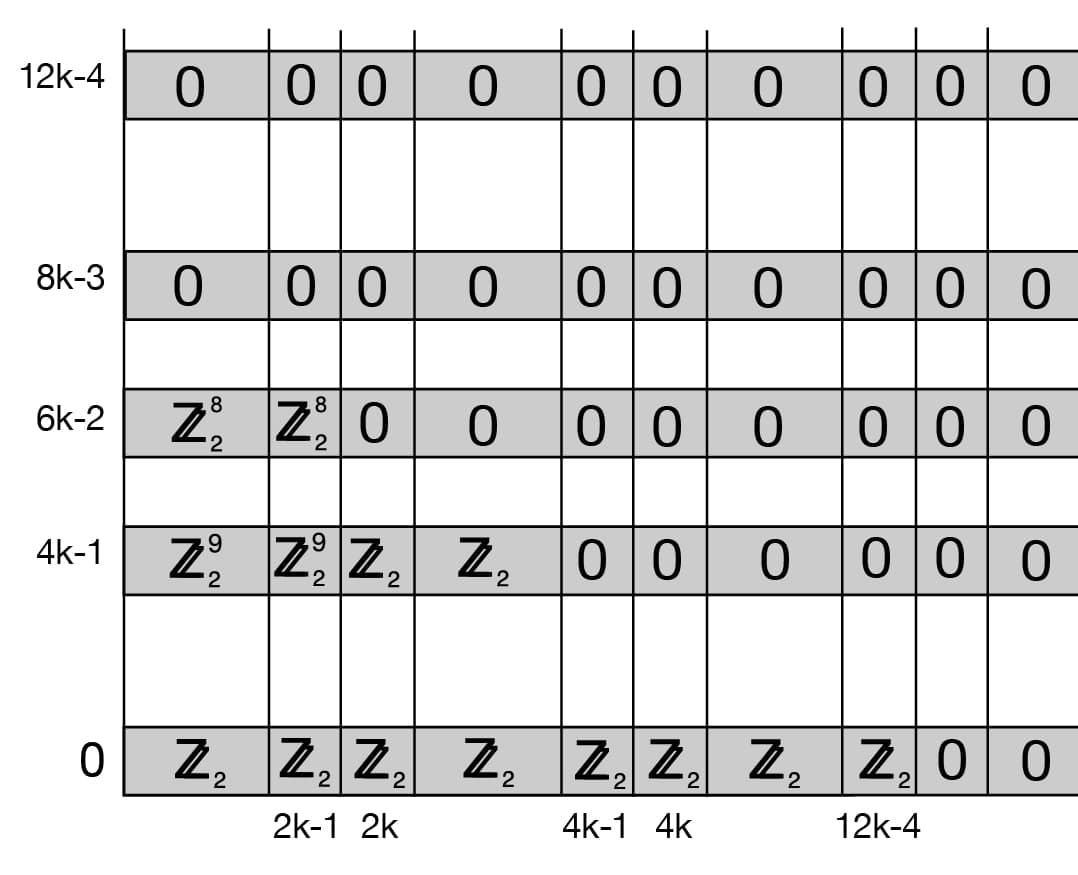}
\caption{ $E_{\infty}$ page}
  \label{fig:s5}
\end{figure}

It is not clear how to determine $E_{\infty}$ page using standard methods of algebraic topology. Let us show that the spectral sequence of Oh helps to find $E_{\infty}$ page and provides some information about $H^{*}(\mathcal{R}_P, \mathbb{Z}_2)$.

\begin{lemma}
Assume that $k \geqslant 3$. Then, Figure $\ref{fig:s4}$ is the $E_{\infty}$ page of the Cartan-Leray spectral sequence.
\end{lemma}

\textbf{Remark.} The goal of this Lemma is to demonstrate how the spectral sequence of Oh can be applied to algebraic topology. The rigorous proof of the lemma above is long. We give general ideas and skip all details.
\\
\\
\emph{Sketch proof.} Denote the nontrivial element of $E_{\infty}^{1,0}$ by $\alpha$. Let us recall that
\begin{equation*}
L=\mathcal{R}_P \times T^3.
\end{equation*}
Let $c_1, c_2, c_3$ be generators of $H^1(T^3, \mathbb{Z}_2)$.
\\
\\
Let us consider Figure $\ref{fig:s5}$. Denote generators of $E_{\infty}^{0,4k-1}$ by $x_1, \ldots, x_9$, generators of $E_{\infty}^{0, 6k-2}$ by $y_1,\ldots, y_8$.
\begin{equation*}
\begin{gathered}
E_{\infty}^{m, 6k-2} = \mathbb{Z}_2\langle \alpha^m y_1,\ldots, \alpha^m y_8 \rangle, \quad m \leqslant 2k-1, \\
E_{\infty}^{m, 4k-1} = \mathbb{Z}_2\langle \alpha^m x_1,\ldots, \alpha^m x_9 \rangle, \quad m \leqslant 2k-1, \\
E_{\infty}^{r, 4k-1} = \mathbb{Z}_2\langle \alpha^r x_1\rangle, \quad 2k \leqslant r \leqslant 4k-2, \\
E_{\infty}^{\ell, 0} = \mathbb{Z}_2\langle \alpha^{\ell}\rangle, \quad  0 \leqslant \ell \leqslant 12k-4.
\end{gathered}
\end{equation*}

Let us denote elements of $H^{*}(\mathcal{R}, \mathbb{Z}_2)$ by the same symbols.

Denote the differentials of the spectral sequence of Oh by $\delta_k$. Since the minimal Maslov number $N_L = 2k$, we get that the differential $\delta_1$ has degree $2k-1$.

We see that $\alpha^{6k-2}x_1$ equals either $\alpha^{10k-3}$, or $0$. Since $k>1$ we have $\delta_1(\alpha) = \delta_1(\alpha^{\ell}) = 0$. Therefore (here we should say more to make the expression below rigorous)
\begin{equation*}
0 = \delta_1(\alpha^{6k-2}x_1) = \alpha^{6k-2}\delta_1(x_1).
\end{equation*}
where $\delta_1(x_1) \in \mathbb{Z}_2\langle \alpha^{2k-2}, \alpha^{2k-3}c_i,  \alpha^{2k-4}c_ic_j, \alpha^{2k-5}c_1c_2c_3  \rangle$, $i,j=1,\ldots,3$. We see that
\begin{equation*}
\alpha^{6k-2}\delta_1(x_1) = 0 \Leftrightarrow  \delta_1(x_1) = 0.
\end{equation*}
We get that $\delta_1(x_1) = 0$. Analogously, let $y \in H^{6k-2}(L, \mathbb{Z}_2)$ be an element such that $\delta_1(y) = x_1$. Then, $\alpha^{4k-2}y$ equals either $\alpha^{10k-4}$, or $0$. We get
\begin{equation*}
0 = \delta_1(\alpha^{4k-2}y) = \alpha^{4k-2}\delta_1(y) = \alpha^{4k-2}x_1 \neq 0.
\end{equation*}
This contradiction proves that there is no $y$ such that $\delta_1(y) = x_1$. Hence, $x_1$ survives and is nontrivial element in the second page of Oh's spectral sequence.

Arguing as before, we can show that $x_1$ survives on all other pages, but we know that $QH^{*}(L, \mathbb{Z}_2[T, T^{-1}]) = 0$ and the spectral sequence of Oh converges to zero. Therefore, Figure $\ref{fig:s5}$ is not $E_{\infty}$ page of $\mathcal{R}$.

In the same way we can show that Figure $\ref{fig:s2}$ is not $E_{\infty}$ page of $\mathcal{R}$.

Department of Mathematics and Statistics, University of Montreal,\\
C.P. 6128 Succ. Centre-Ville Montreal, QC, H3C 3J7, Canada.\\
Department of Mathematics and Mechanics, Novosibirsk State University, Novosibirsk, 630090, Russia. \\
E-mail address: vardan8oganesyan@gmail.com

\end{document}